\documentclass[11pt,oneside]{amsart}
\usepackage{amsrefs}
\usepackage[T1]{fontenc}
\usepackage[utf8]{inputenc}
\usepackage[body={16cm, 24cm}]{geometry}
\usepackage{graphicx,mathrsfs,dsfont,amssymb,color,esint,longtable,eucal} %
\usepackage{enumitem}
\usepackage{epsfig,bbm,mathtools}
\usepackage{hyperref}
\usepackage{tikz-cd}
\usepackage{enumitem}

\newcommand{\R}{\mathbb{R}}
\renewcommand{\S}{\mathbb{S}}

\newcommand{\frB}{\mathfrak B}

\renewcommand{\P}{\mathbb{P}}
\newcommand{\N}{\mathbb{N}}
\newcommand{\T}{\mathbb{T}}
\newcommand{\HH}{\mathscr{H}}
\newcommand{\KK}{\mathscr{K}}
\newcommand{\PP}{\mathscr{P}}
\renewcommand{\SS}{\mathscr{S}}
\newcommand{\UU}{\mathscr{U}}
\newcommand{\BB}{\mathscr{B}}
\newcommand{\CC}{\mathscr{C}}
\newcommand{\AC}{\mathrm{AC}}
\newcommand{\FF}{\mathscr{F}}

\newcommand{\cC}{{\ensuremath{\mathcal C}}}

\newcommand{\cA}{{\ensuremath{\mathcal A}}}
\newcommand{\cB}{{\ensuremath{\mathcal B}}}

\newcommand{\cL}{{\ensuremath{\mathcal L}}}
\newcommand{\cK}{{\ensuremath{\mathcal K}}}
\newcommand{\cP}{{\ensuremath{\mathcal P}}}
\newcommand{\cN}{{\ensuremath{\mathcal N}}}
\newcommand{\Leb}{{\ensuremath{\mathrm{Leb}}}}
\newcommand{\sfd}[1]{\mathsf d_{#1}}

\newcommand{\eeta}{\boldsymbol\eta}

\newcommand{\de}{\partial}
\renewcommand{\div}{\mathrm{div}}

\newcommand{\m}{\mathrm{m}}
\newcommand{\BM}{\mathrm{M}}
\newcommand{\Bor}{\mathrm{B}}
\newcommand{\Parts}{\mathcal{S}}

\newcommand{\mm}{\mathfrak m}

\renewcommand{\d}{{\mathrm d}}
\newcommand{\ubar}{\underline}
\newcommand{\mres}{\mathbin{\vrule height 1.6ex depth 0pt width
0.13ex\vrule height 0.13ex depth 0pt width 1.3ex}}
\newcommand{\argmin}{\mathop{\rm argmin}\limits}
\newcommand{\supp}{\mathop{\rm supp}\nolimits}
\newcommand{\eps}{\varepsilon}
\newcommand{\pL}{\textup{\textbf{L}}} 
\renewcommand{\L}{\textup{\textbf{L}}} 
\newcommand{\RL}{\textup{\textbf{RL}}} 
\newcommand{\E}{\textup{\textbf{E}}} 
\newcommand{\FL}{\textup{\textbf{FL}}} 
\newcommand{\K}{\textup{\textbf{K}}} 

\newcommand{\rmC}{\mathrm C}
\newenvironment{system} 
{\left\lbrace\begin{array}{@{}l@{}}}
{\end{array}\right.}

\newtheorem{theorem}{\textbf{Theorem}}[section]
\newtheorem*{theorem*}{\textbf{Theorem}}
\newtheorem{proposition}[theorem]{\textbf{Proposition}}
\newtheorem{lemma}[theorem]{\textbf{Lemma}}

\theoremstyle{definition}
\newtheorem{definition}[theorem]{\textbf{Definition}}
\newtheorem{assumption}[theorem]{\textbf{Assumption}}

\theoremstyle{remark}

\newtheorem{remark}[theorem]{Remark}

\numberwithin{equation}{section}

\title[Lagrangian, Eulerian and Kantorovich control problems]{Lagrangian, Eulerian and Kantorovich formulations of multi-agent optimal control problems: Equivalence and Gamma-convergence}

\author[G. Cavagnari]{Giulia Cavagnari}
\address{\hspace{-0.5em}\begin{tabular}{ll}Giulia Cavagnari:& Politecnico di Milano,\\&Dipartimento di Matematica ``F. Brioschi''\\& 
Piazza Leonardo da Vinci, 32, I-20133 Milano, Italy.\end{tabular}}
\email{giulia.cavagnari@polimi.it}

\author[S. Lisini]{Stefano Lisini}
\address{\hspace{-0.5em}\begin{tabular}{ll}Stefano Lisini:&Universit\`a di Pavia,\\& Dipartimento di Matematica ``F. Casorati''\\ &Via Ferrata 5, I-27100 Pavia, Italy.\end{tabular}}
\email{stefano.lisini@unipv.it}

\author[C. Orrieri]{Carlo Orrieri}
\address{\hspace{-0.5em}\begin{tabular}{ll}Carlo Orrieri:&Universit\`a di Pavia,\\& Dipartimento di Matematica ``F. Casorati''\\ &Via Ferrata 5, I-27100 Pavia, Italy.\end{tabular}}
\email{carlo.orrieri@unipv.it}

\author[G. Savar\'e]{Giuseppe Savar\'e}
\address{\hspace{-0.5em}\begin{tabular}{ll}Giuseppe Savar\'e:&Universit\`a Bocconi,\\& Dipartimento di Scienze delle Decisioni\\ &Via Roentgen 1, I-20136 Milano, Italy.\end{tabular}}
\email{giuseppe.savare@unibocconi.it}

\date{\today}


\begin{document}

\begin{abstract}
This paper is devoted to the study of multi-agent deterministic optimal control problems.
We initially provide a thorough analysis of the Lagrangian, Eulerian and Kantorovich formulations of the problems, as well as of their relaxations.
Then we exhibit some equivalence results among the various representations and compare the respective value functions.
To do it, we combine techniques and ideas from optimal transportation, control theory, Young measures and  evolution equations in Banach spaces.  
We further exploit the connections among Lagrangian and Eulerian descriptions  to derive consistency results as the number of particles/agents tends to infinity.
To that purpose we prove an empirical version of the Superposition Principle and obtain suitable Gamma-convergence results for the controlled systems.
\end{abstract}

\maketitle

\tableofcontents

\section{Introduction}

{In recent years there has been an impressive increase in the analysis of models with interactions and associated optimal control problems.}
The motivations {and fields of interest} are various and range from statistical mechanics to biology, from crowd dynamics to the description of economical and financial phenomena, and many others.      
A lot of different mathematical models and techniques have been proposed in the literature and it seems impossible to be exhaustive in accounting here all the developments.        
We refer to \cite{carrillo2014derivation, fornasier2014mean, carmona2018probabilistic, fornasier2018mean}
and the references therein for  some of the recent results.

A large part of the literature concentrates the attention on the evolution of populations of similar individuals where the single agent feels the interaction with the others through an averaged term.
In this case, when the number of individuals is very large, an aggregation effect takes place and the (discrete) collection of agents is usually replaced by its spatial density.
This idea comes from the so called \emph{mean field} approach in statistical physics where it has been fruitfully used to develop a limit theory when the number of particles goes to infinity.

Within this framework, optimal control problems, both at the microscopic and macroscopic levels, are naturally considered, see e.g. \cite{fornasier2014mean, piccoli2015control, carrillo2020mean, jimenez2020optimal, cavagnari2020generalized, cristiani2014}.
A first motivation for the introduction of (centralized) controls is the incompleteness of the concept of self-organization.
In fact, for a population of interacting particles/agents, global coordination or pattern formation is not a priori guaranteed and the intervention of a central planner on the dynamics could promote these mechanisms: this leads to the definition of \emph{multi-agent} control problems.
A further motivation is the analysis of interacting rational agents with similar optimization goals.
In this case, the mean field approach consists in approximating a large number of agents with a single representative individual, whose aim is to solve a control problem constrained to a field equation (encoding the averaged behaviour of the population).
The mean field term influences both the dynamics and the cost functional. 
Whether the representative agent can or cannot influence the mean field term depends on the model under consideration.
The interested reader is referred to the books \cite{carmona2018probabilistic, bensoussan2013mean} for a detailed description of various aspects of mean field models.

\medskip

In the present paper we study different formulations of \emph{multi-agent optimal control problems}, that we denote respectively with \emph{Lagrangian}, \emph{Eulerian} and \emph{Kantorovich} as well as the corresponding \emph{limit theory}. 

\smallskip
    
As already mentioned above, \emph{multi-agent optimal control}, also known in literature as centralized optimal control of Vlasov dynamics, represents non-standard optimal control problems where each individual is influenced by the averaged behaviour of all the others and
the central planner aims at minimizing a cost functional which depends on the distribution of all the agents.
 
A large effort has been devoted in the last years to extend results of classical optimal control theory to the mean field setting, with a particular attention to the measure-formulation of the problems in Wasserstein spaces.
In this direction, let us mention \cite{bonnet2019pontryagin, pogodaev2016} for necessary conditions for optimality in the form of a Pontryagin maximum principle,  \cite{cavagnari2020generalized} for a generalized version of dynamic programming, \cite{bonnet2020mean} for the study of differential inclusions and  the contribution \cite{jimenez2020optimal} for the analysis of {necessary and sufficient conditions for optimality in the form of a Hamilton-Jacobi-Bellman} equation in the Wasserstein space.	  

\smallskip
 
\emph{Lagrangian}, \emph{Eulerian} and \emph{Kantorovich} refer to different points of view that can be adopted to study the dynamics of the problems.
The Lagrangian and Eulerian terminologies come from fluid-dynamics and they have been recently adopted in the theory of optimal transport, from which we also took inspiration for the Kantorovich formulation.  
In general terms, the \emph{Lagrangian} approach consists in labelling each particle and following the corresponding trajectory.
The \emph{Eulerian} description, on the other hand, aims at measuring the velocity of particles flowing at a point at a fixed time.
In this paper we introduce a further point of view, that we name \emph{Kantorovich} in analogy with the Kantorovich extension of Monge problem (in the same spirit, see also \cite{ambrosio2009geodesics, ambrosio2008gradient, benamou2017variational}). 
The Kantorovich formulation turns out to be fundamental in connecting the Lagrangian and Eulerian points of view and it is based on the representation of solutions of the continuity equation provided by the superposition principle (see e.g. \cite[Theorem 8.2.1]{ambrosio2008gradient} or \cite[Theorem 5.8]{bernard2008young}).

\smallskip

\emph{Limit theory} refers instead to the question of connecting optimal control problems with a finite number of agents with the infinite dimensional description given by a continuum of players.  
The study of interacting particle systems becomes intractable when the population is very large and, in many circumstances, the connection with a limit (mean field) approximation heavily simplifies the study.
For this reason, many results concerning asymptotic behaviour when the size of the system grows have been developed in the literature, both from the theoretical and applicative point of view.
Even if mean field approximations  are mainly proved in the uncontrolled stochastic setting,  some applications to deterministic models can be found e.g.  in \cite{naldi2010mathematical} and \cite{difrancesco2015rigorous}.
Concerning the controlled case, an important effort has been directed to the case of mean field games (see below for some references) but, to the best of our knowledge, only few rigorous results for deterministic multi-agent systems can be found in the literature. 
Let us mention \cite{fornasier2014mean}, \cite{fornasier2018mean} for Gamma-convergence techniques under different assumptions on the controls and velocity fields, and  \cite{cesaroni20201} for an extension when distance constraints among the agents are imposed.

\smallskip

The analysis we develop here aims at providing a unifying framework for the study of deterministic multi-agent optimal control problems and surely benefits from the connections with the theory of optimal transport.  
This is more evident when dealing with a measure-formulation of the problem, where a non-local continuity equation is guiding the dynamics, or with the Kantorovich description, where a superposed measure in the space of continuous paths selects the trajectories of the system.
Another source of inspiration is the theory of Young measures, classically used in control theory, which plays a crucial role in the description of relaxed problems. 

Apart from its theoretical interest, we trust the present investigation  could serve as a founding step into a general treatment of optimality conditions in the rapidly-growing context of Wasserstein spaces.  
Within this context a major role is played by the Hamilton-Jacobi-Bellman (HJB) equations in the space of probability measures, for which different notions of solution has been already proposed in the literature.
We think that the analysis of HJB equation could benefit from the equivalence results and the limit theory developed here and we leave it to future investigation .
Let us just briefly mention some contributions in this direction:
a general analysis in metric setting can be found  in \cite{ambrosio2008hamiltonan, gangbo2015metric}, see also  \cite{gangbo2019differentiability} for the particular choice of the  Wasserstein space. 
Viscosity solutions for HJB are studied e.g. in  \cite{cardaliaguet2008deterministic, cosso2019} for the case of random differential games and in \cite{jimenez2020optimal} for multi-agent systems.

\smallskip

Let us finally report on the stochastic counterpart of the theory (that we do not treat herein) and make some further comments on the connections with  mean field games.  

\smallskip 

\textbf{Stochastic counterpart.}
In the stochastic setting, mean field behaviour of interacting particles systems is classically referred to as  \emph{propagation of chaos}.
The literature on the subject is far too vast to be discussed here and the interested reader is referred to Sznitman's Saint-Flour lectures \cite{sznitman1991topics} for a beautiful treatment of the subject and to the references in e.g. \cite{djete2020mckean, orrieri2018large} for some of the more recent  developments.

Concernig stochastic control problems, a rigorous consistency result  for controlled  McKean-Vlasov dynamics has been obtained by Lacker in \cite{lacker2017limit}
using martingale problems and relaxation.
It is interesting to notice that the result contained in \cite{lacker2017limit} allows also for degenerate diffusion.
Further extensions has been pursued in 
\cite{djete2020mckean}, where a common noise is also introduced,  and 
 in \cite{djete2020extended} where the state dynamics depends upon the joint distribution of state and control.

For what concerns equivalence results, to the best of our knoledge, the more general study in the stochastic setting  is formulated in  \cite{djete2020mckean}.
There, the authors prove existence of optimal controls and show the equivalence at the level of value functions of the so-called strong, weak (in a probabilistic sense) and relaxed (Lagrangian) formulations of the stochastic control problems.
Notably, the results contained in \cite{djete2020mckean} extends the ones in \cite{lacker2017limit} to the far more general case of common noise.
Observe also that the equivalence between the weak and strong formulations has a fundamental role in establishing a dynamic programming principle in \cite{djete2019mckean}.

To conclude, let us just remark on a possible application of deterministic consistency results to the study of limit behaviour of uncontrolled stochastic particle systems.
The idea stems from the link between Gamma-convergence and large deviations developped by Mariani in \cite{mariani2018gamma}.
If a consistency result for deterministic control problems is proved in terms of Gamma-convergence, then it is possible to translate the $\Gamma-\liminf$ and $\Gamma-\limsup$ inequalities in corresponding lower and upper bound estimates for suitable associated stochastic systems.
An example of this technique is contained in \cite{orrieri2018large} (see also \cite{budhiraja2020asymptotic} for a recent extension) where the limit theory developed in \cite{fornasier2018mean} has been used to prove a large deviations principle for stochastic equations in the mean-field and small-noise regime.

\smallskip

\textbf{Connections with MFGs.} The theory of Mean Field Games (MFGs), separately introduced by Lasry and Lions \cite{lasry2007mean} and Huang, Caines, Malham\'e \cite{huang2006large} aims at describing non-cooperative indistinguishable players interacting through their empirical distribution. 
Opposed to Vlasov control problems, the optimization problem of each agent in mean field games leads to the notion of Nash equilibrium for the system. 
Differences and similarities between centralized optimal control of (McKean)-Vlasov dynamics and equilibria in MFGs are discussed in the literature, see e.g. \cite{carmona2013control}, \cite{bensoussan2013mean} and \cite{cardaliaguet2019inefficiency}.
For what concerns the limit of $N$-players  differential games towards the MFG system, a fundamental result was obtained in \cite{cardaliaguet2019master} via the so-called Master equation. 
Among the various extensions of \cite{cardaliaguet2019master}, let us mention the convergence of open and closed-loop Nash equilibria to MFGs equilibria obtained respectively in \cite{fischer2017connection} and \cite{lacker2020convergence}. 
The case of first-order MFGs has been taken into account in  \cite{fischer2020asymptotic}.

\bigskip    
    
We now briefly describe the various formulations of the optimal control problems we deal with and we present the main results contained in the paper.

To improve the readability of the introduction, we just sketch the essential features of the problems, omitting the details and heavily simplifying the setting and theorems whenever possible.
Precise statements are given in the forthcoming sections.    
Throughout the paper we use interchangeably the terms particles/agents as they differ only in view of the different applications.

\smallskip

\textbf{Lagrangian formulation {($\pL$)}.} The Lagrangian formulation has a probabilistic flavour and it is built upon a probability space $(\Omega, \frB, \P)$ which acts as a parametrization space for the particles.
More precisely, given a finite time horizon $T >0$ and a (compact metric) space of control actions $U$,  
the controlled dynamics $X:[0,T] \times \Omega \to \R^d$ is given by 
\begin{equation}\label{eq:systemL_intro}
\begin{cases}
\dot X_t(\omega)=f(X_t(\omega),u_t(\omega),(X_t)_\sharp\P), &\textrm{for a.e. }t\in (0,T)\\ 
X_{|t=0}(\omega)=X_0(\omega),&
\end{cases}
\end{equation}
where the dependence of  the vector field $f$ on the measure $(X_t)_\sharp \P$ models the interaction among particles {and/or the interaction of the mass with the surrounding environment} and it is usually referred to as a \emph{mean field} interaction.

 A natural motivation for the introduction of a parametrization space comes from a large variety of problems where a finite number of particles/agents are involved.
In this case $\Omega$ can be simply  interpreted as a set of labels $\Omega^N = \{1, \ldots, N\}$, with $\frB^N$ the associated algebra of parts and $\P^N(\{\omega\}) = \frac{1}{N}$, $\omega = 1, \ldots, N$, the normalized counting measure.
Each particle is indistinguishable from the others and the interaction enters the system through the empirical measure $\mu^N_t=\frac{1}{N}\sum_{\omega=1}^N \delta_{X_t(\omega)}$, with $t \in [0,T]$.
\smallskip

Given $X_0\in L^p(\Omega;\R^d)$, for $p\geq 1$, an admissible pair $(X,u) \in \cA_{\pL}(X_0)$ 
 for the Lagrangian optimal control problem ($\pL$) consists in a measurable control  $u:[0,T] \times \Omega \to U$ and a solution (in a suitable sense) of \eqref{eq:systemL_intro}.
Associated to the dynamics, the \emph{cost functional} to be minimized has the form   
\[J_{\pL}(X,u):= \int_\Omega \int_0^T\mathcal C(X_t(\omega),u_t(\omega),(X_t)_\sharp \P)\,\d t \, \d\P(\omega) +\int_\Omega \mathcal C_T(X_T(\omega),(X_T)_\sharp \P)\,\d\P(\omega),\]
where the running cost $\mathcal C$ and the final cost $\mathcal C_T$ are non-local as they could depend on the measures $(X_\cdot)_\sharp \P$ and $(X_T)_\sharp \P$. 

{The optimization of the cost functional among admissible pairs leads to the definition of the so-called} \emph{value function} which, for the Lagrangian problem, can be written as   
\begin{equation}\label{value:funct:L_intro}
V_{\pL}(X_0):=\inf\left\{J_{\pL}(X,u)\,:\,(X,u)\in\mathcal A_{\pL}(X_0)\right\}.
\end{equation}
Let us notice that existence of optimal pairs $(X,u)$ (for which the minimum is achieved in \eqref{value:funct:L_intro}) is not guaranteed in general. 
A counterexample for the Lagrangian problem, even in the relaxed formulation,  is given by the Wasserstein barycenter problem with suitable initial distribution (see Section \ref{sec:counterexample} for details).
From a probabilistic point of view, the Lagrangian formulation can be thought as a \emph{random} optimal control problem in strong formulation, where the randomness is encoded in the initial distribution of the dynamics. 
In this context, we do not consider stochastic perturbation of the dynamics given e.g. from independent Brownian motions and/or common noise.

\smallskip

\textbf{Relaxed Lagrangian formulation {($\RL$)}.} A classical generalization of optimal control problems is the so called \emph{relaxed} version, where controls are allowed to take values in the space of probability measures $\sigma: [0,T] \times \Omega \to \PP(U)$.
This greatly enlarges the class of admissible pairs (classical controls can be recovered choosing $\sigma:= \delta_u$, with $u \in U$) and provides a convexification of the problem under consideration. Indeed, the controlled trajectories satisfy the linear (in the control action) dynamics
\begin{equation*}
\begin{cases}
\dot X_t(\omega)=\displaystyle\int_U f(X_t(\omega),u,(X_t)_\sharp\P)\,\d\sigma_{t,\omega}(u), &\textrm{for a.e. }t\in]0,T]\\
X_{|t=0}(\omega)=X_0(\omega), &
\end{cases}
\end{equation*}
with control $\sigma_{t,\omega}:=\sigma(t,\omega)\in\mathscr P(U)$.
Furthermore, the cost functional takes the form 
\begin{equation*}
J_{\RL}(X,\sigma):=\int_{\Omega}\int_0^T\int_U \mathcal C(X_t(\omega),u,(X_t)_\sharp\P)\,\d\sigma_{t,\omega}(u)\, \d t \,\d\P(\omega) \ +\int_\Omega \mathcal C_T(X_T(\omega),(X_T)_\sharp\P)\,\d\P(\omega),
\end{equation*}
with associated \emph{value function} $V_{\RL}:L^p(\Omega;\R^d)\to[0,+\infty)$ given by
\begin{equation*}
V_{\RL}(X_0):=\inf\left\{J_{\RL}(X,\sigma)\,:\,(X,\sigma)\in\mathcal A_{\RL}(X_0)\right\}.
\end{equation*}

Relaxation is a fundamental concept in optimal control theory and has its roots in the theory of Young measures.
In Section \ref{sec_relaxed} of the paper,  we provide a detailed analysis of the relaxation procedure in the context of multi-agent systems, emphasizing its connections with the Lagrangian problem. 
Of particular interest is the extension of a suitable version of the chattering theorem, which permits to approximate the Relaxed Lagrangian formulation with a sequence of (not relaxed) Lagrangian ones.  This readily implies the equality of the respective value functions: $V_{\RL}(X_0) = V_\pL(X_0)$ for any $X_0 \in L^p(\Omega;\R^d)$.

Let us finally notice that the Relaxed Lagrangian formulation is the prototype for the class of control problems satisfying suitable  \emph{Convexity Assumptions} (see Assumption \ref{CA} below).
In this particular case, the control space is the convex space of measures $\PP(U)$, the dynamics is affine in the controls and the cost functional is convex (actually linear).   
However, as already observed before, this relaxation procedure is not sufficient to guarantee existence of minimizers for a general optimal control problem (see Section \ref{sec:counterexample}).
A further step in this direction is the introduction of the Eulerian formulation of the problem. 
\smallskip

\textbf{Eulerian formulation {($\E$)}.} 
To simplify the presentation, here we suppose to directly deal with relaxed controls, which are represented by a Borel measurable map $\sigma: [0,T] \times \R^d \to \PP(U)$: this is fundamental to get existence of minimizers.
Then, the evolution of the system is guided by the following non-local Vlasov equation
   \begin{equation*}
\begin{cases}
\partial_t\mu_t+\mathrm{div}\,(v_t\mu_t)=0, &\textrm{in }[0,T]\times\R^d\\ 
\mu_{t=0}=\mu_0,&
\end{cases}
\end{equation*}
where the controlled vector field $v$ depends on the evolving state itself and it is given by $v_t(x):=\int_{U}f(x, u,\mu_t)\d \sigma_{t,x}(u)$.
Within this framework,  the cost functional takes the form
\begin{equation*}
J_{\E}(\mu,\sigma):=\int_0^T\int_{\R^d}\int_U\cC(x, u,\mu_t)\, \d \sigma_{t,x}(u) \,\d\mu_t(x)\,\d t+\int_{\R^d} \cC_T(x,\mu_T)\,\d\mu_T(x),
\end{equation*}

and the \emph{value function} $V_{\E}:\PP_p(\R^d)\to[0,+\infty)$ is given by
\begin{equation*}
V_{\E}(\mu_0):=\inf\{J_{\E}(\mu,\sigma)\,:\,(\mu,\sigma)\in\mathcal A_{\E}(\mu_0)\}.
\end{equation*}

In the Eulerian description of the optimal control problem, the system can be described by a curve of probability measures. 
This point of view is intimately connected with the theory of optimal transport from which ideas and techniques are borrowed.

\smallskip

\textbf{Kantorovich formulation {($\K$)}.}
A somewhat intermediate formulation is given by the \emph{Kantorovich} optimal control problem (in analogy to the Kantorovich formulation of the optimal transport problem).
This formulation has its roots in the representation of solutions of the continuity equation by superposition of continuous curves belonging to $\Gamma_T:= C([0,T];\R^d)$. 
An admissible pair for the Kantorovich problem is given by $(\eeta,\sigma)$ where $\eeta \in \PP(\Gamma_T)$ is a probability measure on the space of continuous curves, and $\sigma: [0,T] \times \Gamma_T \to \PP(U)$ is a relaxed control.
Furthermore, given $\mu_0\in\PP(\R^d)$ with finite $p$-moment, an admissible measure $\eeta$ has to match the initial condition in the form $(e_0)_\sharp\eeta = \mu_0$.
Even more important, defining  $\mu_t:=(e_t)_\sharp\eeta$ for all $t\in[0,T]$, 
$\eeta$ has to be concentrated on the set of absolutely continuous solutions  of the differential equation
\[\dot\gamma(t)=\int_U f(\gamma(t),u,\mu_t)\d \sigma_{t, \gamma}(u) ,\qquad \text{for $\cL_T$-a.e. }t\in[0,T].\]
This clearly links the Kantorovich formulation with the Eulerian one via the superposition principle (see Theorem \ref{thm:sup_princ}). 
The \emph{cost functional} associated to the Kantorovich formulation is written in the form 
\begin{equation*}
J_{\K}(\eeta,\sigma):=\int_{\Gamma_T}\int_0^T\int_U\cC(\gamma(t),u,\mu_t)\, \d \sigma_{t, \gamma}(u)\,\d t \, \d\eeta(\gamma) +\int_{\R^d} \cC_T(x,\mu_T)\,\d\mu_T(x),
\end{equation*}
where $\Gamma_T$ act as a parametrization space (as in the Lagrangian framework), but the minimization involves measures $\eeta \in \PP(\Gamma_T)$ instead of trajectories, in line with the Kantorovich formulation of the optimal transportation problem: $\mu_t = (e_t)_\sharp \eeta$ can be considered as a time dependent family of marginals of $\eeta$.
The associated \emph{value function} is given by
\begin{equation*}
V_{\K}(\mu_0):=\inf\{J_{\K}(\eeta,\sigma)\,:\,(\eeta,\sigma)\in\cA_{\K}(\mu_0)\}.
\end{equation*}
Let us stress that, also in this setting, the choice of relaxed controls is sufficient to prove the existence of minimizers.

\smallskip

\textbf{Equivalence results.}
A natural question is whether the problems introduced above are somewhat related.
One of  the aims of the present paper is to prove equivalences among the various formulations introduced above.
At the level of the value functions, we can summarize the main result in the following theorem (see Theorem \ref{cor:VL=VE'} for a precise statement)
\begin{theorem*}[equivalence]
Let $(\Omega, \cB, \P)$ be a Polish space such that $\P$ is without atoms.
If $X_0\in L^p(\Omega;\R^d)$, then 
$$V_{\pL}(X_0)=V_{\RL}(X_0)= V_{\E}((X_0)_\sharp\P)=V_{\K}((X_0)_\sharp\P).$$
\end{theorem*}

The first step of the proof consists in the approximation of the Lagrangian problem by piecewise constant controls (see Theorem \ref{prop:RLntoRL}).
This is possible whenever the probability space $(\Omega, \frB, \P)$ under consideration satisfies a suitable 
\emph{finite approximation property} (Definition \ref{approx_prop}), which surely holds in the Polish framework.  
Once the piecewise approximation is established, we are able to formulate a suitable version of the \emph{chattering theorem}  (see Theorem \ref{thm:chat}) where trajectories, controls and cost functional of the Relaxed Lagrangian formulation are approximated by the corresponding objects in the Lagrangian setting.

The comparison between the Lagrangian and Eulerian formulations (see { Theorems \ref{cor:VL=VE}, \ref{cor:VL=VE'} and Section \ref{sec:LEK}}) is 
more delicate and it is achieved by exploiting the Kantorovich description of the control problem on the space of curves $\Gamma_T$. 
The idea is to separately connect the Eulerian and Kantorovich descriptions (Theorem \ref{cor:E=K}) and then the Kantorovich and Lagrangian ones (Theorem \ref{prop:K>L}). 

Starting from an admissible pair for the Eulerian problem, the application of the superposition principle given in Theorem \ref{thm:sup_princ} easily provides a candidate admissible pair for the Kantorovich problem paying the same cost (see  the proof of Proposition \ref{prop:E>K} for details). 
Conversely, if a pair $(\eeta,\sigma)$ for the Kantorovich problem is given, the Eulerian control action can be obtained by averaging with respect to a suitable disintegration of the measure $\eeta$ on curves, as it is shown in the proof of Proposition  \ref{prop:K>E}. 

For what concerns the Kantorovich/Lagrangian comparison, in Lemma \ref{l:fromKtoL} we firstly interpret the Kantorovich problem as a Lagrangian one with parametrization space given by $\Gamma_T$, i.e. the space of curves, and with trajectories given by the evaluation map.
Due to the continuity of the initial datum (i.e. the evaluation map $e_0$),  we then approximate controls with continuous ones (see Proposition \ref{lemmaB} for a general result in this direction) and finally we approximate the {obtained Lagrangian problem in $\Gamma_T$} with Lagrangian {ones which are set in} a generic parametrization space $\Omega$, not necessarily the space $\Gamma_T$.
A precise description of this technique is contained in the proof of Theorem \ref{prop:K>L}.  

An immediate consequence of the equivalence theorem is the equality of the value functions for different initial data, whenever the respective laws coincide (see Theorem \ref{cor:VL=VE}). In fact, if $\mu_0 := (X_0)_\sharp\P=(X_0')_\sharp\P$, then it holds that
\[V_{\pL}(X_0)=V_{\pL}(X_0') \quad \big( \, =V_{\E}(\mu_0) \big).\]
A further consequence is the continuity of the value functions with respect to the initial data, see Theorems \ref{th:contVE} and \ref{th:contVL} for precise statements.

\smallskip

Equivalence results between Lagrangian, Eulerian and Kantorovich formulations represent a first step towards a general analysis of optimality conditions  for multi-agents control systems. A second step in this direction is the study of the corresponding limit theory.   
\smallskip

\textbf{Approximation by finite particle systems.} 
We aim to provide a rigorous limit theory for multi-agent optimal control problems both for the Lagrangian and Eulerian formulations.
To do it, we have to define appropriate discrete versions of the two formulations. 
The $N$-particle Lagrangian control problem $\pL^N$ simply relies on the choice of  $\Omega^N= \{1, \ldots, N \}$ as parametrization space.
On the other hand, a genuine discrete Eulerian problem $\E^N$  requires the introduction of a constraint on the number of particles (see Definition \ref{def:E^N}), precisely an admissible trajectory $\mu$ 
satisfies $\mu_t\in\PP^N(\R^d)$, where
$$\PP^N(\R^d):= \left\{\mu = \frac{1}{N} \sum_{i =1}^N \delta_{x_i} \; \text{ for some } x_i \in \R^d \right\}.$$ 
The discrete Eulerian and Lagrangian control problems $\pL^N$ and $\E^N$ turn out to be equivalent (see Theorem \ref{cor:LN=EN}) in the sense that 
\begin{equation}\label{equiv_E_L_N}
V_{\pL^N}(X_0)= V_{\E^N}((X_0)_\sharp\P^N), \quad \text{ for any } X_0\in L^p(\Omega^N;\R^d).
\end{equation}
Note that this is not a direct consequence of the general equivalence result given above, where the reference probability measure $\P$ was required to be without atoms. 
To prove the equality in \eqref{equiv_E_L_N} we derive a discrete formulation of the superposition principle for empirical probability measures  (see Theorem \ref{lem:EN}) that we believe might be of interest in itself. 

Once the equivalence at the level of $N$-particle systems is established, we derive Gamma-convergence results respectively for the Lagrangian and Eulerian problems as the number of particles diverges (see Propositions \ref{prop:gamma_conv} and \ref{prop:ENtoE}). 
A major consequence is contained in the following theorem (see Theorem  \ref{thm:convVE} for a detailed description).
\begin{theorem*}\label{thm:convVE_intro}
Let $(\Omega,\cB,\P)$ be a Polish space such that $\P$ is
without atoms. Assume that the Convexity assumption \ref{CA} holds.
\begin{itemize}
\item If  $X_0\in L^p(\Omega;\R^d)$ and $X^N_0:\Omega^N\to\R^d$, $N\in\N$, satisfy   $X^N_0 \to X_0$ as $N\to+\infty$ (see \eqref{CNOmega}), 
then  
\[\lim_{N \to +\infty} V_{\pL^N} (X_0^N) = V_{\pL}(X_0).\]
\item If $\mu_0\in\PP_p(\R^d)$ and $\mu_0^N\in\PP^N(\R^d)$, $N\in\N$, satisfy $W_p(\mu_0^N,\mu_0)\to 0$ as $N\to+\infty$,
then 
\[\lim_{N \to +\infty} V_{\E^N} (\mu_0^N) = V_{\E}(\mu_0).\]
\item Moreover, if $(X_0^N)_\sharp \P^N = \mu_0^N $ it holds that  
\[\lim_{N \to +\infty} V_{\pL^N} (X_0^N) = V_{\E}(\mu_0).\]
\end{itemize}
\end{theorem*}
Notice that the usual \textit{mixed} Lagrangian-Eulerian consistency in the third item is a simple byproduct of the equivalence in \eqref{equiv_E_L_N}.
  
\medskip

\textbf{Structure of the paper.}
The paper is organized as follows.
In Section \ref{sec_prel} we fix the notation and present some preliminary material.
We start by revising some properties of Borel probability measures and we recall a refined version of the Skorohod representation theorem. 
We further provide some material on optimal transport, Wasserstein spaces and we present the classical superposition principle.
Finally, we discuss the disintegration theorem and give some properties of Young measures.

Section \ref{sec:assumptions} contains our standing hypotheses, divided into two sets:  the \emph{Basic} Assumptions and the \emph{Convexity} Assumptions.  
The first ones require compactness of the metrizable space of controls and Lipschitz continuity of the velocity field with respect to the state and the mass distribution.
The cost functional has to be continuous and to satisfy a polynomial  growth condition.
The convexity assumptions impose convexity of the control set and of the cost functional (with respect to controls). Furthermore, the dynamics has to be affine with respect to the control actions.   
The relaxed setting is the guiding example of the convex case.

In Section \ref{sec_lagrangian} we present and study the Lagrangian optimal control problem. 
In particular, we exhibit two different approximation procedures: the first one by piecewise constant controls and the second one by continuous controls and trajectories. 
 
The Relaxed Lagrangian problem is defined in Section \ref{sec_relaxed}, where its representation as a Lagrangian problem in the lifted space of measures is also discussed.
We prove the equivalence between Lagrangian and Relaxed Lagrangian formulations of the control problem by approximating relaxed controls with (non-relaxed) ones in the sense of Young convergence.  
This procedure, known as chattering theorem, exploits the approximation by piecewise constant controls developed in the Lagrangian setting. 

Section \ref{sec_eulerian} contains the definition and properties of the Eulerian control problem. 
Under the convexity Assumptions we are able to prove existence of minimizers for the control problem via a direct method.

In Section \ref{sec:K} we introduce the Kantorovich problem and  we prove its equivalence with the Eulerian one under the convexity assumptions. 
 
The main equivalence results are contained in Section  \ref{sec:E=L}.
Exploiting the Kantorovich formulation, we firstly show the equality between value functions of Lagrangian and Eulerian problems under the Convexity Assumptions.
The general case is then obtained by interpreting the  Relaxed Lagrangian as a  (convex) Lagrangian problem in the space of probability measures.
As a consequence, we also get the continuity of the value functions, with respect to the initial condition, for the various formulations.
In Subsection \ref{sec:counterexample} we discuss the possible non-existence of minimizers for the Lagrangian and Relaxed Lagrangian problems.
This is not guaranteed, if the initial condition is assigned, even under the Convexity Assumptions. 
We produce a counterexample to the existence of minimizers given by the Wasserstein barycenter problem. 

Section \ref{sec:finite} contains all the material regarding finite particle problems and the respective limit theory.
We define a discrete version of the Eulerian control problem imposing a constraint on the number of particles.
To connect the Eulerian and Lagrangian formulations in the $N$-particle case we introduce a \emph{Feedback} Lagrangian control problem, where control actions are indeed in feedback form, and we make use of the already mentioned discrete superposition principle. 
A Gamma-convergence result both for the Lagrangian and Eulerian formulations is then established  as the number of particles tends to infinity.
We finally conclude proving the convergence of the associated value functions.
\smallskip

The appendix contains various technical tools.
In particular, in Appendix \ref{app:SP} we state and prove the superposition principle for the evolution of empirical measures.

\section{Preliminaries and notations}
\label{sec_prel}

{We list here the main notation.}

{\small\begin{longtable}{ll}
$\# A$&the cardinality of a set $A$;\\
$i_X(\cdot)$&the identity function on a set $X$, $i_X:X\to X$ defined by $i_X(x)=x$;\\
$\mathds 1_A(\cdot)$&the characteristic function of $A\subset X$,\\ 
&$\mathds 1_A:X\to \R$ defined by
  $\mathds 1_A(x)=1$ if $x\in A$, $\mathds 1_A(x)=0$ if $x \in X\setminus A$;\\
$(S, \frB)$&measurable space $S$ with $\sigma$-algebra $\frB$;\\
$(S, \cB_{S})$&topological space $S$ with Borel $\sigma$-algebra $\cB_S$;\\   
$\BM(X;Y)$&the set of measurable functions from the measurable space $X$ \\
&to the measurable space $Y$;\\
$\Bor(X;Y)$&the set of Borel measurable functions from the topological space $X$ \\
&to the topological space $Y$;\\
$C(X;Y)$&the set of continuous functions from the topological space $X$ \\
&to the topological space $Y$;\\
$C_c(X;Y)$&the set of continuous compactly supported functions from the topological space $X$\\
& to the topological space $Y$;\\
$C_b(X;Y)$&the set of continuous bounded functions from the topological space $X$\\
& to the metric space $Y$;\\
$\AC^p([0,T];X)$&the set of absolutely continuous functions from $[0,T]$ to the metric space $X$\\
&with metric derivative in $L^p([0,T];\R)$;\\
$\Gamma_T$&the set of continuous curves from $[0,T]$ to $\R^d$, i.e., $\Gamma_T=C([0,T];\R^d)$;\\
$e_t$&the evaluation map at time $t\in[0,T]$, $e_t:\Gamma_T\to\R^d$ defined by $e_t(\gamma)=\gamma(t)$;\\
$\mathscr P(X)$&the set of probability measures on the measurable space $X$;\\
{$\PP^N(\R^d)$}&{the set of empirical probability measures on $\R^d$ defined in \eqref{eq:discreteProb};}\\
$\mathrm{m}_p(\mu)$&the $p$-th moment of a probability measure $\mu\in\PP(\R^d)$, defined by\\ & ${\m}_p(\mu)=\left(\int_{\R^d}|x|^p\,\d\mu(x)\right)^{1/p}$;\\
$r_\sharp\mu$&the push-forward of the measure $\mu\in\mathscr P(X)$ by the measurable map $r\in\BM(X;Y)$;\\
$\mu\otimes\nu$&the product measure of $\mu \in\mathscr P(X)$ and $\nu \in \PP(Y)$;\\
$\pi^i$&the $i$-th projection map $\pi^i:X_1\times\cdots\times X_N\to X_i$ defined by $\pi^i(x_1,\dots,x_N)=x_i$;\\
$\pi^{i,j}$&the $(i,j)$-th projection map  $\pi^{i,j}:X_1\times\cdots\times X_N\to X_i\times X_j$\\
& defined by $\pi^{i,j}(x_1,\dots,x_N)=(x_i,x_j)$;\\
$W_p(\mu,\nu)$&the $p$-Wasserstein distance between $\mu$ and $\nu$ (see Definition \ref{def:was});\\
$\mathscr P_p(\R^d)$&the metric space of the elements in $\mathscr P(X)$ with finite $p$-moment, \\
&endowed with the $p$-Wasserstein distance;\\
$\cL_T$&the normalized Lebesgue measure restricted to the interval $[0,T]$,\\
&i.e. $\cL_T:=\frac{1}{T}\cL \mres [0,T]$.\\
\end{longtable}}

\subsection{Borel probability measures}
Let $(S,\frB)$ be a measurable space.
When $S$ is a Polish topological space, we will
implicitely assume that $\frB$ coincides
with the Borel $\sigma$-algebra $\cB_S$ of $S$.
We say that 
$(S,\frB)$
is \emph{a standard Borel space} if it is isomorphic (as a
measurable space) to a Borel subset
of a complete and separable metric space; equivalently, one can find a
Polish topology
$\tau$ on $S$ such that $\frB=\cB_{(S,\tau)}$.

If $(E,\tilde \frB)$ is another measurable space, we denote by $\BM(S;E)$ the set of
measurable functions from $S$ to $E$.
If $S$ is a topological space we denote with $\Bor(S;E)$ the set of Borel measurable functions. 
$\PP(S)$ is the set of probability measures on $S$;
when $S$ is a Polish space (and $\frB=\cB_S$) 
we will endow $\PP(S)$ with the
weak (Polish) topology induced by the duality
with the continuous and bounded functions of $\mathrm C_b(S):= C_b(S;\R)$.

Given $\mu\in\PP(S)$ and $r:S\to E$ a measurable map, 
we define the \emph{push forward of $\mu$ through} $r$, denoted by $r_\sharp\mu\in\PP(E)$, by 
$r_\sharp\mu(B):=\mu(r^{-1}(B))$ for all measurable sets
$B\in\tilde \frB$ {(the $\sigma$-algebra on $E$)}, or equivalently,
\[\int_S f(r(x))\,\d\mu(x)=\int_E f(y)\,\d r_\sharp\mu(y),\]
for every positive, or $r_\sharp\mu$-integrable, function $f:E\to\R$. \\
Given another measurable space $Z$, $\mu\in\PP(S)$, and $r:S\to E$, $s:E\to Z$ measurable maps, 
the following composition rule holds
\begin{equation}\label{eq:composition}
(s\circ r)_\sharp\mu=s_\sharp(r_\sharp\mu).
\end{equation}
Moreover, if $r:S\to E$ is a continuous map (with respect to suitable
Polish topologies in $S$ and $E$) then
$r_\sharp:\PP(S)\to\PP(E)$ is continuous as well.

The following proposition generalizes to some extent the classical Skorohod representation Theorem, see e.g. \cite[Theorem 6.7]{billingsley1999convergence}. 
For a (more general) result and the proof we refer to \cite[Theorems~3.1~and~3.2]{berti2007skorohod}.

\begin{proposition}\label{prop:Skorohod}
Let $(\Omega, \frB, \P)$ be a probability space such that $\P$ is without atoms and let $S$ be a Polish space. 
\begin{itemize}
\item[(i)] If $\nu \in \PP(S)$, then there exists a measurable map $X: \Omega \to S$ such that $X_{\sharp} \P = \nu$.
\item[(ii)] If $\nu^n, \nu \in \PP(S)$ with $\nu^n \to \nu$ weakly,  then there exist measurable maps $X^n,X : \Omega \to S$, $n \in \N$,  such that $X^n_\sharp \P = \nu^n$, $X_\sharp \P = \nu$ and $X^n(\omega) \to X(\omega)$ for $\P$-a.e. $\omega \in \Omega$. 
\end{itemize}  
\end{proposition}
Notice that, when $(\Omega,\frB,\P)$ is a standard Borel space, $\tau$ a Polish topology on $\Omega$ such that $\frB = \cB_{(\Omega,\tau)}$, then the maps $X$ and $X^n$ in Proposition \ref{prop:Skorohod} are Borel measurable.
A particular and significant case occurs when we choose $(\Omega, \frB, \P) = ([0,1], \cB, \cL_1)$, with $\cB$ the Borel $\sigma$-algebra and $\cL_1$ the Lebesgue measure restricted to $[0,1]$.

\medskip

If $ \mm\in\PP(S)$ and $E$ is a separable 
Banach space, we denote by 
$L^p_\mm(S;E)$ the space of (the equivalence classes of) $\mm$-measurable  functions $f:S\to E$
such that $\int_S \|f(x)\|^p\,\d\mm(x)<+\infty$. Since $E$ is separable, the notions of weak and strong
measurability coincide.
We will often adopt the notation $L^p(S;E)$ in place of $L^p_\mm(S;E)$ when the measure $\mm$ is clear from the context.\\
We say that a sequence of measurable functions $u_n\in\BM(S;E)$ converges in $\mm$-measure to 
$u\in\BM(S;E)$ if
 \begin{equation}
   \forall\,\varepsilon > 0,
   \quad
   \lim_{n \to +\infty} \mm \left( \lbrace x \in S: \|u_n(x) - u(x)\| \geq \varepsilon \rbrace  \right) = 0.
 \end{equation} 
If $u_n$ take values in a compact subset $U$ of $E$,  $u_n\in\BM(S;U)$,
the convergence of $u_n$ to $u\in\BM(S;U)$ in $\mm$-measure
is equivalent to the convergence of $u_n$ to $u$ in $L^p(S;E)$ for every $p\in[1,+\infty)$.

Given $(S, d)$ a metric space and $p \in [1,+\infty]$, we say that a curve $\gamma: [0,T] \to S$ belongs to $\AC^p([0,T];S)$ if there exists $m \in L^p(0,T;\R)$ such that 
\begin{equation*}
d(\gamma(t_1), \gamma(t_2)) \leq \int_{t_1}^{t_2} m(s) \, \d s , \qquad  \forall \, t_1, t_2 \in [0,T], \; t_1 \leq t_2. 
\end{equation*}

\medskip

\subsection{The Wasserstein metric and the Superposition Principle}
We provide a brief collection of the main notions on optimal transport and Wasserstein distance, 
addressing the reader to \cites{ambrosio2008gradient,santambrogio2015optimal,villani2003topics}. \par\medskip\par

Given $\mu\in\PP(\R^d)$ and $p\geq 1$, we define the $p$-moment of $\mu$ by 
$$\mathrm{m}_p(\mu):=\left(\int_{\R^d}|x|^p\,\d\mu(x)\right)^{1/p}.$$
We define $\PP_p(\R^d):=\{\mu\in\PP(\R^d):\mathrm{m}_p(\mu)<+\infty \}.$
The set $\PP_p(\R^d)$ can be metrized by the following distance.

\begin{definition}[Wasserstein distance]\label{def:was}
Let  $p\ge 1$. Given $\mu_1,\mu_2\in\mathscr P_p(\R^d)$, we define the $p$-\emph{Wasserstein distance} between $\mu_1$ and $\mu_2$
by setting
\begin{equation}\label{eq:wasserstein}
W_p(\mu_1,\mu_2):=\left(\min\left\{\int_{\R^d\times \R^d}|x_1-x_2|^p\,d\gamma(x_1,x_2)\,:\,\gamma\in \Gamma(\mu_1,\mu_2)\right\}\right)^{1/p}\,,
\end{equation}
where the set of \emph{admissible transport plans} $\Gamma(\mu_1,\mu_2)$ is given by
\begin{align*}
\Gamma(\mu_1,\mu_2):=&\left\{\gamma\in \mathscr P(\R^d\times \R^d):\,\pi^1_\sharp\gamma=\mu_1,\pi^2_\sharp\gamma=\mu_2 \right\},
\end{align*}
with $\pi^i:\R^d\times\R^d\to\R^d$, $\pi^i(x^1,x^2)=x^i$, the projection operator, $i=1,2$.
\end{definition}
By the previous definitions,
given a measurable space $\Omega$ and $\P\in\PP(\Omega)$, it follows immediately that
 for any $Z\in  L^p(\Omega;\R^d)$, we have $\mu:=Z_\sharp\P\in\PP_p(\R^d)$ and 
\begin{equation}\label{eq:MpLp}
	\mathrm m_p(\mu)= \|Z\|_{ L^p(\Omega;\R^d)},
\end{equation}
moreover
\begin{equation}\label{eq:WpLp}
	W_p(Z^1_\sharp\P,Z^2_\sharp\P)\leq \|Z^1-Z^2 \|_{ L^p(\Omega;\R^d)}, \qquad \forall\, Z^1,Z^2\in  L^p(\Omega;\R^d).
\end{equation}

\medskip 

The space 
$\PP_p(\R^d)$ endowed with the $p$-Wasserstein metric $W_p$ is a complete and separable metric space.

\medskip

The existence of a minimizer in \eqref{eq:wasserstein} can be proved
by the direct method in Calculus of Variations.
When the measure $\mu_1$ is absolutely continuous with respect to Lebesgue measure $\cL^d$ on $\R^d$,
 the minimizer $\gamma$ is unique and it is concentrated on the graph
of a map, $\gamma=(i_{\R^d},T)_\sharp\mu_1$, where $i_{\R^d}$ is the identity map of $\R^d$ and $T$ 
is a minimizer in the Monge transport problem
\begin{equation}\label{eq:Monge}
\inf\left\{\int_{\R^d}|x-S(x)|^p\,\d\mu_1(x)\,:\,S_\sharp\mu_1=\mu_2\right\}.
\end{equation}
The Wasserstein distance has the following characterization, known as Benamou-Brenier formula:
\begin{equation}\label{B&B}
	W_p^p(\mu_0,\mu_1)=\min \left\{ \int_0^1\int_{\R^d} |v_t(x)|^p \,\d\mu_t(x)\,\d t : (\mu,v) \in \operatorname{CE}, \, \mu_{t=0}=\mu_0,  \mu_{t=1}=\mu_1\right\},
\end{equation}
where 
\begin{equation*}
\begin{aligned}
\operatorname{CE}:= \Big\{&(\mu,v): \mu \in C([0,T];\PP_p(\R^d)), v \in L^p([0,1] \times \R^d; \mu_t \otimes \d t)\\
 &\text{ such that  } \de_t\mu_t+\div(v_t\mu_t)= 0 \text{ in the sense of distributions } \Big\}.
\end{aligned}
\end{equation*}
Notice that the minimizers are the constant speed geodesics joining $\mu_0$ to $\mu_1$, i.e. $\{\sigma_t\}_{t\in[0,1]}$ such that
 $\sigma_0=\mu_0$,  $\sigma_1=\mu_1$ and $W_p(\sigma_t,\sigma_s)=|t-s|W_p(\mu_0,\mu_1)$ for any $t,s\in[0,1]$.

\medskip 

We recall the following definition as in \cite[Definition~2.2]{fornasier2018mean}.
\begin{definition}\label{def:admphi}
We say that $\psi:[0,+\infty)\to[0,+\infty)$ is an \emph{admissible} function if $\psi(0)=0$, 
$\psi$ is strictly convex and of class $C^1$ with $\psi'(0)=0$, 
superlinear at $+\infty$, i.e., $\lim_{r\to+\infty}\dfrac{\psi(r)}{r}=+\infty$, 
and doubling, i.e., there exists $A>0$ such that
\[\psi(2r)\le A (1+\psi(r))\quad \textrm{for any }r\in[0,+\infty).\]
\end{definition}

\medskip

We observe that an admissible function $\psi$ satisfies
\begin{equation}\label{doubling}
	r\psi'(r)\leq A(1+\psi(r)),\qquad \forall\,r\in [0,+\infty).
\end{equation}

\medskip
 
The following result provides equivalent conditions for the convergence in the space $\PP_p(\R^d)$ and the 
characterization of compactness.

\begin{proposition}\label{prop:wassconv}
Let $\{\mu_n\}_{n\in\N}\subseteq\mathscr P_p(\R^d)$ and $\mu\in\mathscr P_p(\R^d)$, the following assertions are equivalent:
\begin{enumerate}
\item $\displaystyle\lim_{n\to\infty}W_p(\mu_n,\mu)=0$;
\item \label{Nmoment} $\mu_n$ weakly converges to $\mu$ and $\m_p(\mu_n)\to \m_p(\mu)$ as $n\to+\infty$;
\item $\displaystyle\lim_{n\to+\infty}\int_{\R^d} \varphi(x)\,\d\mu_n(x)=\int_{\R^d}\varphi(x)\,\d\mu(x)$,\\
for every continuous function $\varphi:\R^d\to\R$ s.t. $|\varphi(x)|\leq C(1+|x|^p)$ for any $x\in\R^d$;
\item  \label{Npsi} $\mu_n$ weakly converges to $\mu$ and there exists $\psi:[0,+\infty)\to[0,+\infty)$ admissible, 
according to Definition \ref{def:admphi},
such that
\begin{equation}\label{eq:unifmompsi}
\sup_{n \in \N}\int_{\R^d} \psi(|x|^p)\,\d\mu_n(x) < +\infty.
\end{equation}
\end{enumerate}
Moreover, a family $\KK\subset \PP_p(\R^d)$ is relatively compact if and only if
there exists an admissible function $\psi:[0,+\infty)\to[0,+\infty)$ such that
\begin{equation}\label{eq:unifmompsiF}
\sup_{\mu \in \KK}\int_{\R^d} \psi(|x|^p)\,\d\mu(x) < +\infty.
\end{equation}

\end{proposition}

The proof can be carried on using \cite[Lemma 5.1.7, Proposition 7.1.5]{ambrosio2008gradient}.
Concerning the implication \eqref{Nmoment} to \eqref{Npsi}, it follows by De la Vall\'ee Poussin and  Dunford-Pettis theorems
together with \cite[Lemma~2.3]{fornasier2018mean} for the admissibility property.

\medskip

The following representation result for the (absolutely continuous) solutions of the continuity equation will play a key role in the sequel
(see \cite[Theorem 8.2.1]{ambrosio2008gradient}).
We denote by $\Gamma_T=\rmC([0,T];\R^d)$ the Banach space of the continuous functions, 
endowed with the $\sup$ norm. We denote by $e_t:\Gamma_T\to\R^d$  the evaluation map at time $t\in[0,T]$ 
 defined by $e_t(\gamma):=\gamma(t)$. 
 We say that $\eeta \in \PP(\Gamma_T)$ is concentrated on a set $B$ if $\eeta(\Gamma_T \setminus B) = 0$. 

\begin{theorem}[Superposition principle]\label{thm:sup_princ} Let $p\ge1$.
Let $\mu=\{\mu_t\}_{t\in [0,T]}\in C([0,T];\PP_p(\R^d))$ be a distributional solution
of the continuity equation $\partial_t \mu_t+\mathrm{div}(v_t\mu_t)=0$
for a Borel vector field $v:[0,T]\times\R^d\to\R^d$ satisfying
\begin{equation}\label{Summabilityv}
\int_0^T\int_{\R^d}|v_t(x)|^p\,\d\mu_t(x)\,\d t<+\infty.
\end{equation}
Then there exists a probability measure $\eeta\in \PP(\Gamma_T)$
such that
\begin{enumerate}
\item[(i)] $\mu_t=(e_t)_\sharp\eeta$ for every $t\in [0,T]$;
\item[(ii)] $\eeta$ is concentrated on the set of curves $ \gamma \in \AC^p([0,T];\R^d)$ satisfying
$$\dot\gamma(t)=v_t(\gamma(t)),\qquad \text{for $\cL_T$-a.e. }t\in[0,T].$$
\end{enumerate}
Conversely, given $\eeta\in\PP(\Gamma_T)$ satisfying item (ii) {and \eqref{Summabilityv} with} $\mu_t:= (e_t)_\sharp\eeta$ for every $t \in [0,T]$,  
then  $(\mu,v)$ is a distributional solution of $\partial_t\mu_t+\mathrm{div}(v_t\mu_t)=0$ .
\end{theorem}
In Theorem \ref{lem:EN} in Appendix \ref{app:SP}, we prove a version of the superposition principle in the discrete setting.

\subsection{Disintegration and Young measures}

Let $\S$ and $S$ be Polish spaces. We say that a map $x\in S \mapsto \mu_x\in\PP(\S)$  is a Borel map 
if $x\mapsto \mu_x(A)$ is a Borel map for any open set $A\subset\S$. 

If $x\in S \mapsto \mu_x\in\PP(\S)$ is a Borel map and $\lambda\in\PP(S)$ we define the measure
$\mu_x\otimes\lambda \in \PP(\S)$ by 
$$(\mu_x\otimes\lambda) (A):= \int_S \mu_x(A)\,\d\lambda(x)$$ for any Borel set $A\subset\S$.
Equivalently 
$$\int_\S \varphi(z)\,\d(\mu_x\otimes\lambda)(z):= \int_S \int_\S \varphi(z)\,\d\mu_x(z)\,\d\lambda(x)$$
for any bounded Borel function $\varphi:\S\to\R$.

We state the following disintegration result (see for instance \cite[Section 5.3]{ambrosio2008gradient}).
\begin{theorem}[Disintegration]\label{thm:disint}
Let $\S$ and $S$ be Polish spaces.
Let $\mu\in\PP(\S)$ and $r:\S\to S$ a Borel map.
Then there exists a Borel measurable family of probability
measures $\{\mu_x\}_{x\in S}\subset \PP(\S)$,
uniquely defined for $r_\sharp\mu$-a.e. $x\in S$, such that 
$\mu_x(\S\setminus r^{-1}(x))=0$ for $r_\sharp\mu$-a.e. $x\in S$, and
$\mu=\mu_x\otimes(r_\sharp \mu)$.
In particular, for any bounded Borel map $\varphi:\S\to \R$ we have
\begin{equation}\label{disintegration}
\int_{\S}\varphi(z)\,d\mu(z)=\int_S \int_{r^{-1}(x)}\varphi(z)\,d\mu_x(z)\,d(r_\sharp \mu)(x).
\end{equation}
\end{theorem}
\begin{remark}\label{rem:disint_proj}
A typical case is given by $\S=S\times Y$, where $Y$ is a Polish space, and $r=\pi^1$. 
Since $(\pi^1)^{-1}(x)=\{x\}\times Y$ for all $x\in S$, we identify 
each measure $\mu_x\in\PP(S\times Y)$, which is concentrated in $\{x\}\times Y$, with a measure $\mu_x\in\PP(Y)$.
With this identification, the formula \eqref{disintegration} takes the form
\begin{equation}\label{disintegration2}
\int_{S\times Y}\varphi(x,y)\,d\mu(x,y)=\int_S \int_{Y}\varphi(x,y)\,d\mu_x(y)\,d(r_\sharp \mu)(x).
\end{equation}
\end{remark}

\medskip

Let  $\T$ and $S$ be Polish spaces, $\lambda\in\PP(\T)$ and $E$ be a Banach space. 
We say that $h: \T \times S \to E$ is a Carath\'eodory function if
$$ \text{for $\lambda$-a.e. }t\in \T, \qquad x\mapsto h(t,x) \text{ is continuous}, $$
$$ \forall\,  x\in S, \qquad t\mapsto h(t,x) \text{ is $\lambda$-measurable}. $$

\medskip

Let us now recall the definition of \emph{Young measure} (see \cite{bernard2008young,valadier2004young}) and a density
 result which will turn out to be a crucial tool in our treatment.
\begin{definition}\label{def:Youngconv}
Let $\T$ and $S$ be Polish spaces and $\lambda\in\PP(\T)$. 
We say that $\nu \in \PP(\T \times S)$ is a \textit{Young measure} on $\T\times S$ if $\pi^1_\sharp \nu = \lambda$. 
Furthermore given $\nu^n, \nu \in \PP(\T\times S)$ Young measures, 
we say that $\nu^n \xrightarrow{\mathcal{Y}} \nu$ as $n \to +\infty$ if
\[ \lim_{n \to + \infty} \int_{\T \times S} h(\tau, u) \,\d \nu^n(\tau,u) = \int_{\T \times S} h(\tau, u) \,\d \nu(\tau,u), \]
for any $h: \T \times S \to \R$ Carath\'eodory and bounded.
\end{definition}

\begin{remark}\label{rem:Young-weak}
Let $\T$, $S$ be Polish spaces, $\lambda\in\PP(\T)$ and $\nu^n,\nu$ Young measures on $\T\times S$. 
Then $\nu^n \xrightarrow{\mathcal{Y}} \nu$ in the sense of Definition \ref{def:Youngconv} if and only if $\nu^n\to\nu$ weakly. One implication follows immediately from the definitions, while the other comes from \cite[Theorem~7]{valadier1990young} (see also \cite{valadier2004young}).

We also recall that weak convergence in $\PP([0,T] \times S)$ is induced by a distance $\delta$.
When $S$ is compact, we can choose as $\delta$ any Wasserstein distance on $\PP([0,T] \times S)$. 
\end{remark}

To any Borel map $u:\T\to S$ we can associate the Young measure
$\nu:=(i_\T,u)_\sharp\lambda$, which is concentrated on the graph of $u$.
In this case, $\nu$ can be written as $\nu=\delta_{u(\tau)}\otimes\lambda$ and,
using the disintegration Theorem \ref{thm:disint}, we have that $\nu_\tau=\delta_{u(\tau)}$ for $\lambda$-a.e. $\tau\in\T$.
Given a Young measure $\nu$, in general the disintegration  $\nu_\tau$ of $\nu$ w.r.t. $\lambda$
 is not of the form $\delta_{u(\tau)}$ on a set of $\lambda$ positive measure, 
for some $u:\T\to S$. 
The following classical Lemma states that the Young measures induced by maps are ``dense'', in the set of Young measures, 
provided $\lambda$ is non atomic. 
We say that a measure $\lambda\in\PP(\T)$ is non atomic if $\lambda(\{\tau\})=0$ for any $\tau\in\T$. 

\begin{lemma}[{see \cite[Theorem~2.2.3]{valadier2004young}}]\label{lemma:young}
Let $\T$ and $S$ be Polish spaces and $\lambda\in\PP(\T)$ non atomic. 
If $\nu \in \PP(\T \times S)$ is a Young measure, then there exists a sequence of Borel maps $u^n: \T \to S$ such that 
\[ \nu^n:= (i_\T,u^n)_\sharp \lambda \xrightarrow{\;\;\mathcal{Y}\;\;\;} \nu=\nu_\tau\otimes\lambda.  \]
Precisely,
\begin{equation}
\lim_{n \to +\infty} \int_{\T} h(\tau, u^n(\tau)) \,\d \lambda(\tau) = \int_{\T} \int_S h (\tau,u) \,\d \nu_\tau(u)\,\d\lambda(\tau),
\end{equation}
for every $h: \T \times S \to \R $ Carath\'eodory and bounded.
\end{lemma}

\section{Structural assumptions for the dynamics of the optimal control
problems}\label{sec:assumptions}

In this section we collect our main structural assumptions on the
system $\S=(U,f,\cC,\cC_T)$ characterizing the dynamics and the
cost of the control problems under study, where $U$ is the space of controls, $f$
is the vector field driving the particles motion, $\cC$ and
$\cC_T$ are the running and terminal cost functionals. 

We fix $p\in[1,+\infty)$ and denote by $\sfd p$ the following metric on $\R^d\times\mathscr P_p(\R^d)$:
\[\sfd p((x,\mu),(y,\nu)):=\left(|x-y|^p+W_p^p(\mu,\nu)\right)^{1/p}.\]

\begin{assumption}[Basic Assumption]\label{BA}
We assume that the system $\S:=(U,f,\mathcal C,\mathcal C_T)$ satisfies:
\begin{enumerate}[label=(A.\arabic*)]
\item \label{hp:U} $U$ is a compact metrizable space;
\item \label{itemf:main}  
$f:\R^d\times U\times\mathscr P_p(\R^d)\to\R^d $ is continuous and Lipschitz continuous w.r.t. the metric $\sfd p$, uniformly in $u\in U$. Precisely, there exists $L>0$ such that 
\begin{equation}\label{eq:Lipf}
 |f(x,u,\mu)-f(y,u,\nu)|\le L\, \sfd p((x,\mu),(y,\nu)) \ ,
 \end{equation}
for every $u\in U$ and $(x,\mu),(y,\nu)\in\R^d\times\mathscr P_p(\R^d)$.
\item \label{itemg:main}
$\mathcal C:\R^d\times U\times\mathscr P_p(\R^d)\to[0,+\infty)$ 
and $\mathcal C_T:\R^d\times\mathscr P_p(\R^d)\to[0,+\infty)$ are continuous functions such that 
\begin{equation}\label{eq:growthC}
\begin{aligned}
& \mathcal C(x,u,\mu) \le D\left(1+|x|^p+\mathrm{m}_p^p(\mu)\right) \qquad \forall\, (x,u,\mu)\in\R^d\times U\times\mathscr P_p(\R^d)\\
& \mathcal C_T(x,\mu) \le D\left(1+|x|^p+\mathrm{m}_p^p(\mu)\right)  \qquad \forall\, (x,\mu)\in\R^d\times \mathscr P_p(\R^d),
\end{aligned}
\end{equation}
for some $D>0$. 
\end{enumerate}
\end{assumption}

\begin{remark}
From Assumption \ref{BA} it holds
\begin{equation}\label{f:growth}
|f(x,u,\mu)|\le C\left(1+|x|+\mathrm{m}_p(\mu)\right), \qquad \forall\,(x,u,\mu)\in\R^d\times U\times\mathscr P_p(\R^d),
\end{equation}
for some $C>0$.
Indeed, it is sufficient to choose $(y,\nu)=(0,\delta_0)$ in \eqref{eq:Lipf} and observe that 
$f(0,u,\delta_0)$ is bounded and $W_p(\mu,\delta_0)=\mathrm{m}_p(\mu)$.
\end{remark}

Concerning item \ref{hp:U} of Assumption \ref{BA}, let us recall the following result.
\begin{proposition}
If $U$ is compact metrizable space then, for every distance $d_{U}$ inducing the original topology of $U$, there exists a separable Banach space $V$ and an isometry $j: U \to V$. 
In particular, the image $j(U)$ is a compact subset of $V$. 
\end{proposition}
\begin{proof}
Fix a point $u_0\in U$ and consider the Banach space $B:=\{F\in \mathrm{Lip}(U): F(u_0)=0\}$ endowed with the norm
$$ \|F\|_B:=\sup_{u,v\in U, u\not=v}\frac{|F(u)-F(v)|}{d_{U}(u,v)}.$$
Denoting by $B'$ the dual space of $B$, we define the map $j: U \to B'$ by $\langle j(u),F\rangle_{B',B}:=F(u)$. 
By the definition of dual norm, it is immediate to check that
$$\|j(u)-j(v) \|_{B'}\leq d_{ U}(u,v), \quad \forall \, u,v \in U$$
On the other hand, evaluating $\langle j(u)-j(v), F\rangle_{B',B}$ with $F(z):=d_{U}(z,u)-d_{ U}(u,u_0)$,
we obtain that
 $$\|j(u)-j(v) \|_{B'}= d_{U}(u,v),$$
so that $j$ is an isometry from $U$  to $j(U)\subset B'$.
 We eventually set $V:=\overline{\mathrm{span}(j(U))}^{B'}$, which is a separable Banach space since $U$, and therefore $j( U)$, 
 is separable.
\end{proof}

{When specified, we will assume the following further hypothesis.}

\begin{assumption}[Convexity Assumption]\label{CA}
We say that $\S =(U,f,\cC,\cC_T)$ satisfies  the \emph{convexity assumption} if $\S$ satisfies Assumption \ref{BA} and
\begin{enumerate}[label=(C.\arabic*)]
\item $U$ is a compact convex subset of a separable Banach space $V$;
\item for any $x\in\R^d$ and $\mu\in\PP_p(\R^d)$, the map $u\mapsto f(x,u,\mu)$ satisfies the \emph{affinity condition}: 
\begin{equation*}
f(x,\alpha u + (1-\alpha) v,\mu) = \alpha f(x,u,\mu) + (1-\alpha) f(x,v,\mu), \quad \forall\,u,v \in U,\; \forall\, \alpha \in [0,1];
\end{equation*}
\item for any $x\in\R^d$ and $\mu\in\PP_p(\R^d)$ the map $u \mapsto\mathcal C(x,u,\mu)$ is convex:
\begin{equation*}
\cC(x,\alpha u + (1-\alpha) v,\mu) \leq \alpha \cC(x,u,\mu) + (1-\alpha) \cC(x,v,\mu), \quad \forall\,u,v \in U,\; \forall\, \alpha \in [0,1].
\end{equation*}
 \end{enumerate}
\end{assumption}

\subsection{The relaxed setting}\label{sub:Relax}
For later use, we define a so-called \emph{relaxation/lifting} of $\S$ as follows.

\begin{definition}\label{def:relax_setting}
Given  the system $\S = (U,f,\cC,\cC_T)$ satisfying Assumption \ref{BA}, 
we define $\S'=(\UU,\FF,\CC,\CC_T)$ as follows:
\begin{enumerate}[label=(\roman*)]
\item $\UU:=\PP(U)$;
\item $\FF:\R^d\times\UU\times\PP_p(\R^d)\to \R^d$ with 
$$\FF(x,\sigma,\mu):=\int_U f(x,u,\mu)\,\d\sigma(u);$$
\item $\CC:\R^d\times\UU\times\PP_p(\R^d)\to [0,+\infty)$ with 
$$\CC(x,\sigma,\mu):=\int_U \cC(x,u,\mu)\,\d\sigma(u).$$
\item $\CC_T := \cC_T$.
\end{enumerate}
\end{definition}

\begin{proposition}\label{prop:ConvexRL}
If $\S = (U,f,\cC,\cC_T)$ satisfies Assumption \ref{BA},
then {its relaxation}  $\S' = (\UU,\FF,\CC,\cC_T)${, given in Definition \ref{def:relax_setting},} satisfies the Convexity Assumption \ref{CA}.
Moreover, defining $D_U:=\{\delta_u: u\in U\}\subset\UU$, the maps $\FF$ and $\CC$ restricted to  
$\R^d\times D_U\times\PP_p(\R^d)$ coincide with $f$ and $\cC$ respectively. 
\end{proposition}
\begin{proof}
The space $\UU:=\PP(U)$ can be identified with a subset of  the dual space $B'$, 
where $B$ is the Banach space $B:=\{F\in \mathrm{Lip}(U): F(u_0)=0\text{ for some  }u_0\in U\}$.
The identification is given associating to  $\sigma \in \PP(U)$ the continuous linear functional $F \mapsto \int_U F(u)\,\d\sigma(u)$. 
With this identification, the norm in $\PP(U)$ is given by $$\|\sigma\|=\sup_{F\in B, \|F\|_B\leq 1} \int_U F(u)\,\d\sigma(u).$$
By the Kantorovich-Rubinstein Theorem (see e.g. \cite[Theorem 1.14]{villani2003topics}) it holds that $ \|\sigma\|=W_1(\sigma,\delta_{u_0})$ and
 $ \|\sigma^1-\sigma^2\|=W_1(\sigma^1,\sigma^2)$. 
Hence, the topology on $\PP(U)$ induced by $B'$ coincides with the topology induced by the Wasserstein distance $W_1$. 
Since $U$ is compact, this coincides with the topology induced by the weak convergence. 
By Prokhorov Theorem, $\PP(U)$ is compact. 
Finally, $(\PP(U), \| \cdot \|)$ is a separable Banach space thanks to the separability of the (complete) metric space $(\PP(U),W_1)$.  
The convexity of $\UU$, the affinity of $\FF$ and the convexity of $\CC$ with respect to $\sigma$ easily follows from their definitions.
\end{proof}

\section{Lagrangian optimal control problem}
\label{sec_lagrangian}

In this section we deal with a (finite-horizon) optimal control
problem in Lagrangian formulation.
It relies on a system $\S=(U,f,\cC,\cC_T)$ satisfying Assumptions
\ref{BA} and on a probability space $(\Omega,\frB,\P)$, whose elements
act as parameters of the particles.
We also fix a final time horizon $T>0$ and we denote with $\Leb_{[0,T]}$  the  $\sigma$-algebra of Lebesgue measurable sets on $[0,T]$ and with $\cL_T$ the normalized Lebesgue measure restricted to $[0,T]$.
Recall that $\BM([0,T] \times \Omega;U)$ denotes the set of measurable functions with respect to the product $\sigma$-algebra $\Leb_{[0,T]} \otimes \frB$.

\begin{definition}[Lagrangian optimal control problem $(\pL)$]\label{def:L}
Let $\S:=(U,f,\mathcal C,\mathcal C_T)$ satisfy Assumption \ref{BA}
and let $(\Omega,\frB,\P)$ be a probability space.

Given $X_0\in L^p(\Omega;\R^d)$, we say that
$(X,u)\in\cA_{\pL}(X_0)$ if
\begin{itemize}
\item [(i)]
$u\in \BM([0,T]\times\Omega;U)$;
\item  [(ii)] $X\in L^p(\Omega;\AC^p([0,T];\R^d))$ and for 
$\P$-a.e. $\omega\in\Omega$, $X(\omega)$ is a 
solution of the following Cauchy problem
\begin{equation}\label{eq:systemL}
\begin{cases}
\dot X_t(\omega)=f(X_t(\omega),u_t(\omega),(X_t)_\sharp\P), &\textrm{for $\cL_T$-a.e. }t\in (0,T)\\ 
X_{|t=0}(\omega)=X_0(\omega),&
\end{cases}
\end{equation}
where $X_t: \Omega \to \R^d$ is defined by $X_t(\omega):= X(t,\omega)$ for $\P$-a.e. $\omega \in \Omega$. 
\end{itemize}
We refer to $(X,u)\in\mathcal A_{\pL}(X_0)$ as to an \emph{admissible pair}, with $X$ a \emph{trajectory} and $u$ a \emph{control}.\\
We define the \emph{cost functional}  
$J_{\pL}:L^p(\Omega;C([0,T];\R^d))\times \BM([0,T]\times\Omega;U)\to[0,+\infty)$,  by
\[J_{\pL}(X,u):= \int_\Omega \int_0^T\mathcal C(X_t(\omega),u_t(\omega),(X_t)_\sharp \P)\,\d t \, \d\P(\omega) +\int_\Omega \mathcal C_T(X_T(\omega),(X_T)_\sharp \P)\,\d\P(\omega),\]
and the \emph{value function} $V_{\pL}:L^p(\Omega;\R^d)\to[0,+\infty)$ by
\begin{equation}
V_{\pL}(X_0):=\inf\left\{J_{\pL}(X,u)\,:\,(X,u)\in\mathcal A_{\pL}(X_0)\right\}.
\end{equation}
\end{definition}

\smallskip
 In the following, $\pL(\Omega, \frB,\P;\S)$ denotes the Lagrangian problem given in Definition \ref{def:L}.
We will frequently shorten the notation to $\pL(\Omega, \frB, \P)$  when the system $\S$ is clear from the context.  
 
\begin{remark}\label{rem:lagrangian}
Observe that, thanks to condition \eqref{eq:growthC}, the functional $J_{\pL}$ is finite.
Moreover, from Proposition \ref{prop:existL} below it follows that $\cA_{\pL}(X_0)\not=\emptyset$, for any $X_0\in L^p(\Omega;\R^d)$, and so the value function
$V_{\pL}$ is well defined.
We point out that existence of minimizers for the Lagrangian problem is not guaranteed in general, even under the Convexity Assumption \ref{CA}.
This will be further discussed in Section \ref{sec:counterexample}.
\end{remark}

\begin{remark}\label{re:equiv}
In view of Proposition \ref{prop:equivSpaces} in Appendix A, we will frequently identify $X \in L^p(\Omega;\AC^p([0,T];\R^d))$ and $X \in \AC^p([0,T];L^p(\Omega;\R^d))$, depending on the convenience.
\end{remark}

Let us introduce a suitable equivalence relation among Lagrangian problems when the parametrization space is varying.

\begin{definition}[Equivalence of Lagrangian problems]\label{def:equiv_L}
Let $\S:=(U,f,\mathcal C,\mathcal C_T)$ satisfy Assumption \ref{BA}.
Let $(\Omega_1,\frB_1,\P_1)$ and $(\Omega_2,\frB_2,\P_2)$ be probability spaces. 
We say that $\pL_1 := \pL(\Omega_1,\frB_1,\P_1; \S)$ and $\pL_2:=\pL(\Omega_2,\frB_2,\P_2; \S)$ are \emph{equivalent} (and we write $\pL_1 \sim \pL_2$) if 
\begin{itemize}
\item[(i)] for every $X_0^1\in L^p(\Omega_1;\R^d)$ and every $(X^1,u^1) \in \cA_{\pL_1}(X_0^1)$ there exist $X_0^2\in L^p(\Omega_2;\R^d)$ and $(X^2,u^2) \in \cA_{\pL_2}(X_0^2)$ such that 
\[ J_{\pL_1}(X^1,u^1) = J_{\pL_2}(X^2,u^2), \qquad \quad  V_{\pL_1}(X^1_0) = V_{\pL_2}(X^2_0);\] 
\item[(ii)] for every $X_0^2\in L^p(\Omega_2;\R^d)$ and every $(X^2,u^2) \in \cA_{\pL_2}(X_0^2)$ there exist $X_0^1\in L^p(\Omega_1;\R^d)$ and $(X^1,u^1) \in \cA_{\pL_1}(X_0^1)$ such that 
\[ J_{\pL_2}(X^2,u^2) = J_{\pL_1}(X^1,u^1), \qquad \quad  V_{\pL_2}(X^2_0) = V_{\pL_1}(X^1_0).\] 
\end{itemize}
\end{definition}

\begin{remark}

The relation $\sim$ of Definition \ref{def:equiv_L} is an equivalence relation on the set of Lagrangian problems $\lbrace \pL(\Omega, \frB,\P) :  (\Omega, \frB,\P) \text{ probability space} \rbrace$.
\end{remark}

\begin{proposition}\label{p:equivalence}
Let $\S:=(U,f,\mathcal C,\mathcal C_T)$ satisfy Assumption \ref{BA}.
Let $(\Omega_1,\frB_1,\P_1)$ and $(\Omega_2,\frB_2,\P_2)$ be probability spaces. 
Suppose there exist measurable maps $\psi: \Omega_1 \to \Omega_2$ and $\phi: \Omega_2 \to \Omega_1$ such that 
$\psi_\sharp\P_1 = \P_2$,  $\phi_\sharp\P_2 = \P_1$ and
\begin{align}
\label{eq:strana1} &\forall\, X_0^1 \in L^p(\Omega_1;\R^d) \text{ it holds } X_0^1 = X_0^1 \circ \phi \circ \psi \, ; \\
\label{eq:strana2} &\forall\, X_0^2 \in L^p(\Omega_2;\R^d) \text{ it holds } X_0^2 = X_0^2 \circ \psi \circ \phi.
\end{align}
Then $\pL(\Omega_1,\frB_1,\P_1; \S) \sim \pL(\Omega_2,\frB_2,\P_2; \S)$.
\end{proposition}

\begin{proof}
For every $(X^1,u^1) \in \cA_{\pL_1}(X^1_0)$, we define $X^2:= X^1\circ \phi$, and $u^2(t,\omega_2):= u^1(t, \phi(\omega_2))$, for every $(t, \omega_2) \in [0,T] \times \Omega_2$. 
Using that $\phi_\sharp \P_2 = \P_1$, it easily follows that $(X^2,u^2)\in\cA_{\pL_2}(X_0^1\circ\phi)$ and  $J_{\pL_1}(X^1,u^1) = J_{\pL_2}(X^2,u^2)$. 
Hence, for every $X_0^1 \in L^p(\Omega_1; \R^d)$, we have 
\begin{equation}\label{V1geV2}
V_{\pL_1}(X^1_0) \geq V_{\pL_2}(X^1_0 \circ \phi).
\end{equation}
Analogously, for every $(X^2,u^2) \in \cA_{\pL_2}(X^2_0)$, we define $X^1:= X^2\circ \psi$, and $u^1(t,\omega_1):= u^2(t, \psi(\omega_1))$, for every $(t, \omega_1) \in [0,T] \times \Omega_1$. So that, from $\psi_\sharp \P_1 = \P_2$ it holds $(X^1,u^1)\in\cA_{\pL_1}(X_0^2\circ\psi)$ and  $J_{\pL_1}(X^1,u^1) = J_{\pL_2}(X^2,u^2)$.
Moreover, for every $X_0^2 \in L^p(\Omega_2; \R^d)$ we have 
\begin{equation}\label{V2geV1}
V_{\pL_2}(X^2_0) \geq V_{\pL_1}(X^2_0 \circ \psi).
\end{equation}
The combination of \eqref{V1geV2} and \eqref{V2geV1} gives
\[ V_{\pL_2}(X^2_0) \geq V_{\pL_1}(X^2_0 \circ \psi) \geq V_{\pL_2}(X^2_0 \circ \psi \circ \phi)\] 
hence, using \eqref{eq:strana1} we have
\begin{equation}\label{eq:prima}
V_{\pL_2}(X^2_0) = V_{\pL_1}(X^2_0 \circ \psi).
\end{equation}
Thanks to \eqref{eq:strana2} and \eqref{eq:prima} we finally get
\[V_{\pL_1}(X^1_0) = V_{\pL_1}(X^1_0 \circ \phi \circ \psi) = V_{\pL_2}(X^1_0 \circ \phi). \]
\end{proof}

\begin{remark}
\
\begin{enumerate}
\item Notice that the assumptions of Proposition \ref{p:equivalence} are satisfied if there exists a bijective function $\psi:\Omega_1\to\Omega_2$ such that $\psi$ and $\psi^{-1}$ are measurable and $\psi_\sharp\P_1=\P_2$. Indeed, it sufficies to choose $\phi=\psi^{-1}$.
\item Proposition \ref{p:equivalence} still holds when the maps $\psi$ and $\phi$ are defined up to sets of null measure, meaning that
\[\psi: \Omega_1\setminus\cN_1 \to \Omega_2\setminus\cN_2,\quad \phi: \Omega_2\setminus\cN_2 \to \Omega_1\setminus\cN_1\]
 for some $\cN_1\in \frB_1$ such that $\P_1(\cN_1)=0$ and  $\cN_2\in \frB_2$ such that $\P_2(\cN_2)=0$.
\end{enumerate}

\end{remark}

\subsection{Basic results}
Here we collect some properties of the Lagrangian problem. 
In particular, we show existence  and uniqueness of solutions, a priori estimates, compactness for the associated laws and we derive a stability result for trajectories and cost when initial data and control converge in a suitable sense.

\begin{proposition}[Existence and uniqueness]\label{prop:existL}
Let $\S:=(U,f,\mathcal C,\mathcal C_T)$ satisfy Assumption \ref{BA}
and $(\Omega,\frB,\P)$ be a probability space.
Let $X_0\in L^p(\Omega;\R^d)$ and $u\in \BM([0,T]\times\Omega;U)$ be given.
Then there exists a unique $X\in L^p(\Omega;\AC^p([0,T];\R^d))$ such that $(X,u)\in\mathcal A_{\pL}(X_0)$. 
Moreover, if $(X^i,u^i)\in\cA_\pL(X_0)$, $i=1,2$, and $u^1=u^2$ $\cL_T\otimes\P$-a.e., then $X^1=X^2$.
\end{proposition}
\begin{proof}
We define  
 $F_u: [0,T] \times L^p(\Omega;\R^d)\to L^p(\Omega;\R^d)$ by
\begin{equation}\label{eq:defFu}
F_u(t,Z)(\omega):=f(Z(\omega),u(t,\omega),Z_\sharp\P).
\end{equation}
We observe that the continuity of $f$ and the  measurability of $u$ imply that $F_u$ is a Carath\'eodory function.
Moreover, by \eqref{eq:Lipf} and \eqref{eq:WpLp}, $F_u$ satisfies condition \eqref{LipFE}.
Since $F_u(t,0)(\omega)=f(0,u_t(\omega),\delta_0)$, by continuity of $f$ and compactness of $U$ it follows that
$F_u$ satisfies \eqref{boundFE}.
Theorem \ref{th:ExistenceODE} with the choice $E=L^p(\Omega;\R^d)$ and $F=F_u$ yields the existence and uniqueness of a curve  $X \in\AC^p([0,T];L^p(\Omega;\R^d))$ solving
\begin{equation*}
	X_t=X_0+\int_0^t F_u(s,X_s)\,\d s, \qquad \forall\,t\in[0,T].
\end{equation*}
Thanks to Proposition \ref{prop:EquivLp} we finally get $X\in L^p(\Omega;\AC^p([0,T];\R^d))$ which is the unique solution of \eqref{eq:systemL}. 
The last assertion follows from the equality $F_{u^1}(t,Z)=F_{u^2}(t,Z)$ for $\cL_T$-a.e. $t\in[0,T]$ and for every $Z \in L^p(\Omega;\R^d)$.
\end{proof}

\begin{proposition}[A priori estimates]\label{prop:estimatesL}
Let $\S:=(U,f,\mathcal C,\mathcal C_T)$ satisfy Assumption \ref{BA}
and $(\Omega,\frB,\P)$ be a probability space.
Let $X_0\in L^p(\Omega;\R^d)$ and $(X,u)\in\mathcal A_{\pL}(X_0)$.
Then there exist $C$ and $C_T$ independent of $u$ and $X_0$ such that
\begin{equation}\label{boundXt}
	\sup_{t\in[0,T]}\|X_t\|_{L^p(\Omega;\R^d)}\le e^{2CT}\left(\|X_0\|_{L^p(\Omega;\R^d)}+CT\right), 
\end{equation}
\begin{equation}\label{LipXt}
	\|X_t-X_s\|_{L^p(\Omega;\R^d)}\le C_T\,|t-s|\,\left(1+\|X_0\|_{L^p(\Omega;\R^d)}\right)  \qquad \forall\, s,t\in[0,T],
\end{equation}
\begin{equation}\label{boundXtomega}
	\sup_{t\in[0,T]}|X_t(\omega)| \le e^{CT}\left(|X_0(\omega)|+ C_T(1+ \|X_0\|_{L^p(\Omega;\R^d)})\right),  \quad
	\text{for }\P\text{-a.e. }\omega\in\Omega.
\end{equation}
\end{proposition}
\begin{proof}
The estimates \eqref{boundXt} and \eqref{LipXt} follows from \eqref{boundZt} and \eqref{LipZt} 
for $F=F_u$ defined in \eqref{eq:defFu} and $E=L^p(\Omega;\R^d)$ .

In order to prove \eqref{boundXtomega} we write \eqref{eq:systemL} in integral form
\begin{equation}\label{systemLintp}
X_t(\omega)= X_0(\omega)+\int_0^t f(X_s(\omega),u(s,\omega),(X_s)_\sharp\P)\,\d s, \quad \forall\,t\in[0,T] \quad
	\text{for }\P\text{-a.e. }\omega\in\Omega.
\end{equation}
Then by \eqref{f:growth} we have
\begin{align*}
|X_t(\omega)| &=
\left |X_0(\omega) + \int_0^t f(X_s(\omega),u(s,\omega),(X_s)_\sharp\P)\,\d s \right |\\
&\le |X_0(\omega)| + \int_0^t | f(X_s(\omega),u(s,\omega),(X_s)_\sharp\P) | \,\d s\\
&\le |X_0(\omega)| + \int_0^t C\left( 1+|X_s(\omega)|+ \|X_s\|_{L^p(\Omega;\R^d)}\right) \,\d s.
\end{align*}
Using \eqref{boundXt} and Gronwall inequality we obtain \eqref{boundXtomega}.

\end{proof}

In the following Lemma, we derive a compactness result for the laws of Lagrangian trajectories, when the initial data belong to 
a compact subset of $L^p(\Omega;\R^d)$. 

\begin{lemma}\label{lemma:cpt}
Let $\S:=(U,f,\mathcal C,\mathcal C_T)$ satisfy Assumption \ref{BA}
and $(\Omega,\frB,\P)$ be a probability space.
Let $K \subseteq L^p(\Omega;\R^d)$ compact. Then the set
\begin{equation}\label{eq:defcompactKmu}
\cK_K := \left\{ \mu \in \operatorname{AC}([0,T]; \PP_p(\R^d)) : \mu_t = (X_t)_\sharp \P, \; (X,u) \in \mathcal A_{\pL}(X_0), \; X_0 \in K \right\}
\end{equation}
is relatively compact in $C([0,T]; \PP_p(\R^d))$.
\end{lemma}
\begin{proof}
Let $\{\mu^n\}_{n\in\N}\subset\cK_K$ be a sequence. By definition, there exist $(X^n,u^n) \in \mathcal A_{\pL}(X^n_0)$,  $X^n_0 \in K$ such that
$\mu_t = (X_t)_\sharp \P$ for all $t\in[0,T]$.
Since $\sup_{n\in\N}\|X^n_0\|_{L^p(\Omega;\R^d)}<+\infty$, by the estimate \eqref{boundXtomega} there exits a constant
$C>0$ such that
\begin{equation}\label{eqK}
	|X^n_t(\omega)|^p \le C\left(1+|X^n_0(\omega)|^p\right),  \quad \forall\, n\in\N, \; \forall\, t\in[0,T], 
	\text{ for }\P\text{-a.e. }\omega\in\Omega.
\end{equation}
Since $K$ is compact in $L^p(\Omega;\R^d)$, there exists an admissible $\psi:[0,+\infty)\to[0,+\infty)$, 
according to Definition \ref{def:admphi}, such that
\[ \sup_{n\in\N}\int_\Omega \psi(|X_0^n(\omega)|^p)\,\d\P(\omega)<+\infty.\]
By the doubling and monotonicity property of $\psi$ and \eqref{eqK} we have
\begin{equation*}\label{}
	\psi(|X^n_t(\omega)|^p) \le C\left(1+\psi(|X^n_0(\omega)|^p)\right),  \quad \forall\, n\in\N, \; \forall\, t\in[0,T], 
	\text{ for }\P\text{-a.e. }\omega\in\Omega.
\end{equation*}
and then
\[ \sup_{t\in[0,T],n\in\N}\int_\Omega \psi(|X_t^n(\omega)|^p)\,\d\P(\omega)<+\infty,\]
that can be rewritten as
\[ \sup_{t\in[0,T],n\in\N}\int_{\R^d} \psi(|x|^p)\,\d\mu_t^n(x)<+\infty.\]
By Proposition \ref{prop:wassconv} there exists a compact $\KK\subset  \PP_p(\R^d)$ such that $\mu^n_t\in\KK$ for any $t\in[0,T]$ and $n\in\N$. 

Moreover, by \eqref{LipXt} and the boundedness of   $\|X^n_0\|_{L^p(\Omega;\R^d)}$, there exists $C>0$ such that 
\[ W_p(\mu^n_t,\mu^n_s)\le \|X^n_t-X^n_s\|_{L^p(\Omega;\R^d)} \le C |t-s|, \quad \forall\, s,t \in[0,T], \; \forall\,n\in\N.  \]
We can thus apply Ascoli-Arzel\`a theorem in $C([0,T]; \PP_p(\R^d))$ to conclude.
\end{proof}

\medskip

 We conclude the subsection proving a  first stability result for the Lagrangian problem. 
\begin{proposition}[Stability for \pL]\label{prop:costRLntoRL}
Let $\S:=(U,f,\mathcal C,\mathcal C_T)$ satisfy Assumption \ref{BA}
and $(\Omega,\frB,\P)$ be a probability space.
Let $X_0\in L^p(\Omega;\R^d)$ and $(X,u)\in\mathcal A_{\pL}(X_0)$.
Let $X_0^n \in L^p(\Omega;\R^d)$ be a sequence such that $\|X^n_0-X_0\|_{L^p(\Omega;\R^d)}\to 0$, as $n \to +\infty$.
If $(X^n,u^n)\in\cA_{\pL}(X^n_0)$, for any $n\in\N$,  and 
$u^n\to u$  in $\cL_T\otimes\P$-measure as $n\to+\infty$, then 
\begin{equation}\label{convX}
\sup_{t\in[0,T]}\|X^n_t-X_t\|_{L^p(\Omega;\R^d)}\to 0,\quad \textrm{as }n\to+\infty,
\end{equation}
and
\begin{equation}\label{convJ}
J_{\pL}(X^n,u^n)\to J_{\pL}(X,u),\quad \textrm{as }n\to+\infty.
\end{equation}
\end{proposition}

\begin{proof}
In order to prove \eqref{convX} we apply Proposition \ref{prop:StabilityODE} with the choice
$E=L^p(\Omega;\R^d)$, $F=F_u$ and $F^n=F_{u^n}$, defined as in \eqref{eq:defFu}.
We have to check that \eqref{Resto} holds.
Defining $G^n,G:[0,T]\times\Omega\to\R^d$ by $G^n(t,\omega):=F_{u^n}(t,X_t)(\omega)$ and $G(t,\omega):=F_{u}(t,X_t)(\omega)$,
it is sufficient to prove that $G^n\to G$ in $L^p([0,T]\times\Omega;\R^d)$.
Since $u^n$ converges to $u$ in $\cL_T\otimes\P$-measure, there exists a subsequence $u^{n_k}$ such that  
$u^{n_k}_t(\omega)$ converges to $u_t(\omega)$ for  $\cL_T\otimes\P$-a.e. $(t,\omega)\in[0,T]\times\Omega$.
By the continuity of $f$, we have that
$$|G^{n_k}(t,\omega) - G(t,\omega)|=|F_{u^{n_k}}(t,X_t)(\omega) - F_{u}(t,X_t)(\omega)|\to 0$$ for $\cL_T\otimes\P$-a.e. $(t,\omega)\in[0,T]\times\Omega$.
Moreover 
$$|G^{n}(t,\omega) - G(t,\omega)|^p = |F_{u^{n}}(t,X_t)(\omega) - F_{u}(t,X_t)(\omega) |^p 
\le C\left(1+|X_t(\omega)|^p+\|X_t\|^p_{L^p(\Omega;\R^d)}\right).$$
By dominated convergence we conclude that $G^{n_k}\to G$ in $L^p([0,T]\times\Omega;\R^d)$.
Since the limit is independent of the subsequence, we conclude that
\begin{equation*}
\int_0^T \int_\Omega |F_{u^n}(t,X_t)(\omega) - F_{u}(t,X_t)(\omega)|^p\,\d\P(\omega)\,\d t \to 0.
\end{equation*}

Let us prove \eqref{convJ}. 
For any $t\in[0,T]$, we use the notation $\mu_t^n:=(X^n_t)_\sharp\P$ and $\mu_t:=(X_t)_\sharp\P$.
By \eqref{convX} we have
\begin{equation}\label{convmu}
\sup_{t\in[0,T]}W_p(\mu^n_t,\mu_t) \to 0,\quad \textrm{as }n\to+\infty.
\end{equation}
We focus on the running cost $\cC$. 
Since
\begin{align}\label{eq:estimCchat}
\begin{split}
&\left|\int_0^T\int_\Omega\left(\cC(X^n_t(\omega),u^n_t(\omega),\mu_t^n)- \cC(X_t(\omega),u_t(\omega),\mu_t)\right)\,\d\P(\omega)\, \d t\right|\\
\le& \int_0^T\int_\Omega\left|\cC(X^n_t(\omega),u^n_t(\omega),\mu_t^n)-\cC(X_t(\omega),u_t(\omega),\mu_t)\right|\,\d\P(\omega)\, \d t,
\end{split}
\end{align}
defining $H^n,H:[0,T]\times\Omega\to\R^d$ by $H^n(t,\omega):=\cC(X^n_t(\omega),u^n_t(\omega),\mu_t^n)$ and 
$H(t,\omega):=\cC(X_t(\omega),u_t(\omega),\mu_t)$,
it is sufficient to prove that $H^n\to H$ in $L^1([0,T]\times\Omega;\R^d)$.
Since $u^n$ converges to $u$ in $\cL_T\otimes\P$-measure, \eqref{convX} and  \eqref{convmu} hold, and $\cC$ is continuous,
then there exists a subsequence $H^{n_k}$ such that 
$H^{n_k}(t,\omega)$ converges to $H(t,\omega)$ for  $\cL_T\otimes\P$-a.e. $(t,\omega)\in[0,T]\times\Omega$.
Moreover, by the growth assumptions \eqref{eq:growthC} we have
\begin{equation*}
H^n(t,\omega) \leq C(1+|X^{n}_t(\omega)|^p+\|X^n_t\|^p_{L^p(\Omega;\R^d)}).
\end{equation*}
By a variant of the dominated convergence Theorem (see Theorem 1.20 in \cite{evans2015measure}) we conclude
that $H^{n_k}\to H$ in $L^1([0,T]\times\Omega;\R^d)$.
For the same argument as before we obtain that the whole sequence $H^n\to H$ in $L^1([0,T]\times\Omega;\R^d)$.

The proof that
$$\int_\Omega \cC_T(X^n_T(\omega),\mu^n_T)\,\d\P(\omega) \to \int_\Omega \cC_T(X_T(\omega),\mu_T)\,\d\P(\omega),
\quad \textrm{as }n\to+\infty$$
follows from the same argument.
\end{proof}

\begin{proposition}[Upper semicontinuity of the value function]\label{prop:uscVL}
Let $\S:=(U,f,\cC,\cC_T)$ satisfy Assumption \ref{BA}
and $(\Omega,\frB,\P)$ be a probability space.
If $X_0^n,X_0 \in L^p(\Omega;\R^d) $  satisfy 
$\|X^n_0 - X_0\|_{L^p(\Omega;\R^d)} \to 0$ as $n \to +\infty$, then 
\begin{equation*}
\limsup_{n\to +\infty} V_{\pL} (X_0^n) \leq V_{\pL}(X_0).
\end{equation*}

\end{proposition}
\begin{proof}
Let $\varepsilon>0$ and $(X^\varepsilon,u^\varepsilon)\in\cA_{\pL}(X_0)$ such that 
$J_{\pL}(X^\varepsilon,u^\varepsilon) \leq V_{\pL}(X_0) + \varepsilon$. 
By Proposition \ref{prop:existL}, for any $n\in\N$ there exists $X^{\eps,n}$ such that $(X^{\eps,n},u^\varepsilon)\in\cA_{\pL}(X_0^n)$. 
By Proposition \ref{prop:costRLntoRL}, 
$J_{\pL}(X^{\eps,n},u^\varepsilon) \to J_{\pL}(X^\varepsilon,u^\varepsilon)$, as $n\to +\infty$. 
Hence \[ \limsup_{n\to +\infty}V_{\pL}(X_0^n) \leq \limsup_{n\to +\infty} J_{\pL}(X^{\eps,n},u^\varepsilon) = J_{\pL}(X^\varepsilon,u^\varepsilon) \leq V_{\pL}(X_0) + \varepsilon. \]
Since $\varepsilon$ is arbitrary, we conclude.
\end{proof}

\subsection{Approximation by piecewise constant controls}\label{sec:approxLpiecew}
In this subsection, we approximate admissible controls for the Lagrangian problem with a sequence of suitable piecewise constant controls (i.e. measurable with respect to finite algebras of $\Omega$)
so that the corresponding trajectories and costs converge.  
This is the content of Theorem \ref{prop:RLntoRL}.
The same result is then rephrased in the context of finite particle approximations in  Proposition  
\ref{prop:approx_lagrangian_n}.

\smallskip

Let $(\Omega,\bar \frB, \P)$ be a probability space with $\bar \frB $ a finite algebra.
It can be shown that $\bar \frB$ induces a unique 
minimal (with respect to the inclusion) partition of $\Omega$, that we denote by 
\begin{equation}\label{eq:partBn}
\cP(\bar \frB)=\{A_k:k=1,\ldots,m\}.
\end{equation}
Given a topological space $E$, observe that, since $\bar \frB$ is finite, $g \in \BM((\Omega, \bar\frB); (E,\cB_E))$  if and only if $g$ is constant on the elements of $\cP$.

Let us give the following definition.
{
\begin{definition}\label{def:discreteL}
Let $(\Omega,\bar\frB, \P)$ be a probability space with $\bar\frB$ a finite algebra and $\cP(\bar \frB)$ the associated unique minimal  partition \eqref{eq:partBn}.
Given $m := \# \cP(\bar \frB)$, we define the probability space
$(\Omega^m, \Parts(\Omega^m), \P^m)$, where $\Omega^m:=\{1,\dots,m\}$, $\Parts(\Omega^m)$ is the algebra generated by $( \lbrace 1\rbrace, \ldots \lbrace m\rbrace )$ and $\P^m(\{k\}):=\P(A_k)$, $k=1,\dots,m$.
\end{definition}
}
\begin{proposition}\label{prop:correspLdiscrete}
Let $\S:=(U,f,\mathcal C,\mathcal C_T)$ satisfy Assumption \ref{BA}.
 Let $(\Omega,\bar\frB, \P)$ and $(\Omega^m, \Parts(\Omega^m), \P^m)$ as in Definition \ref{def:discreteL}. 
 Then the Lagrangian problems $\pL_{\bar \frB}:=\pL(\Omega,\bar\frB,\P)$ and $\pL^m:=\pL(\Omega^m, \Parts(\Omega^m), \P^m)$ are equivalent in the sense of Definition \ref{def:equiv_L}.
 \end{proposition}
\begin{proof}
{Let $\psi: \Omega \to \Omega^m$ the function given by
\[ \psi(\omega)=k, \quad \text{ if } \omega\in A_k, \quad k=1,\dots,m\]
and $\phi:\Omega^m \to \Omega$ defined by
\[ \phi(k)=\omega_k, \quad k=1,\dots,m,\]
for a fixed $\omega_k\in A_k$.
We prove that the maps $\psi$ and $\phi$ satisfy} the assumptions of Proposition \ref{p:equivalence}. Measurability of the map $\psi$ follows from the fact that $\psi^{-1}(\{k\})=A_k$ for any $k=1,\dots,m$, while the measurability of $\phi$ is trivial since $\Omega^m$ is equipped with the algebra $\Parts(\Omega^m)$. Moreover, it is immediate to verify that $\psi\circ\phi=i_{\Omega^m}$, which implies \eqref{eq:strana1}. We have to verify \eqref{eq:strana2}: given $X_0\in L^p(\Omega;\R^d)$, we have that $X_0$ is constant on the elements $A_k$ of the partition $\cP(\bar\frB)$, hence it easily follows that $X_0=X_0\circ\phi\circ\psi$. 
Let us verify that $\psi_\sharp\P=\P^m$: for any $k=1,\dots,m$, we have $(\psi_\sharp\P)(\{k\})=\P(\psi^{-1}(\{k\}))=\P(A_k)=\P^m(\{k\})$. Finally, we verify that $\phi_\sharp\P^m=\P$:
for any measurable function $g:\Omega\to\R$, recalling that $g$ is piecewise constant on the elements of $\cP(\bar\frB)$, we have
\begin{align*}
\int_\Omega g(\omega)\,\d(\phi_\sharp\P^m)(\omega)&=\int_{\Omega^m}g(\phi(k))\,\d\P^m(k)=\sum_{k=1}^m g(\omega_k)\P(A_k)\\
&=\int_\Omega g(\omega)\,\d\P(\omega).
\end{align*}
\end{proof}

We recall the notation $\cB_{[0,T]}$ for the Borel $\sigma$-algebra on $[0,T]$.

\begin{definition}\label{approx_prop}[Finite Approximation Property]
Let $(\Omega,\frB, \P)$ be a probability space. 
We say that the family of finite algebras $\frB^n \subset \frB$, $n\in\N$,
satisfies the \emph{finite approximation property} if
for any Banach space $E$ and any $g \in L_\P^1(\Omega; E)$, 
there exists a sequence $g^n: \Omega \to E$ such that  
\begin{enumerate}[label=\textbf{(\roman*)}]
\item \label{FAP1} $g^n$ is $\frB^n$-measurable for any $n\in\N$;
\item  \label{FAP2} $g^n(\Omega) \subset \mathrm{\overline{co}}\left(g(\Omega)\right)$, where $\mathrm{\overline{co}}\left(g(\Omega)\right)$ 
denotes the closed convex hull of $g(\Omega)$;
\item  \label{FAP3} $\left\|  g^n - g \right\|_{L_\P^1(\Omega;E)} \to 0$, as $n \to +\infty$;
\item  \label{FAP4} {if $G:[0,T]\times\Omega\to E$ is $(\cB_{[0,T]}\otimes\frB)$-measurable
and $g_t(\cdot):=G(t,\cdot)\in L_\P^1(\Omega; E)$ for any $t\in[0,T]$,
then the maps
$G^n:[0,T]\times\Omega\to E$ defined by $G^n(t,\omega):=g^n_t(\omega)$, where $g^n_t$ 
is a sequence associated to $g_t$ satisfying items \ref{FAP1},\ref{FAP2},\ref{FAP3},
are $(\cB_{[0,T]}\otimes\frB^n)$-measurable for any $n\in\N$.} \end{enumerate}
\end{definition}

\begin{proposition}\label{prop:FAP}
Let $(\Omega,\frB,\P)$ be a standard Borel space. 
\begin{enumerate}
\item Then there exists a family of  finite algebras
$\frB^n \subset \frB$, $n \in \N$, satisfying the finite approximation property of Definition \ref{approx_prop}.
\item If $\P$ is without atoms, then there exists a family $\frB^n \subset \frB$, $n \in \N$, satisfying the finite approximation property of Definition \ref{approx_prop} such that the associated minimal partition $\cP(\frB^n)= \{A_k^n:k=1,\ldots,n\} $ contains exactly $n$ elements and
\begin{equation}\label{eq:equipart}
\P(A^n_k)=\frac1n, \qquad  k=1,\ldots,n.
\end{equation}
\end{enumerate}
\end{proposition}
The proof of Proposition \ref{prop:FAP} is postponed in Appendix \ref{app_FP}.
Results similar to item (1) of Proposition \ref{prop:FAP} can be found in \cite[Theorem 6.1.12]{stroock2011probability}, where a martingale approach is employed.

\begin{theorem}[Approximation by piecewise constant controls for \pL]\label{prop:RLntoRL}
Let $\S:=(U,f,\mathcal C,\mathcal C_T)$ satisfy Assumption \ref{BA} with $U$ a convex compact subset of a separable Banach space $V$.
Let $(\Omega,\frB,\P)$ be a probability space and assume that there exists $\{\frB^n\}_{n\in\N}$ satisfying the finite approximation property of Definition \ref{approx_prop}.
Let $X_0\in L^p(\Omega;\R^d)$ and $(X,u)\in\mathcal A_{\pL}(X_0)$.
If $\{X^n_0\}_{n\in\N}\subset L^p(\Omega;\R^d)$ satisfies $\|X^n_0-X_0\|_{L^p(\Omega;\R^d)}\to0$,
then there exists a sequence $(X^n,u^n)\in\mathcal A_{\pL}(X^n_0)$ such that
\begin{enumerate}
\item $u^n$ is $(\cB_{[0,T]}\otimes\frB^n)$-measurable;
\item $u^n\to u$ in $(\cL_T\otimes\P)$-measure, as $n\to+\infty$;
\item $\sup_{t\in[0,T]}\|X^n_t-X_t\|_{L^p(\Omega;\R^d)}\to 0$, as $n\to+\infty$;
\item $J_{\pL}(X^n,u^n)\to J_{\pL}(X,u)$, as $n\to+\infty$.
\end{enumerate}
Moreover, if $X_0^n$ is $\frB^{n}$-measurable, $n\in\N$, then
$X_t^n$ is $\frB^{n}$-measurable for any $t\in[0,T]$.
\end{theorem}

\begin{proof} 
For any $t \in [0,T]$,  we denote with $u_t: \Omega \to U$ the measurable control function $u$ at time $t$.
For any $n\in\N$ and $t\in[0,T]$ let $u^n_t:\Omega\to U$ be the $\frB^n$-measurable approximation of $u_t$ given by Definition \ref{approx_prop}.
By the convexity of $U$ and its compactness, from the property \ref{FAP2} of Definition \ref{approx_prop} 
it follows that $u^n_t(\Omega) \subset U$. 
Defining $u^n(t,\omega):=u^n_t(\omega)$, thanks to the property \ref{FAP4} of Definition \ref{approx_prop}, we have that $u^n$ is $(\cB_{[0,T]}\otimes\frB^n)$-measurable. 
By Proposition \ref{prop:existL} we have the existence of $X^n$ with $(X^n,u^n)\in\mathcal A_{\pL}(X^n_0)$.

From the compactness of $U$ and the dominated convergence theorem it follows that 
$\|u^n-u\|_{L^1([0,T]\times\Omega;V)}\to 0$, as $n\to+\infty$.
Consequently, (2) holds.
Properties (3) and (4) follow by Proposition \ref{prop:costRLntoRL}.
\end{proof}

{In the following, we reformulate the approximation result of Theorem \ref{prop:RLntoRL} with the language of particles.}
Let $(\Omega,\frB, \P)$ be a probability space and $\frB^n$ a finite algebra, $n \in \N$. 
Denote with $\cP(\frB^n):= \{ A_k^n: k = 1, \ldots, k(n) \}$ the associated unique minimal partition and define $(\Omega^{k(n)}, \Parts(\Omega^{k(n)}), \P^{k(n)})$ by
\begin{equation}\label{eq:omega_kn}
\begin{aligned}
&\Omega^{k(n)}:=\{1,\dots,k(n)\}, \quad \Parts(\Omega^{k(n)}):= \sigma ( \lbrace 1\rbrace, \ldots \lbrace k(n)\rbrace) \\
&\P^{k(n)}(\{k\}):=\P(A^{n}_k), \quad k=1,\dots,{k(n)}.
\end{aligned}
\end{equation}
In order to approximate trajectories, controls and costs  of  a Lagrangian problem $\pL=\pL(\Omega,\frB,\P)$ with the respective quantities in $\pL^{k(n)}=\pL(\Omega^{k(n)},\Parts(\Omega^{k(n)}),\P^{k(n)})$ we introduce, for every $n \in \N$,  the maps $\psi^n, \phi^n$ and $\cK^n$.
This is necessary due to the fact that the trajectories are not defined on the same space.   

For every $n \in \N$,  we denote with $\psi^n, \phi^n$ the maps 
\begin{equation}\label{def:psinphin}
\begin{split}
&\psi^n: \Omega \to \Omega^{k(n)},\quad\psi^n(\omega)=k, \quad \text{ if } \omega\in A^n_k, \quad k=1,\dots,k(n);\\
&\phi^n:\Omega^{k(n)} \to \Omega,\quad \phi^n(k)=\omega_{k}^n, \quad k=1,\dots,k(n), \quad \text{for a fixed }\omega_{k}^n\in A^n_k.
\end{split}
\end{equation}
Moreover, for every $n \in \N$, we introduce the map $\cK^n: L^p(\Omega^{k(n)}; C([0,T];\R^d)) \times \BM([0,T] \times \Omega^{k(n)};U)$ $\to L^p((\Omega, \frB^n, \P); C([0,T];\R^d)) \times \BM([0,T] \times \Omega;U)$ given by
\begin{equation}\label{def:Kn}
\begin{aligned}
&\cK^n( Y, v):= (Y\circ\psi^n, \hat v), \\
&\text{ where} \qquad \hat v(t,\omega):= v(t,\psi^n(\omega)), \quad  \forall \, (t,\omega)\in[0,T]\times\Omega.\\
\end{aligned}
\end{equation}

\begin{proposition}\label{prop:approx_lagrangian_n}
Let $\S:=(U,f,\mathcal C,\mathcal C_T)$ satisfy Assumption \ref{BA} with $U$ a convex compact subset of a separable Banach space $V$.
Let $(\Omega,\frB,\P)$ and  be a probability space and assume that there exists $\{\frB^n\}_{n\in\N}$ satisfying the finite approximation property of Definition \ref{approx_prop}.
For every $n \in \N$, let $(\Omega^{k(n)},\Parts(\Omega^{k(n)}),\P^{k(n)})$ as in \eqref{eq:omega_kn}.
 
Let $X_0\in L^p(\Omega;\R^d)$ and $(X,u)\in\mathcal A_{\pL}(X_0)$.
If  $ Y^n_0 \in L^p(\Omega^{k(n)};\R^d)$, $n \in \N$, 
satisfies 
\[\lim_{n \to +\infty} \|Y^n_0\circ\psi^n-X_0\|_{L^p(\Omega;\R^d)} = 0,\]
then there exists 
  a sequence $( Y^n,v^n)\in\mathcal A_{\pL^{k(n)}}(Y^n_0)$ such that
\begin{enumerate}
\item $\cK^n(Y^n, v^n) \to (X,u)$ in $C([0,T];L^p(\Omega;\R^d)) \times L^1([0,T] \times \Omega;V)$,  as $n\to+\infty$;
\item $J_{\pL^{k(n)}}(Y^n, v^n)\to J_{\pL}(X,u)$, as $n\to+\infty$.
\end{enumerate}
\end{proposition}
\begin{proof}
For every $n \in \N$, by Proposition \ref{prop:correspLdiscrete} it holds that $\pL_{\frB^n} \sim \pL^{k(n)}$ in the sense of Definition \ref{def:equiv_L}.
Thanks to Proposition \ref{prop:equivSpaces} we have that $X \in \AC^p([0,T];L^p(\Omega;\R^d))$, hence we conclude applying
Theorem \ref{prop:RLntoRL}.
Since $U$ is compact, recall that the convergence $u^n \to u \in L^1([0,T] \times \Omega;V)$ is equivalent to the convergence in $(\cL_T \otimes \P)$-measure. 
\end{proof}

\begin{remark}\label{rmk:noequiprob}
If $(\Omega,\frB,\P)$ is a standard Borel space, the existence of a sequence of finite algebras $\frB^n$ is guaranteed by item (1) in Proposition \ref{prop:FAP}.
Moreover, if $\P$ is without atoms, it is possible to choose $\frB^n$ s.t. $\# \frB^n = n$ and satisfying the property \eqref{eq:equipart} given in item (2) in Proposition \ref{prop:FAP}.  

\noindent Notice that the assumption $\P$ without atoms is necessary to get a sequence of $\frB^n$ with the property \eqref{eq:equipart}. Indeed, if there exists $\omega_0\in\Omega$ such that $\P(\{\omega_0\})=\alpha>0$, then for $n \in \N$ big enough the property \eqref{eq:equipart} fails.
\end{remark}

\subsection{Approximation by continuous controls and trajectories}

The objective of the subsection is twofold.
In Proposition \ref{lemmaA}, under continuity assumptions on both the initial datum and the control,  we exhibit a stability result for a Lagrangian problem $\pL(\Omega, \frB, \P)$ when $\P$ is approximated by a sequence of probability measures $\P^n$.
Then, in Theorem \ref{lemmaB} we approximate admissible controls with \emph{continuous} controls so that the associated trajectories are continuous as well and
the associated costs converge. 
These results are useful to prove the equivalence between Lagrangian and Eulerian optimal control problems (see the proof of Theorem \ref{prop:K>L}).\\ 

Throughout the section, we assume that
\begin{equation}\label{ass:topo}
\text{$(\Omega,\frB,\P)$ standard Borel space, $\tau$ a Polish topology on $\Omega$ such that $\frB = \cB_{(\Omega,\tau)}$.} 
\end{equation}

\noindent In the following regularity result, we prove the existence of a  continuous trajectory {for the Lagrangian dynamics} whenever both the initial datum and the control are continuous.

\begin{lemma}[Continuity]\label{lemmaAA}
Let $\S:=(U,f,\mathcal C,\mathcal C_T)$ satisfy Assumption \ref{BA}
and let $(\Omega,\frB,\P)$ satisfy \eqref{ass:topo}.
Let $\tilde X_0\in C(\Omega;\R^d)$ such that $\tilde X_0\in L^p(\Omega;\R^d)$ and $u\in C([0,T]\times\Omega;U)$. 
If $(X,u) \in \cA_{\pL}(\tilde X_0)$ and $\mu_t:=(X_t)_\sharp\P$ then there exists a unique $\tilde X \in C([0,T] \times \Omega; \R^d)$ satisfying for any $\omega \in \Omega$
\begin{equation}\label{eq:systemLomega}
\begin{cases}
\dot {\tilde X}_t(\omega)=f(\tilde X_t(\omega),u_t(\omega),\mu_t), & \forall \ t\in (0,T)\\ 
\tilde X_{|t=0}(\omega)=\tilde X_0(\omega). &
\end{cases}
\end{equation}
Moreover, $\tilde X_t(\omega) = X_t(\omega)$ for every $t \in [0,T]$ and  $\P$-a.e. $\omega \in \Omega$.
\end{lemma}
\begin{proof} 
For any $\omega\in\Omega$, there exists a unique solution $\tilde X(\omega)\in C^1([0,T];\R^d)$ of \eqref{eq:systemLomega} thanks to the Lipschitz assumptions on the vector field $f$.
Since $u, (\mu_t)_{t \in [0,T]}$ and $\tilde X_0$ are fixed, 
the solutions of \eqref{eq:systemL}  and \eqref{eq:systemLomega} coincide, hence $\tilde X_t(\omega)= X_t(\omega)$  for any $t\in[0,T]$, for $\P$-a.e. $\omega\in\Omega$.

Denoting with $\tilde X :[0,T] \times \Omega\to \R^d$ the function 
$\tilde X(t,\omega)=\tilde X_t(\omega)$, we prove the continuity of $\tilde X$. 
We fix $(t,\omega)\in[0,T]\times\Omega$, and a sequence 
$(t_n,\omega_n)\in[0,T]\times\Omega$ converging to $(t,\omega)$ as $n\to+\infty$.
By triangular inequality,
\begin{equation*}
| \tilde X_{t_n}(\omega_n) - \tilde X_{t}(\omega)| \leq 
| \tilde X_{t_n}(\omega_n) - \tilde X_{t_n}(\omega)| +| \tilde X_{t_n}(\omega) - \tilde X_{t}(\omega)|.
\end{equation*}
The second term is estimated by
\[
| \tilde X_{t_n}(\omega) - \tilde X_{t}(\omega)| \leq \int_t^{t_n} \left| f(\tilde X_s(\omega),u_s(\omega),\mu_s) \right| \d s .
\]
Concerning the first term,
\begin{equation*}
\begin{split}
| \tilde X_{t_n}(\omega_n) - \tilde X_{t_n}(\omega)| &\leq |\tilde X_0(\omega_n) - \tilde X_0(\omega)| \\
&+ \int_0^{t_n} \left| f(\tilde X_s(\omega_n),u_s(\omega_n),\mu_s) - f(\tilde X_s(\omega),u_s(\omega_n),\mu_s) \right| \d s \\
&+ \int_0^{t_n} \left|  f(\tilde X_s(\omega),u_s(\omega_n),\mu_s) - f(\tilde X_s(\omega),u_s(\omega),\mu_s) \right| \d s \\
&\leq |\tilde X_0(\omega_n) - \tilde X_0(\omega)| + L \int_0^{t_n}  | \tilde X_s(\omega_n) - \tilde X_s(\omega) |  \,\d s \\
&+ \int_0^T \left|  f(\tilde X_s(\omega),u_s(\omega_n),\mu_s) - f(\tilde X_s(\omega),u_s(\omega),\mu_s) \right| \d s.
\end{split}
\end{equation*}
By Gronwall lemma we have
\begin{equation*}
\begin{split}
|\tilde X_{t_n}(\omega_n) - \tilde X_{t_n}(\omega)| \leq e^{LT}\left(|\tilde X_0(\omega_n) - \tilde X_0(\omega)| + \int_0^T R_s(\omega_n,\omega)\,\d s\,\right),
\end{split}
\end{equation*}
where
$R_s(\omega_n,\omega):=\left|  f(\tilde X_s(\omega),u_s(\omega_n),\mu_s) - f(\tilde X_s(\omega),u_s(\omega),\mu_s) \right|$.
Collecting the previous inequalities we get
\begin{equation*}
\begin{split}
| \tilde X_{t_n}(\omega_n) - \tilde X_{t}(\omega)| & \leq e^{LT}\left(|\tilde X_0(\omega_n) - \tilde X_0(\omega)| + \int_0^T R_s(\omega_n,\omega)\,\d s\,\right) \\
& +  \int_t^{t_n} \left| f(\tilde X_s(\omega),u_s(\omega),\mu_s) \right| \d s,
\end{split}
\end{equation*}
By the continuity of $\tilde X_0$, the growth property  \eqref{f:growth} and the continuity of $u$ 
we can pass to the limit in the right hand side and we conclude.

\end{proof}

\begin{proposition}[Stability for $\P$]\label{lemmaA}
Let $\S:=(U,f,\mathcal C,\mathcal C_T)$ satisfy Assumption \ref{BA}
and  $(\Omega,\frB,\P)$ satisfying \eqref{ass:topo}.
Let $\P^n,\P \in \PP(\Omega)$, $n \in \N$, such that $\P^n \to \P$ weakly.
Let $\tilde X_0\in C(\Omega;\R^d)$ and $u\in C([0,T]\times\Omega;U)$ such that $\tilde X_0\in \left[\bigcap_{n\in\N}L^p_{\P^n}(\Omega;\R^d)\right]\cap L^p_\P(\Omega;\R^d)$
and 
\begin{equation}\label{eq:conv_X0}
\|\tilde X_0\|_{L^p_{\P^n}(\Omega;\R^d)}\to\|\tilde X_0\|_{L^p_{\P}(\Omega;\R^d)} \quad \text{ if }\ n\to+\infty.
\end{equation}
We denote by $\pL^{n}:= \pL(\Omega,\frB, \P^n)$ and $\pL:= \pL(\Omega,\frB, \P)$. 

Let $(X,u) \in \cA_{\pL}(\tilde X_0)$ and $(X^n,u) \in \cA_{\pL^{n}}(\tilde X_0)$ and denote with $\tilde X, \tilde X^n \in C([0,T] \times \Omega; \R^d)$ {the corresponding solutions} given in Lemma \ref{lemmaAA} associated with $\P$ and $\P^n$, respectively.
Then
\begin{equation}\label{unifconvX}
 \sup_{(t,\omega)\in [0,T] \times \Omega}|\tilde X^n_t(\omega)-\tilde X_t(\omega)|\to 0,   \qquad \text{ as } n \to +\infty,
 \end{equation}
\begin{equation}\label{convJcont}
 J_{\pL^{n}}(X^n,u) \longrightarrow J_{\pL}(X,u), \qquad \text{ as } n \to +\infty.
 \end{equation}  
\end{proposition}

\begin{proof}
We denote $\mu^n_t:= (X^n_t)_\sharp \P^n$ and $\mu_t:= (X_t)_\sharp \P$.

Since $\tilde X_0:\Omega\to\R^d$ is continuous, the weak convergence $\P^n \to \P$ implies that
$\mu_0^n=(\tilde X_0)_\sharp\P^n\to \mu_0=(\tilde X_0)_\sharp\P$ weakly and  \eqref{eq:conv_X0} guarantees $\m_p(\mu_0^n)\to\m_p(\mu_0)$. 
Consequently, by Proposition \ref{prop:wassconv}, it holds $W_p(\mu_0^n,\mu_0)\to 0$ as $n\to+\infty$ and there exists an admissible $\psi:[0,+\infty)\to[0,+\infty)$, 
according to Definition \ref{def:admphi}, such that
\begin{equation}\label{mompsi}
\sup_{n \in \N}\int_{\R^d} \psi(|x|^p)\,\d\mu_0^n(x) < +\infty.
\end{equation}
Using the same argument of the proof of Lemma \ref{lemma:cpt}, thanks to the estimates \eqref{boundXtomega}
and  \eqref{LipXt} 
there exist $\tilde \mu \in C([0,T];\PP_p(\R^d))$ and a (not relabelled) subsequence such that 
\begin{equation}\label{convmutilde}
	\lim_{n\to+\infty}\sup_{s \in [0,T]}W_p(\mu^n_s, \tilde \mu_s) =0.
\end{equation}

We define $\bar X\in C([0,T]\times\Omega;\R^d)$ through the system  \eqref{eq:systemLomega}
using $\tilde\mu_t$ instead of $\mu_t$, i.e., for any $\omega\in\Omega$, $t\mapsto \bar X(t,\omega)$ is the
solution of the problem 
\begin{equation}\label{eq:systemLomegabar}
\begin{cases}
\dot{\bar{X}}_t(\omega)=f(\bar X_t(\omega),u_t(\omega),\tilde\mu_t), &\forall \, t\in (0,T)\\ 
\bar X_{|t=0}(\omega)=\tilde X_0(\omega). &
\end{cases}
\end{equation}
We show that
\begin{equation}\label{unifconvbarX}
 \sup_{(t,\omega)\in [0,T] \times \Omega}|\tilde X^n_t(\omega)-\bar X_t(\omega)|\to 0,   \qquad \text{ as } n \to +\infty.
 \end{equation}
Indeed, for any $\omega\in\Omega$ and $t\in[0,T]$,
\begin{equation*}
\begin{split}
| \tilde X_t^n(\omega) - \bar X_t(\omega)| &\leq \int_0^t \left| f(\tilde X_s^n(\omega),u_s(\omega),\mu^n_s) - f(\bar X_s(\omega),u_s(\omega),\tilde\mu_s) \right| \d s \\
&\leq L \int_0^t \left( | \tilde X_s^n(\omega) -\bar X_s(\omega) |  + W_p(\mu^n_s, \tilde \mu_s) \right)\,\d s,
\end{split}
\end{equation*}
and, by Gronwall inequality, we obtain 
\begin{equation*}
  |\tilde X_t^n(\omega) - \bar X_t(\omega)|  \leq LT e^{LT} \sup_{s \in [0,T]}W_p(\mu^n_s, \tilde \mu_s),
\end{equation*}
which, by \eqref{convmutilde}, proves \eqref{unifconvbarX}.\\
We have to show that $\bar X=\tilde X$.
We first prove that $\tilde\mu_t=(\tilde X_t)_\sharp\P$.
By the uniform convergence \eqref{unifconvbarX}, the continuity of $\bar X_t$ and the weak convergence $\P^n\to\P$,
we obtain that (see \cite[Lemma 5.2.1]{ambrosio2008gradient})
$$ \int_\Omega \phi(\tilde X_t^n(\omega))\,\d\P^n(\omega) \to \int_\Omega \phi(\bar X_t(\omega))\,\d\P(\omega), \qquad
\forall\, \phi \in {C}_b(\R^d;\R).$$
Since $\displaystyle{\int_\Omega \phi(\tilde X_t^n(\omega))\,\d\P^n(\omega) =  \int_{\R^d} \phi(x)\,\d\mu_t^n(x)}$,
by the uniqueness of the weak limit we obtain that $\tilde\mu_t=(\bar X_t)_\sharp\P$.
Then, $\bar X_t(\omega)$ satisfies
\begin{equation}\label{eq:systemLomegabar2}
\begin{cases}
\dot{\bar X}_t(\omega)=f(\bar X_t(\omega),u_t(\omega),(\bar X_t)_\sharp\P), &\forall \, t\in (0,T)\\ 
\bar X_{|t=0}(\omega)=\tilde X_0(\omega). &
\end{cases}
\end{equation}
By the uniqueness result of Proposition \ref{prop:existL} and the definition of the Lagrangian problem, we obtain that $\bar X_t=X_t$ in $L^p_\P(\Omega;\R^d)$ for any $t\in[0,T]$.
In particular we have that $\tilde\mu_t=\mu_t$ for any $t\in[0,T]$. 
It follows that the systems \eqref{eq:systemLomegabar} and \eqref{eq:systemLomega} are the same, 
and then $\bar X_t(\omega)=\tilde X_t(\omega)$ for any $(t,\omega)\in[0,T]\times\Omega$.
Finally, the convergence \eqref{unifconvX} follows by \eqref{unifconvbarX}, 
because the limit $\tilde \mu$ given by the compactness is uniquely determined and it is independent of the subsequence.

For what concerns \eqref{convJcont}, 
we first observe that 
\begin{equation*}
 J_{\pL^{n}}(X^n,u)=J_{\pL^{n}}(\tilde X^n,u), \qquad J_{\pL}(X,u)=J_{\pL}(\tilde X,u).
 \end{equation*}  
We write the running cost as
\begin{equation*}
\begin{split}
 & \int_{\Omega}\int_0^T \cC(\tilde X^n_t(\omega), u_t(\omega),\mu^n_t) \, \d t \,\d\P^n(\omega) \\
 &= \int_{\Omega} \int_0^T \left(\cC(\tilde X^n_t(\omega), u_t(\omega),\mu^n_t) - \cC(\tilde X_t(\omega), u_t(\omega),\mu_t) \right) \,\d t \,\d\P^n(\omega) \\
 &+ \int_{\Omega}\int_0^T \cC(\tilde X_t(\omega), u_t(\omega),\mu_t)  \,\d t\, \d\P^n(\omega).
  \end{split}
 \end{equation*}  
By \eqref{unifconvX} and the continuity of $\cC$, we have that 
$ \left(\cC(\tilde X^n_t(\omega), u_t(\omega),\mu^n_t) - \cC(\tilde X_t(\omega), u_t(\omega),\mu_t) \right)\to 0$ 
uniformly on compact sets of $[0,T]\times\Omega$. 
By the weak convergence of $\P^n$ towards $\P$ we conclude that
\begin{equation*}
\begin{split}
 \lim_{n\to +\infty } \int_{\Omega} \int_0^T \left(\cC(\tilde X^n_t(\omega), u_t(\omega),\mu^n_t) - \cC(\tilde X_t(\omega), u_t(\omega),\mu_t) \right) \,\d t \,\d\P^n(\omega) = 0.
  \end{split}
 \end{equation*}  
We have to prove that 
\begin{equation}\label{eq:lastlimit}
\lim_{n \to +\infty}\int_{\Omega} \int_0^T \cC(\tilde X_t(\omega), u_t(\omega),\mu_t) \d t \,\d\P^n(\omega)  = \int_{\Omega} \int_0^T \cC(\tilde X_t(\omega), u_t(\omega),\mu_t) \d t \,\d\P(\omega). 
\end{equation}
By \eqref{mompsi}, which can be rewritten as
\begin{equation*}
\sup_{n \in \N}\int_{\Omega} \psi(|\tilde X_0(\omega)|^p)\,\d \P^n(\omega) < +\infty,
\end{equation*}
by estimate \eqref{boundXtomega} and the doubling property of $\psi$ we  get
\begin{equation*}
\sup_{t\in[0,T]}\sup_{n \in \N}\int_{\Omega} \psi(|\tilde X_t(\omega)|^p)\,\d \P^n(\omega) < +\infty.
\end{equation*}
By the growth condition \eqref{eq:growthC} and the doubling property of $\psi$ we obtain that 
the map $(t,\omega) \mapsto \cC(\tilde X_t(\omega), u_t(\omega),\mu_t)$ is uniformly
integrable w.r.t. $\{\cL_T\otimes\P^n\}_n$. 
Since this map is also continuous, by \cite[Lemma~5.1.7]{ambrosio2008gradient} we obtain \eqref{eq:lastlimit}. 

Analogously one proves that
\begin{equation*}
	\lim_{n\to+\infty} \int_\Omega \mathcal C_T(\tilde X^n_T(\omega),\mu^n_T)\,\d\P^n(\omega) =
	 \int_\Omega \mathcal C_T(\tilde X_T(\omega),\mu_T)\,\d\P(\omega).
\end{equation*}
\end{proof}

\begin{proposition}[Approximation by continuous controls]\label{lemmaB}
Let $\S:=(U,f,\mathcal C,\mathcal C_T)$ satisfy Assumption \ref{BA} with $U$ a convex compact subset of a separable Banach space $V$, and $(\Omega,\frB,\P)$ satisfying \eqref{ass:topo}.
Let $\tilde X_0\in C(\Omega;\R^d)$ and $u \in \BM([0,T] \times \Omega;U)$ such that $\tilde X_0\in L^p(\Omega;\R^d)$.
If $(X,u) \in \cA_{\pL}(\tilde X_0)$ then there exits a sequence $(\tilde X^n,u^n) \in \cA_{\pL}(\tilde X_0)$  such that 
\begin{enumerate}
\item $u^n \in C([0,T]\times\Omega; U)$ and $\tilde X^n \in C([0,T] \times \Omega; \R^d)$ for any $n \in \N$;
\item $u^n(t,\omega) \to u(t,\omega)$    for $(\cL_T\otimes\P)$-a.e. $(t,\omega)\in[0,T]\times\Omega$; 
\item $\,\lim_{n \to +\infty}\sup_{t \in [0,T]} |\tilde X^n_t(\omega) - X_t(\omega)|=0$ for  $\P$-a.e. $\omega\in\Omega$,\\
 $\,\lim_{n \to +\infty}\sup_{t \in [0,T]} \|\tilde X^n_t - X_t\|^p_{L^p(\Omega;\R^d)} = 0$;
\item $J_{\pL}(\tilde X^n,u^n) \to J_{\pL}(X,u)$, as $n\to +\infty$.
\end{enumerate}
\end{proposition}

\begin{proof}
Since $u\in\BM([0,T]\times\Omega;U)$,
by Lusin's theorem applied to the space $[0,T]\times\Omega$ with the measure $\cL_T\otimes\P$, 
there exists a sequence of compact subsets $A_n\subset[0,T]\times\Omega$ 
such that $A_n \subset A_{n+1}$,  $\cL_T\otimes\P([0,T]\times\Omega \setminus A_n) < \frac{1}{n}$ for every $n\in\N$ and 
$u|_{A_n}: A_n \to U$ is continuous.   
Applying Dugundji's extension theorem \cite[Theorem~4.1]{dugundji1951extension} we can extend $u|_{A_n}$ to a continuous map 
$u^n: [0,T]\times\Omega \to V$ such that $u^n([0,T]\times\Omega)$  is contained in the closed convex subset $U$ of $V$.
Moreover, for $(\cL_T\otimes\P)$-a.e. $(t,\omega)\in[0,T]\times\Omega$ it holds that $u^n(t,\omega) \to u(t,\omega)$,
thanks to the convergence $\cL_T\otimes\P([0,T]\times\Omega \setminus A_n) \to 0$ as $n \to +\infty$. 

Thanks to Proposition \ref{prop:existL} and Lemma \ref{lemmaAA}, for any $n \in \N$ there exists  a unique $\tilde X^n \in C([0,T] \times \Omega; \R^d)$ such that $(\tilde X^n,u^n) \in \cA_\pL(\tilde X_0)$.
Defining $\mu^n_t := (\tilde X^n_t)_\sharp\P$, by Lemma \ref{lemma:cpt} there exists $\bar \mu \in C([0,T]; \PP_p(\R^d))$ 
such that, up to subsequences,
\begin{equation}\label{convmun}
	\sup_{t \in [0,T]}W_p(\mu^n_t, \bar \mu_t) \to 0, \qquad \text{ as } n \to +\infty. 
\end{equation}
For every $\omega\in\Omega$ we define $\bar X(\omega)$ as the unique solution to the problem
\begin{equation*}
\begin{cases}
\dot {\bar X}_t(\omega)=f(\bar X_t(\omega),u_t(\omega),\bar\mu_t), &\textrm{for a.e. }t\in (0,T)\\ 
\bar X_{|t=0}(\omega)=\tilde X_0(\omega). &
\end{cases}
\end{equation*}
Then for any $(t,\omega) \in [0,T]\times \Omega$ it holds
\begin{equation}\label{eq:est_x^n_to_X}
\begin{split}
| \tilde X_t^n(\omega) - \bar X_t(\omega)| &\leq \int_0^t \left| f(\tilde X^n_s(\omega),u^n_s(\omega),\mu^n_s) - 
f(\bar X_s(\omega),u^n_s(\omega), \bar \mu_s) \right| \d s \\
&+ \int_0^t \left| f(\bar X_s(\omega),u^n_s(\omega),\bar \mu_s) - f(\bar X_s(\omega),u_s(\omega), \bar \mu_s) \right| \d s \\
&\leq L\int_0^t \left( | \tilde X_s^n(\omega) -\bar X_s(\omega) |  + W_p(\mu^n_s, \bar \mu_s)  + \mathcal G_{s,\omega}(u^n_s(\omega), u_s(\omega))\right)\d s,
\end{split}
\end{equation}
where $ \mathcal G_{s,\omega}(u^n_s(\omega), u_s(\omega)):=\left| f(\bar X_s(\omega),u^n_s(\omega),\bar \mu_s) - f(\bar X_s(\omega),u_s(\omega), \bar \mu_s) \right|$.
Since by \eqref{f:growth} we have 
$\mathcal G_{s,\omega}(u^n_s(\omega), u_s(\omega))\le C(1+ |\bar X_s(\omega)| + m_p(\bar\mu_s))$, 
by the convergence in item $(2)$ we get that 
\begin{equation}\label{eq:convGG}
\lim_{n\to+\infty}  \int_0^T  \mathcal G_{s,\omega}(u^n_s(\omega), u_s(\omega))\d s =0, \quad \mbox{for $\P$-a.e. }\omega\in\Omega.
\end{equation}
By Gronwall inequality, from \eqref{eq:est_x^n_to_X} we have
\begin{equation}\label{eq:est_X}
\begin{split}
| \tilde X_t^n(\omega) - \bar X_t(\omega)| &\leq e^{LT}LT\Big( \sup_{s\in[0,T]}W_p(\mu^n_s, \bar \mu_s)  +
 \int_0^T \mathcal G_{s,\omega}(u^n_s(\omega), u_s(\omega))\,\d s\Big),
\end{split}
\end{equation}
which, by \eqref{convmun} and \eqref{eq:convGG}, implies
\begin{equation}\label{limittilde}
 \lim_{n \to +\infty}\sup_{t \in [0,T]} |\tilde X^n_t(\omega) - \bar X_t(\omega)|=0,\quad \mbox{for $\P$-a.e. }\omega\in\Omega.
 \end{equation}
From \eqref{limittilde} we have that $\mu_t^n$ weakly converges to $(\bar X_t)_\sharp\P$ and by \eqref{convmun} it follows that
$\bar\mu_t=(\bar X_t)_\sharp\P$ for any $t\in[0,T]$.
By the definition of $\bar X$ and $\bar\mu_t=(\bar X_t)_\sharp\P$ and by
the uniqueness result of Proposition \ref{prop:existL}, we obtain that 
$\bar X_t(\omega)= X_t(\omega)$ for $\P$-a.e. $\omega\in\Omega$ and
for any $t\in[0,T]$. 
The first convergence in item $(3)$ follows from \eqref{limittilde}, while the second convergence comes from the first one and 
Proposition \ref{prop:estimatesL} through dominated convergence.

Finally, item $(4)$ follows by the same argument as in the proof of Proposition \ref{prop:costRLntoRL}.
\end{proof}

\section{Relaxed Lagrangian optimal control problem}
\label{sec_relaxed}

In this Section we define a relaxed version of the Lagrangian problem analyzed in Section \ref{sec_lagrangian}, then we study its properties and its relation with the non-relaxed one.

\begin{definition}[Relaxed Lagrangian optimal control problem (\RL)]\label{def:RL}
Let $\S:=(U,f,\mathcal C,\mathcal C_T)$ satisfy Assumption \ref{BA}
and let $(\Omega,\frB,\P)$ be a probability space.
Given $X_0\in L^p(\Omega;\R^d)$, we say that $(X,\sigma)\in\mathcal A_{\RL}(X_0)$ 
if
\begin{itemize}
\item[(i)] $\sigma\in \BM([0,T]\times\Omega;\mathscr P(U))$;
\item[(ii)] $X\in L^p(\Omega;\AC^p([0,T];\R^d))$ 
and for $\P$-a.e. $\omega\in\Omega$, $X(\omega)$ 
is a solution of the following Cauchy problem
\begin{equation}\label{eq:systemRL}
\begin{cases}
\dot X_t(\omega)=\displaystyle\int_U f(X_t(\omega),u,(X_t)_\sharp\P)\,\d\sigma_{t,\omega}(u), &\textrm{for $\cL_T$-a.e. }t\in]0,T]\\
X_{|t=0}(\omega)=X_0(\omega), &
\end{cases}
\end{equation}
where $X_t: \Omega \to \R^d$ is defined by $X_t(\omega):= X(t,\omega)$ for $\P$-a.e. $\omega \in \Omega$ and 
$\sigma_{t,\omega}:=\sigma(t,\omega)\in\mathscr P(U)$.
\end{itemize}
We refer to $(X,u)\in\mathcal A_{\RL}(X_0)$ as to an \emph{admissible pair}, with $X$ a \emph{trajectory} and $\sigma$  a \emph{relaxed control}.
\\
We define the \emph{cost functional} $J_{\RL}: L^p(\Omega;C([0,T];\R^d))\times \BM([0,T]\times\Omega;\PP(U))\to[0,+\infty)$, by
\begin{equation*}
J_{\RL}(X,\sigma):=\int_{\Omega}\int_0^T\int_U \mathcal C(X_t(\omega),u,(X_t)_\sharp\P)\,\d\sigma_{t,\omega}(u)\, \d t \,\d\P(\omega) \ +\int_\Omega \mathcal C_T(X_T(\omega),(X_T)_\sharp\P)\,\d\P(\omega),
\end{equation*}
and the \emph{value function} $V_{\RL}:L^p(\Omega;\R^d)\to[0,+\infty)$ by
\begin{equation}\label{eq:valueRL}
V_{\RL}(X_0):=\inf\left\{J_{\RL}(X,\sigma)\,:\,(X,\sigma)\in\mathcal A_{\RL}(X_0)\right\}.
\end{equation}
\end{definition}

\smallskip

\begin{remark}\label{rem:RL-L}
By Proposition \ref{prop:ConvexRL} the Relaxed Lagrangian problem  $\RL$ in $\S=(U,f,\cC,\cC_T)$ is a particular
Lagrangian convex problem $\pL'$ in the lifted space $\S' = (\UU,\FF,\CC,\cC_T)$ defined in Definition \ref{def:relax_setting}.
In particular, the system \eqref{eq:systemRL} can be rewritten as
\begin{equation*}\label{eq:systemRL_}
\begin{cases}
\dot X_t(\omega)=\FF(X_t(\omega),\sigma_t(\omega),(X_t)_\sharp\P), &\textrm{for a.e. }t\in (0,T]\\
X_{|t=0}(\omega)=X_0(\omega). &
\end{cases}
\end{equation*}
and the cost functional as
\begin{equation*}
J_{\RL}(X,\sigma):=\int_{\Omega} \int_0^T \CC(X_t(\omega),\sigma(t,\omega),(X_t)_\sharp\P)\, \d t \, \d\P(\omega)+\int_\Omega \mathcal C_T(X_T(\omega),(X_T)_\sharp\P)\,\d\P(\omega).
\end{equation*}

As a consequence, the results proved for the Lagrangian problem $\pL$ also apply to the Relaxed Lagrangian problem $\RL$.
We further point out that even in the relaxed Lagrangian setting, existence of minimizers is not guaranteed in general (see also Remark \ref{rem:lagrangian}).
We refer to Section \ref{sec:counterexample} for a detailed discussion and in particular to Remark \ref{rmk:RLex}.
\end{remark}

\subsection{Equivalence of $\pL$ and $\RL$.  Chattering result}
In this subsection we prove that the value functions for {the Lagrangian and the Relaxed Lagrangian optimal control problems, set in the same parametrization space $(\Omega,\frB,\P)$ and same system $\S$,} coincide.
Precisely, we aim at showing the following theorem whose proof is postponed at the end of the section.

\begin{theorem} \label{cor:VL=VRL}
Let $\S:=(U,f,\mathcal C,\mathcal C_T)$ satisfy Assumption \ref{BA} and $(\Omega,\frB,\P)$ be a probability space 
such that there exists $\{\frB^n\}_{n\in\N}$ 
satisfying the finite approximation property of Definition \ref{approx_prop}. 
If $X_0\in L^p(\Omega;\R^d)$, then $V_{\pL}(X_0)=V_{\RL}(X_0)$.
\end{theorem}
The proof of Theorem \ref{cor:VL=VRL} easily follows from the combination of  Theorem \ref{thm:chat} and Proposition \ref{prop:RL<L} given below. Theorem \ref{thm:chat} is a suitable extension of the classical (in optimal control theory) \emph{chattering theorem} which  permits to approximate relaxed controls with piecewise-constant controls.

Notice that Theorem \ref{cor:VL=VRL} holds in particular if $(\Omega,\frB,\P)$ is a standard Borel space  thanks to Proposition \ref{prop:FAP}.

\medskip

Let us start with  the following proposition.
\begin{proposition}\label{prop:RL<L}
Let $\S:=(U,f,\mathcal C,\mathcal C_T)$ satisfy Assumption \ref{BA} and
$(\Omega,\frB,\P)$ be a probability space.
Let $X_0\in L^p(\Omega;\R^d)$. If $(X,u)\in\mathcal A_{\pL}(X_0)$, then, 
defining $\sigma:[0,T]\times\Omega\to\mathscr P(U)$ by $\sigma(t,\omega)=\delta_{u(t,\omega)}$ we have 
$(X,\sigma)\in\mathcal A_{\RL}(X_0)$ and $J_{\pL}(X,u)=J_{\RL}(X,\sigma)$. 
In particular $V_{\RL}(X_0)\le V_{\pL}(X_0)$.
\end{proposition}
\begin{proof}
The result follows immediately by Proposition \ref{prop:ConvexRL}.
\end{proof}

Recall that if  $\bar \frB$ is a finite algebra on $\Omega$ and $E$ is a Banach space, a function $g:\Omega\to E$ is 
$\bar \frB$-measurable if and only if $g$ is constant on the elements of a partition of $\Omega$ contained in  $\bar \frB$. 

In the following proposition, given a piecewise constant \emph{relaxed} control we approximate it with a sequence of piecewise constant \emph{(non-relaxed)} controls so that the associated trajectories and costs converge.

\begin{proposition}\label{p:chattering:finito}
Let $\S:=(U,f,\mathcal C,\mathcal C_T)$ satisfy Assumption \ref{BA}  and
 $(\Omega,\frB,\P)$ be a probability space.
Let $X_0\in L^p(\Omega;\R^d)$,   $\bar\frB\subset \frB$ a finite algebra, 
 $(X,\sigma)\in\mathcal A_{\RL}(X_0)$ such that
 $\sigma$ is $(\cB_{[0,T]}\otimes\bar\frB)$-measurable.  
Then there exists a sequence $\{ (X^m,u^m) \}_{m\in \N} \subset \cA_{\pL}(X_0)$ such that 
\begin{enumerate}
\item $u^m$ are $(\cB_{[0,T]}\otimes\bar\frB)$-measurable;
\item  for any $\omega \in \Omega$, 
$(i_{[0,T]}, u^m(\cdot, \omega))_\sharp \cL_T \xrightarrow{\mathcal{Y}} \sigma_\omega$,
where $\sigma_\omega :=\sigma(t,\omega)\otimes\cL_T\in\mathscr P([0,T]\times U)$;
\item\label{item2:prop415} 
$\sup_{t\in[0,T]}\|X^m_t-X_t\|_{L^p(\Omega;\R^d)}\to 0$ as $m\to+\infty$;
\item $J_{\pL}(X^m,u^m)\to J_{\RL}(X,\sigma)$, as $m\to+\infty$.
\end{enumerate}
Moreover, if $X_0$ is $\bar \frB$-measurable then
$X_t$, $X_t^m$ are $\bar \frB$-measurable for any $m \in \N$ and $t\in [0,T]$.
\end{proposition}
\begin{proof}
We fix the minimal (w.r.t. inclusion) partition associated to the finite algebra $\bar \frB$ that we denote by $\cP:= \lbrace A_k: k=1, \ldots n \rbrace \subset \bar \frB$. 
For any $k=1, \ldots n $ we select $\omega_k \in A_k$ and apply
Lemma \ref{lemma:young} (with $\T = [0,T]$, $S = U$ and $\lambda = \cL_T$) to the measure
$\nu = \sigma_{\omega_k}:=\sigma(t,\omega_k)\otimes\cL_T\in\mathscr P([0,T]\times U)$.
This yields a sequence of $\cB_{[0,T]}$-measurable functions  $u^m_k:[0,T]\to U$ such that 
\begin{equation*}
(i_{[0,T]}, u^m_k(\cdot))_\sharp \cL_T \xrightarrow{\mathcal{Y}} \sigma_{\omega_k}. 
\end{equation*}
Thus, we  define $u^m: [0,T] \times \Omega \to U$ setting $u^m(t,\omega):= u^m_k(t)$ if $\omega \in A_k$. 
By construction, the function $\omega\mapsto u^m(t,\omega)$ is constant 
on $A_k$, for any $k= 1, \ldots, n$. 
Furthermore, for any $\omega\in\Omega$ the maps  $t \mapsto u^m(t,\omega)$ are $\cB_{[0,T]}$-measurable.
The sequence of controls $u^m\in \BM([0,T]\times\Omega)$ readily satisfies items ($1$) and ($2$).

\smallskip 
Given $u^m$ constructed above, by Proposition \ref{prop:existL} there exists a unique $X^m \in L^p(\Omega;\AC^p([0,T];\R^d))$ such that $(X^m,u^m) \in \cA_{\pL}(X_0)$. 
Thanks to Remark, \ref{re:equiv} we interpret $X^m\in \AC^p([0,T]; L^p(\Omega;\R^d))$ and define $ \mu^m \in \AC^p([0,T]; \PP_p(\R^d))$ by  $\mu_t^m := (X_t^m)_\sharp\P$.
By Lemma \ref{lemma:cpt}
there exists a (non relabeled) subsequence $\mu^{m}$ and 
$\tilde \mu \in C([0,T]; \PP_p(\R^d))$ such that 
\begin{equation}\label{eq:convW}
	\lim_{m\to\infty}\sup_{t\in[0,T]}W_p(\mu^{m}_t,\tilde\mu_t)=0.
\end{equation}
We define $g:[0,T]\times\R^d\times U\to \R^d$ by 
\begin{equation*}
	g(t, y , u) := f(y, u, \tilde \mu_t).
\end{equation*} 
Selecting a representative $X_0$ defined for every $\omega\in\Omega$, let  $Y^{m}(\omega) \in \AC^p([0,T]; \R^d)$ be the unique solution of the Cauchy problem
\begin{equation}
\begin{cases}
\dot Y^m_t(\omega) = g(t, Y^m_t(\omega), u^{m}(t,\omega)), & \textrm{for a.e. }t\in(0,T) \\
Y^m_{|t=0}(\omega) = X_0(\omega).
\end{cases}
\end{equation}
For  any $\omega \in \Omega$, let also  
$Y(\omega) \in \AC^p([0,T]; \R^d)$ be the unique solution of the Cauchy problem
\begin{equation}\label{eq:CPY}
\begin{cases}\displaystyle
\dot Y_t(\omega) = \int_U g(t,Y_t(\omega),u) \,\d \sigma_{t,\omega}(u), & \textrm{for a.e. }t\in(0,T) \\
Y_{|t=0}(\omega) = X_0(\omega).
\end{cases}
\end{equation}
Hence, by item $(2)$ and assumptions  \eqref{eq:Lipf} and \eqref{f:growth} we can apply Lemma \ref{lemma:youngODE} to obtain
\begin{equation}\label{eq:convY}
\lim_{m \to + \infty}\sup_{t \in [0,T]} \left|Y^{m}_t(\omega) - Y_t(\omega)\right| = 0,  \qquad \forall\,\omega\in\Omega. 
\end{equation}
Since $(X^m,u^m) \in \cA_{\pL}(X_0)$, by definition of the Lagrangian problem, for $\P$-a.e. $\omega\in\Omega$
$X^m(\omega) \in \AC^p([0,T];\R^d)$  and
\begin{equation}\label{eq:X^m}
\begin{cases}
\dot X^{m}_t(\omega) = f(X^{m}_t(\omega), u^{m}_t(\omega),\mu^m_t), \quad \textrm{for a.e. }t\in(0,T) & \\
X^{m}_{|t=0}(\omega) = X_0(\omega).
\end{cases}
\end{equation}
Then
\begin{equation*}
\begin{split}
\left| Y^{m}_t(\omega) - X^{m}_t(\omega) \right| &\leq  \int_0^t |f(Y^{m}_s(\omega), u^{m}(s,\omega),\tilde\mu_s) - f(X^{m}_s(\omega), u^{m}(s,\omega),\mu^m_s) | \,\d s \\
&\leq L \int_0^t \left(\left| Y^{m}_s(\omega) - X^{m}_s(\omega) \right| + W_p(\tilde \mu_s, \mu_s^{m})\right) \,\d s \\
&\leq L \int_0^t \left| Y^{m}_s(\omega) - X^{m}_s(\omega) \right| \,\d s +LT\sup_{s\in[0,T]} W_p(\tilde \mu_s, \mu_s^{m}).
\end{split}
\end{equation*}
By Gronwall inequality we get 
\begin{equation}\label{conv:yx}
\sup_{t \in [0,T]}\left| Y^{m}_t(\omega) - X^{m}_t(\omega) \right| \leq  LTe^{LT}\sup_{s \in [0,T]}W_p(\tilde\mu_s,\mu_s^{m}).
\end{equation}
From \eqref{conv:yx}, \eqref{eq:convW} and \eqref{eq:convY} it follows that 
\begin{equation}\label{eq:convX}
\lim_{m \to + \infty}\sup_{t \in [0,T]} \left|X^{m}_t(\omega) - Y_t(\omega)\right| = 0,  \qquad \text{for }\P\text{-a.e. }\omega\in\Omega. 
\end{equation}
By \eqref{eq:convX} it follows that $\mu^{m}_t = (X^{m}_t)_\sharp\P \to (Y_t)_\sharp\P$ weakly for any  $t \in [0,T]$,
and then, by \eqref{eq:convW}, it holds that $\tilde\mu_t = (Y_t)_\sharp\P$ for any $t \in [0,T]$.
Thus, thanks to \eqref{eq:CPY} and the definition of $g$ we conclude that $(Y,\sigma) \in \mathcal A_{\RL}(X_0)$.
Since $(X,\sigma) \in \mathcal A_{\RL}(X_0)$, by the uniqueness result of Propositions \ref{prop:existL} we have that $Y = X$
and
\begin{equation}\label{eq:convXunif}
\lim_{m \to + \infty}\sup_{t \in [0,T]} \left|X^{m}_t(\omega) - X_t(\omega)\right| = 0,  \qquad \text{for }\P\text{-a.e. }\omega\in\Omega.
\end{equation}
Finally, to prove item (3) it is enough to observe that 
\begin{equation}\label{eq:convLp}
\sup_{t \in [0,T]} \left\|X^{m}_t-X_t\right\|^p_{L^p(\Omega;\R^d)}   
\le \int_{\Omega}\sup_{t \in [0,T]} \left|X^{m}_t(\omega)-X_t(\omega)\right|^p\,\d\P(\omega).
\end{equation}
By \eqref{eq:convX}, and \eqref{boundXtomega} we can pass to the limit in \eqref{eq:convLp} by dominated convergence.

\smallskip

To prove item ($4$) we write 
\begin{equation}\label{eq:2stabC}
\begin{split}
&\left| \int_\Omega \int_0^T \cC(X^m_t(\omega),u^m_t(\omega),\mu^m_t)\,\d t \, \d\P(\omega) 
	-\int_\Omega \int_0^T \int_U \cC(X_t(\omega),u,\mu_t)\,\d\sigma_{t,\omega}(u)\,\d t \, \d\P(\omega)\right|\\
&\leq  \left| \int_\Omega \left(\int_0^T \cC(X^m_t(\omega),u^m_t(\omega),\mu^m_t) 
	-\cC (X_t(\omega),u^m_t(\omega),\mu_t)\,\d t\right)\,\d\P(\omega) \right|\\
&+\left| \int_\Omega \left(\int_0^T \cC(X_t(\omega),u^m_t(\omega),\mu_t)\,\d t 
	-\int_{[0,T] \times U} \cC(X_t(\omega),u,\mu_t)\,\d\sigma_{\omega}(t,u)\right) \,\d\P(\omega)\right|.\\
\end{split}
\end{equation}
Since \eqref{eq:convXunif} holds, for $\P$-a.e. $\omega\in\Omega$, there exists a compact $K_\omega \subset \R^d$ such that
$X^m_t(\omega),X_t(\omega)\in K_\omega$ for any $m\in\N$ and $t\in[0,T]$. 
Analogously, by \eqref{eq:convW} there exists a compact $\cK \subset \PP_p(\R^d)$ such that
$\mu^m_t,\mu_t\in \cK$ for any $m\in\N$ and $t\in[0,T]$.
By Proposition \ref{prop:wassconv} there exists an admissible $\psi$ such that
\begin{equation}\label{eq:estequi}
\sup_{t\in[0,T]}\sup_{m \in \N}\int_{\Omega} \psi(|X^m_t(\omega)|^p)\,\d \P(\omega)  = \sup_{t\in[0,T]}\sup_{m \in \N}\int_{\R^d} \psi(|x|^p)\,\d \mu_t^m(x) < +\infty.
\end{equation}
By the continuity of $\cC$ there exists a modulus of continuity $\alpha_\omega:[0,+\infty)\to[0,+\infty)$
for the restriction of $\cC$ to the compact set $K_\omega\times U\times\cK$.
Then, for $\P$-a.e. $\omega\in\Omega$, 
\begin{equation*}
	\sup_{t\in[0,T]}|\cC(X_t^{m}(\omega), u_t^{m}(\omega),\mu_t^{m})-\cC(X_t(\omega),u_t^{m}(\omega),\mu_t)|\le
\alpha_\omega\big(\sup_{t\in[0,T]}(|X_t^{m}(\omega)- X_t(\omega)|+W_p(\mu_t^{m},\mu_t))\big).
\end{equation*}
Taking into account the previous consideration together with  \eqref{eq:convXunif}, \eqref{eq:convW}, 
the growth condition \eqref{eq:growthC} and \eqref{eq:estequi},
we obtain
$$  \left| \int_\Omega \left(\int_0^T \cC(X^m_t(\omega),u^m_t(\omega),\mu^m_t) 
	-\cC (X_t(\omega),u^m_t(\omega),\mu_t)\,\d t\right)\,\d\P(\omega) \right| \to 0.$$
For the second term in the right hand side of \eqref{eq:2stabC}, for $\P$-a.e. $\omega \in \Omega$ we define $h_\omega: [0,T] \times U \to \R$ by 
$h_\omega(t,u):= \cC (X_t(\omega),u,(X_t)_\sharp\P)$.
Notice that $h_\omega$ is continuous and bounded in $[0,T] \times U$, hence from the Young convergence of item (2) we get 
\begin{equation}
\left| \int_0^T h_\omega (t,u^m_t(\omega))\,\d t -  \int_{[0,T] \times U} h_\omega (t,u)\,\d \sigma_{\omega}(t,u)\right| \to 0, 
\qquad \text{ for $\P$-a.e. } \omega \in \Omega.
\end{equation}
From the growth assumptions \eqref{eq:growthC} and dominated convergence theorem we obtain that
$$\left| \int_\Omega \left(\int_0^T \cC(X_t(\omega),u^m_t(\omega),\mu_t)\,\d t 
	-\int_{[0,T] \times U} \cC(X_t(\omega),u,\mu_t)\,\d\sigma_{\omega}(t,u)\right) \,\d\P(\omega)\right|\to 0. $$
	Finally, thanks to \eqref{eq:convXunif} and  \eqref{eq:convW} we also obtain that
\begin{equation*}
	\lim_{m\to+\infty} \int_\Omega \cC_T(X^m_T(\omega),\mu^m_T)\,\d\P(\omega) =
	 \int_\Omega \cC_T(X_T(\omega),\mu_T)\,\d\P(\omega).
\end{equation*}
For what concerns the last statement, since $X_0$ is $\bar \frB$-measurable, (hence constant on the elements of the partition $\cP$), the  
measurability of $X_t^m$ with respect to the algebra $\bar \frB$ follows by uniqueness of solutions to \eqref{eq:X^m}. 
The same argument also yields that $X_t$ is $\bar \frB$-measurable.   

\end{proof}

Combining Theorem \ref{prop:RLntoRL} and Proposition \ref{prop:costRLntoRL} applied to the Relaxed Lagrangian problem $\RL$, with 
 Proposition \ref{p:chattering:finito}, we can prove the following Theorem.

\begin{theorem}[Chattering]\label{thm:chat}
Let $\S:=(U,f,\mathcal C,\mathcal C_T)$ satisfy Assumption \ref{BA}.
Let $(\Omega,\frB,\P)$ be a probability space and $\{\frB^n\}_{n\in\N}$ 
satisfying the finite approximation property of Definition \ref{approx_prop}.
Let $X_0\in L^p(\Omega;\R^d)$, $(X,\sigma)\in\mathcal A_{\RL}(X_0)$ and
 $\{X^n_0\}_{n\in\N}\subset L^p(\Omega;\R^d)$, such that 
\begin{equation}\label{eq:conv_X0n}
\|X^n_0-X_0\|_{L^p(\Omega;\R^d)} \to 0 \quad \text{ as } n \to +\infty.
\end{equation}
Then there exists a sequence  $\{(\tilde X^n,\tilde u^n)\}_{n\in \N}$ such that $(\tilde X^n,\tilde u^n) \in  \cA_{\pL}(X^n_0)$ for every $n \in \N$ and the following hold
\begin{enumerate}
\item  $\tilde u^n$ are $(\cB_{[0,T]}\otimes\frB^{n})$-measurable;
\item for $\P$-a.e. $\omega \in \Omega$, $(i_{[0,T]}, \tilde u^n(\cdot,\omega))_\sharp \cL_T \xrightarrow{\mathcal{Y}} \sigma_\omega$, 
as $n\to+\infty$,
where $\sigma_\omega :=\sigma(t,\omega)\otimes\cL_T\in\mathscr P([0,T]\times U)$;
\item $\sup_{t\in[0,T]}\|\tilde X^n_t-X_t\|_{L^p(\Omega;\R^d)}\to 0$ as $n\to+\infty$;
\item $J_{\pL}(\tilde X^n,\tilde u^n)\to J_{\RL}(X,\sigma)$, as $n\to+\infty$.
\end{enumerate}
Moreover, if $X_0^n$ is $\frB^{n}$-measurable, $n\in\N$, then
$\tilde X_t^n$ is $\frB^{n}$-measurable for any $t\in[0,T]$.
\end{theorem}
\begin{proof}
Let $(X,\sigma)\in\mathcal A_{\RL}(X_0)$ and $\{X^n_0\}_{n\in \N}$ satisfying \eqref{eq:conv_X0n}.
Applying Theorem \ref{prop:RLntoRL} to the relaxed problem $\RL$ (which is a Lagrangian problem in a lifted space as discussed in Remark \ref{rem:RL-L}),
there exists a sequence $\{(X^n,\sigma^n)\}_{n\in\N}$ such that $(X^n,\sigma^n) \in \cA_{\RL}(X^n_0)$ for every $n \in \N$. 
Moreover $\sigma^n$ are $(\cB_{[0,T]}\otimes\frB^n)$-measurable,  $\sigma^n\to\sigma$ in $(\cL_T\otimes\P)$-measure and, as a consequence, { we have that $\sigma^n(t,\omega)\to\sigma(t,\omega)$ weakly in $\PP(U)$ for $\cL_T\otimes\P$-a.e. $(t,\omega)$, up to a non-relabelled subsequence. Thus, by Remark \ref{rem:Young-weak}, } we get (up to a non-relabelled subsequence)
\begin{equation}\label{eq:convysigma}
\sigma^n_\omega:=\sigma^n(t,\omega)\otimes\cL_T \xrightarrow{\mathcal{Y}} \sigma_\omega,
\qquad \text{ for $\P$-a.e. } \omega \in \Omega.
\end{equation}

By Proposition \ref{p:chattering:finito}, for any fixed $n\in\N$, there exists a sequence 
$\{(Y^{n,m},u^{n,m})\}_{m\in\N} \subset \cA_{\pL}(X^n_0)$, 
with $u^{n,m}$ $(\cB_{[0,T]}\otimes\frB^n)$-measurable, such that
\begin{enumerate}
\item[(i)] for any $\omega \in \Omega$, $(i_{[0,T]}, u^{n,m}(\cdot, \omega))_\sharp \cL_T \xrightarrow{\mathcal{Y}} \sigma^n_\omega$,  
as $m\to+\infty$;
\item[(ii)] $\sup_{t\in[0,T]}\|X^n_t-Y^{n,m}_t\|_{L^p(\Omega;\R^d)}\to 0$,  as $m\to+\infty$;
\item[(iii)] $J_{\pL}(Y^{n,m},u^{n,m})\to J_{\RL}(X^n,\sigma^n)$, as $m\to+\infty$.
\end{enumerate}
Let us denote by $\cP(\frB^n):= \lbrace A^n_k, k=1,\ldots, k(n)\rbrace$ the minimal (finite) partition induced by $\frB^n$.
Let also $\sigma_\omega^{n,m}:=(i_{[0,T]}, u^{n,m}(\cdot, \omega))_\sharp \cL_T$ and observe that the map $\omega\mapsto \sigma_\omega^{n,m}$ is constant on the elements of $\cP^n$.
Then, if we select a representative $\omega_k \in A^n_k$ for any $k =1,\ldots, k(n)$, from item (i) it follows that
\begin{equation}
\lim_{m\to+\infty}\sup_{k = 1, \ldots, k(n)} \, \delta(\sigma^{n,m}_{\omega_k}, \sigma^n_{\omega_k} )= \lim_{m\to+\infty}\sup_{\omega\in\Omega} \, \delta(\sigma^{n,m}_\omega, \sigma^n_\omega )=0,
\end{equation} 
where $\delta$ metrizes the Young convergence in $[0,T] \times U$. 
Recall that Young convergence is indeed equivalent to the weak convergence in $[0,T] \times U$, see Remark \ref{rem:Young-weak}.

For any $n\in\N$, let $m(n)$ be such that 
$$\sup_{t\in[0,T]}\|X^n_t-Y^{n,m(n)}_t\|_{L^p(\Omega;\R^d)} < \frac1n,$$
$$|J_{\pL}(Y^{n,m(n)},u^{n,m(n)})- J_{\RL}(X^n,\sigma^n)| < \frac1n$$
and
$$\sup_{\omega\in\Omega}\, \delta(\sigma^{n,m(n)}_\omega, \sigma^n_\omega )<\frac1n.$$
Let us define $\tilde X^n:=Y^{n,m(n)}$, the control function $\tilde u^n:=u^{n,m(n)}:[0,T]\times\Omega\to U$ and
 $\tilde\sigma_\omega^{n}:=\sigma_\omega^{n,m(n)}=(i_{[0,T]}, \tilde u^{n}(\cdot, \omega))_\sharp \cL_T$.
Notice that, by construction, $(\tilde X^n,\tilde u^n)\in\mathcal A_{\pL}(X^n_0)$. 

Fix now $\eps>0$. 
By Theorem \ref{prop:RLntoRL} there exists $n_\eps$ such that
$$\sup_{t\in[0,T]}\|X^n_t-X_t\|_{L^p(\Omega;\R^d)} < \eps, \quad \forall n>n_\eps,$$
and  $$|J_{\RL}(X^{n},\sigma^{n})- J_{\RL}(X,\sigma)| < \eps, \quad \forall n>n_\eps.$$
Then, using the definition of  $\tilde X^n$ and $\tilde u^n$, for any $n>n_\eps$ it holds
\begin{align*}
\sup_{t\in[0,T]}\|\tilde X^n_t-X_t\|_{L^p(\Omega;\R^d)}&\le \sup_{t\in[0,T]}\|Y^{n,m(n)}_t-X^n_t\|_{L^p(\Omega;\R^d)}\\
&+\sup_{t\in[0,T]}\|X^n_t-X_t\|_{L^p(\Omega;\R^d)} <\frac1n +\eps,
\end{align*}
and
\begin{align*}
|J_{\pL}(\tilde X^n,\tilde u^n)-J_{\RL}(X,\sigma)|&\le |J_{\pL}(Y^{n,m(n)},u^{n,m(n)})-J_{\RL}(X^n,\sigma^n)|\\
&\quad+|J_{\RL}(X^n,\sigma^n)-J_{\RL}(X,\sigma)| < \frac1n +\eps.
\end{align*}
If we send $n \to +\infty$, items (3) and (4) follow by the arbitrariness  of $\varepsilon >0$.

It remains to show item (2). Fix again $\varepsilon >0$ and choose $\omega\in\Omega$ for which the convergence in \eqref{eq:convysigma} holds.
Then, there exists $n_\eps(\omega)>0$ such that 
$$\delta(\sigma^{n}_\omega, \sigma_\omega ) <\eps, \quad \forall n>n_\eps(\omega),$$
and
\begin{equation*}
\delta(\tilde \sigma^{n}_\omega, \sigma_\omega ) \leq
\delta(\tilde \sigma^{ n}_\omega, \sigma^n_\omega ) + 
\delta(\sigma^{n}_\omega, \sigma_\omega )< \frac1n+\eps,  \quad \forall	\, n>n_\eps(\omega).
\end{equation*}
Sending $n\to+\infty$ we get item (2).
\end{proof}

\section{Eulerian optimal control problem}
\label{sec_eulerian}

In this Section we describe the Eulerian formulation of the optimal control problem and 
we study its properties under the Convexity Assumption \ref{CA}. {In particular, as stated in Theorem \ref{thm:minE}, in this setting we get the existence of minimizers.}
Recall that $\Bor([0,T] \times \R^d;U)$ denotes the set of Borel measurable functions.

\begin{definition}[Eulerian optimal control problem (\E)]\label{def:E1}
Let $\S = (U,f,\cC,\cC_T)$  satisfy Assumption \ref{BA}. 
Given $\mu_0\in\PP_p(\R^d)$, we say that
 $(\mu,\ubar u)\in\cA_{\E}(\mu_0)$ if
\begin{itemize}
\item[(i)] $\ubar u\in\Bor([0,T]\times\R^d;U)$;
\item[(ii)] $\mu\in \AC^p([0,T];\PP_p(\R^d))$ is a distributional solution of the Cauchy problem
\begin{equation}\label{eq:systemE}
\begin{cases}
\partial_t\mu_t+\mathrm{div}\,(v_t\mu_t)=0, &\textrm{in }[0,T]\times\R^d\\ 
\mu_{t=0}=\mu_0,&
\end{cases}
\end{equation}
where 
$v\in\Bor([0,T]\times\R^d;\R^d)$ is defined by
 $v_t(x):=f(x,\ubar u(t,x),\mu_t)$ and $\mu_t:= \mu(t)$.
\end{itemize}
We refer to $(\mu, \ubar u)\in\cA_{\E}(\mu_0)$ as to an \emph{admissible pair}, with $\mu$ a \emph{measure trajectory} and $\ubar u$  a \emph{Eulerian control}.\\
We define the \emph{cost functional}  
$J_{\E}:\operatorname{C}([0,T];\PP_p(\R^d))\times {\Bor([0,T]\times\R^d; U)}\to[0,+\infty)$ by
\begin{equation*}
J_{\E}(\mu,\ubar u):=\int_0^T\int_{\R^d}\cC(x,\ubar u(t,x),\mu_t)\,\d\mu_t(x)\,\d t+\int_{\R^d} \cC_T(x,\mu_T)\,\d\mu_T(x),
\end{equation*}
and the \emph{value function} $V_{\E}:\PP_p(\R^d)\to[0,+\infty)$ by
\begin{equation*}
V_{\E}(\mu_0):=\inf\{J_{\E}(\mu,\ubar u)\,:\,(\mu,\ubar u)\in\mathcal A_{\E}(\mu_0)\}.
\end{equation*}
\end{definition}

\begin{remark}
Notice that, given $\ubar u\in \Bor([0,T]\times\R^d;U)$, $\mu\in C([0,T];\PP_p(\R^d))$ and setting $v_t(x):=f(x,\ubar u(t,x),\mu_t)$ as in Definition \ref{def:E1}, we have
\begin{align*}
\int_{\R^d}|v_t(x)|^p\,\d\mu_t(x)&= \int_{\R^d}|f(x,\ubar u(t,x),\mu_t)|^p\,\d\mu_t(x)\\
&\le \tilde C \left( 1+\int_{\R^d}|x|^p\,\d\mu_t(x)\right)\le \bar C, \quad \forall t\in[0,T]
\end{align*}
for some constants $\tilde C,\bar C>0$, thanks to the growth condition \eqref{f:growth} and since $\mu\in C([0,T];\PP_p(\R^d))$. In particular, we get $v\in L^p(0,T;L^p_{\mu_t}(\R^d;\R^d))$.  Thus, if $\mu$ is also a distributional solution of \eqref{eq:systemE}, then $\mu\in \AC^p([0,T];\PP_p(\R^d))$. Hence, in Definition \ref{def:E1}(ii) we could have just required $\mu\in C([0,T];\PP_p(\R^d))$.

Observe also that the functional $J_\E$ is finite thanks to the growth condition \eqref{eq:growthC}.
\end{remark}

\begin{proposition}
Let $\mu_0\in\PP_p(\R^d)$. Then $\cA_\E(\mu_0)\not=\emptyset$.
\end{proposition}
\begin{proof}
Let us fix  $u_0\in U$ and define $\ubar u(t,x)=u_0$ for any $(t,x)\in[0,T]\times\R^d$. 
Applying Proposition \ref{prop:existL} with 
$\Omega=\R^d$, $\P=\mu_0$, $X_0(x)=x$ for any $x\in\R^d$ and $u = \ubar u$, there exists $X\in L^p_{\mu_0}(\R^d;\AC^p([0,T];\R^d))$  such that $(X,\ubar u) \in \cA_\L(X_0)$.
Defining $\mu_t:=(X_t)_\sharp\mu_0$, from the definition of the Lagrangian problem it holds 
 \begin{equation}\label{flow}
 	X_t(x)=x+\int_0^t f(X_s(x),u_0,\mu_s)\,\d s , \qquad \forall \, t \in [0,T] \text{ and }\mu_0\text{-a.e. }x\in\R^d,
 \end{equation}
Furthermore, in view of Proposition \ref{prop:equivSpaces} we have that 
 $X\in \AC^p([0,T];L^p_{\mu_0}(\R^d;\R^d))$ and
from \eqref{eq:WpLp} we get  $W_p(\mu_t,\mu_s)\le \|X_t-X_s\|_{L^p_{\mu_0}(\R^d;\R^d)}$ for any $t,s\in[0,T]$.
This readily implies that $\mu \in \AC^p([0,T];\PP_p(\R^d))$.
If we define $v_t(x):= f(x,u_0,\mu_t)$, it remains to show that $\mu$ is a distributional solution of \eqref{eq:systemE}.
This is a standard argument, in view of the fact that \eqref{flow} represents the system of characteristics of \eqref{eq:systemE} (see e.g. \cite[Lemma 8.1.6]{ambrosio2008gradient}). 
\end{proof}

\begin{remark}
Given $\mu_0\in\PP_p(\R^d)$, general results granting the existence of solutions to the Cauchy problem \eqref{eq:systemE} are provided for instance in \cite[Theorem~A.2]{orrieri2018large} when $\ubar u\in\Bor([0,T]\times\R^d;U)$ is also a Carath\'eodory function. 
\end{remark}

\begin{definition}\label{def:nu_converges}
Let $U$ be a subset of a separable Banach space $V$, and denote with $V'$ the dual of $V$. Let  $(\mu^n,\ubar u^n), (\mu,\ubar u)\in  \operatorname{C}([0,T];\PP_p(\R^d)) \times {\Bor([0,T]\times\R^d; U)}$.
We say that $(\mu^n,\ubar u^n)$ converges to $(\mu,\ubar u)$   if
\begin{itemize}
\item $\mu^n$ converges to $\mu$ in   $\operatorname{C}([0,T];\PP_p(\R^d))$,
\item for any $\phi\in C_c([0,T]\times\R^d;V')$ we have
\begin{equation}\label{convweaknu}
	\lim_{n\to+\infty}\int_0^T\int_{\R^d}\langle \phi(t,x), \ubar u^n(t,x)\rangle\,\d \mu_t^n(x)\,\d t
	=\int_0^T\int_{\R^d}\langle \phi(t,x), \ubar u(t,x)\rangle\,\d \mu_t(x)\,\d t.
\end{equation}
\end{itemize}
\end{definition}

\begin{proposition}[Compactness]\label{prop:compactness}
Let $\S=(U,f,\cC,\cC_T)$ satisfy the Convexity Assumption \ref{CA}.
Let $\mu_0, \mu_0^n \in \PP_p(\R^d)$ such that $W_p(\mu_0^n, \mu_0) \to 0$, as $n \to +\infty$. 
If $(\mu^n,\ubar u^n) \in \cA_{\E}(\mu^n_0)$, $n\in\N$, then there exist $(\mu,\ubar u) \in \cA_{\E}(\mu_0)$ 
and a subsequence  $(\mu^{n_k},\ubar u^{n_k})$ such that  
$(\mu^{n_k},\ubar u^{n_k})$ converges to $(\mu,\ubar u)$, as $k\to+\infty$, according to Definition \ref{def:nu_converges}.
\end{proposition}

\begin{proof}
Let $(\mu^n,\ubar u^n) \in \cA_{\E}(\mu^n_0)$.
Since $W_p(\mu_0^n,\mu_0) \to 0$, by Proposition \ref{prop:wassconv} it holds that
\begin{equation}\label{bound_mu0p}
	\sup_{n \in \N}\int_{\R^d} |x|^p\d \mu^n_0(x) < +\infty 
\end{equation}
and there exists an admissible $\psi:[0,+\infty)\to[0,+\infty)$, in the sense of Definition \ref{def:admphi}, such that
\begin{equation}\label{eq:unifmompsi2}
	\sup_{n \in \N}\int_{\R^d} \psi(|x|^p)\d \mu^n_0(x) < +\infty.
\end{equation}
In order to apply Ascoli-Arzel\`a Theorem to the sequence $\{\mu^n\}_{n\in\N}$, we show that
\begin{equation}\label{eq:cpt_muN}
	\sup_{n \in \N}\sup_{t \in [0,T]}\int_{\R^d} \psi(|x|^p)\d \mu^n_t(x) < +\infty
\end{equation}
and there exists a constant $C$ such that
\begin{equation}\label{eq:BB}
	W^p_p (\mu^n_s,\mu^n_t)  \leq C|t-s|,  \qquad \forall\, s,t\in[0,T].
\end{equation} 

We start by estimating $\int_{\R^d} |x|^p\d \mu^n_t(x)$. 
We formally use the map $x\mapsto |x|^p$ as a test function for the weak formulation of the continuity equation
(a rigorous approach would require an approximation of this map through cut-off functions,
 see \cite[Section~5]{fornasier2018mean}).
 Defining $v^n(t,x):= f(x,\ubar{u}^n(t,x), \mu^n_t)$, using the growth condition on $f$ given in  \eqref{f:growth} and Young inequality
 we obtain that
\begin{equation*}
\begin{split}
\int_{\R^d} |x|^p\,\d \mu^n_t(x) &\leq \int_{\R^d} |x|^p\,\d \mu^n_0(x) + p\int_0^t \int_{\R^d} |v^n(s,x)||x|^{p-1} \,\d \mu^n_s(x)\,\d s \\
&\leq \int_{\R^d} |x|^p\,\d \mu^n_0(x) + Cp\int_0^t \int_{\R^d}\left( 1 + |x| + \mathrm{m}_p(\mu_s^n) \right)|x|^{p-1} \,\d \mu^n_s(x)\,\d s\\
&\leq \int_{\R^d} |x|^p\d \,\mu^n_0(x) + \tilde C\int_0^t \int_{\R^d} |x|^p \,\d \mu^n_s(x)\,\d s,
\end{split}
\end{equation*}
for some $\tilde C >0$ independent of $n$ and
 $t\in[0,T]$.
By Gronwall's inequality and \eqref{bound_mu0p} we get that
\begin{equation}\label{eq:bound_xp}
	\sup_{n \in \N}\sup_{t \in [0,T]}\int_{\R^d} |x|^p\d \mu^n_t(x) < +\infty. 
\end{equation}
Formally using the map $x\mapsto \psi(|x|^p)$ as a test function for the weak formulation of the continuity equation,
by the growth condition on $f$ in  \eqref{f:growth} and the bound \eqref{eq:bound_xp}, 
 we have
\begin{equation}\label{eq:psimom}
\begin{split}
\int_{\R^d} \psi(|x|^p)\,\d \mu^n_t(x) &\leq \int_{\R^d} \psi(|x|^p)\,\d\mu^n_0(x) + \int_0^t \int_{\R^d} |v^n(s,x)| |\nabla(\psi(|x|^p))| \,\d \mu^n_s(x)\,\d s \\
&\leq \int_{\R^d} \psi(|x|^p)\,\d \mu^n_0(x) + pC\int_0^t \int_{\R^d}\left(1 + |x| +\m_p(\mu^n_s))\right) \psi'(|x|^p)|x|^{p-1} \,\d \mu^n_s(x)\,\d s\\
&\leq \int_{\R^d} \psi(|x|^p)\,\d \mu^n_0(x) + C_1\int_0^t \int_{\R^d}\left(|x|^{p-1}  + |x|^p \right) \psi'(|x|^p)\,\d \mu^n_s(x)\,\d s
\end{split}
\end{equation}
for some $C_1>0$ independent of $n$ and $t\in[0,T]$.
Notice that by the monotonicity of $\psi'$, denoting by $B_1$ the unitary ball of $\R^d$, we have
\begin{equation*}
\begin{split}
  \int_{\R^d} |x|^{p-1} \psi'(|x|^p)\,\d \mu^n_s(x) &\leq  
\int_{B_1} |x|^{p-1} \psi'(|x|^p)\,\d \mu^n_s(x) +  \int_{\R^d\setminus B_1} |x|^{p} \psi'(|x|^p)\,\d \mu^n_s(x)\\
&\leq \psi'(1) +  \int_{\R^d} |x|^{p} \psi'(|x|^p)\,\d \mu^n_s(x).
\end{split}
\end{equation*} 
By the previous inequality and \eqref{doubling}, from \eqref{eq:psimom} we get
\begin{equation*}
\begin{split}
\int_{\R^d} \psi(|x|^p)\,\d \mu^n_t(x) 
&\leq \int_{\R^d} \psi(|x|^p)\,\d \mu^n_0(x) + C_2\int_0^t \int_{\R^d} \left( 1 +\psi(|x|^p)\right) \,\d \mu^n_s(x)\,\d s
\end{split}
\end{equation*}
 for some $C_2>0$ independent of $n$ and $t\in[0,T]$.
 By Gronwall's inequality and \eqref{eq:unifmompsi2} we obtain \eqref{eq:cpt_muN}.

Using Benamou-Brenier formula \eqref{B&B}, the growth condition on $f$ in  \eqref{f:growth},
for $s,t\in[0,T]$, $s \leq t$, it holds
\begin{equation*}
\begin{split}
W^p_p (\mu^n_s,\mu^n_t) &\leq \int_s^t \int_{\R^d} |v^n(r,x)|^p \,\d \mu^n_r(x) \d r \\
&\leq C\int_s^t \int_{\R^d} \left(1 + |x| + \mathrm{m}_p(\mu_r^n) \right)^p \,\d \mu_r^n(x) \d r,  
\end{split}
\end{equation*} 
for some $C>0$ independent of $n$, $s$ and $t$.
Using the bound \eqref{eq:bound_xp} we obtain \eqref{eq:BB}.

By Ascoli-Arzel\`a theorem in $\PP_p(\R^d)$ there exists $\mu \in C([0,T];\PP_p(\R^d))$ 
and a subsequence $\mu^n$ (not relabeled) such that
$\mu^n\to\mu$ in $C([0,T];\PP_p(\R^d))$.

\smallskip

For what concerns the weak compactness of $u^n$ (in the sense of convergence \eqref{convweaknu}),
we denote by  $\tilde{\mu}^n = \mu^n_t\otimes\cL_T\in\PP([0,T] \times \R^d)$ 
and $\tilde{\mu} = \mu_t\otimes\cL_T\in\PP([0,T] \times \R^d)$. 
From the convergence of $\mu^n$ to $\mu$ it follows that $\tilde{\mu}^n\to\tilde{\mu}$ weakly.
Defining 
$\gamma^n:= (i_{[0,T]\times\R^d}, \ubar{u}^n)_\sharp \tilde{\mu}^n \in \PP(([0,T] \times \R^d) \times U)$,
we observe that 
$\pi^1_\sharp \gamma^n = \tilde{\mu}^n\in\PP([0,T] \times \R^d)$ and 
$\pi^2_\sharp \gamma^n = \ubar{u}^n_\sharp\tilde{\mu}^n\in\PP(U)$, where $\pi^1: [0,T] \times \R^d \times U \to [0,T] \times \R^d$ is the projection on $[0,T] \times \R^d$ and $\pi^2: [0,T] \times \R^d \times U \to U$ is the projection on $U$. 
 
Since $\pi^1_\sharp \gamma^n$ weakly converges and $U$ is compact,
the families $\{\pi^1_\sharp \gamma^n\}_{n \in \N}$ and  $\{\pi^2_\sharp \gamma^n\}_{n \in \N}$ are tight.
Thanks to \cite[Lemma 5.2.2]{ambrosio2008gradient} it follows that $\{ \gamma^n \}_{n \in \N}$ is tight and,
by Prokhorov's Theorem, there exists $\gamma \in \PP([0,T] \times \R^d \times U)$ and a subsequence $\gamma^n$ (not relabeled) 
such that,
$\gamma^n \to \gamma$ weakly, as $n \to +\infty$.\\
Let $\phi\in C_c([0,T]\times\R^d;V')$. 
Recalling that $U$ is compact, using the continuous and bounded test function 
$(t,x,u)\mapsto  \langle \phi(t,x), u\rangle$, by the weak convergence of $\gamma^n$ to $\gamma$ we have
\begin{equation}\label{eq:youngun}
\begin{split}
 & \int_0^T\int_{\R^d}\langle\phi(t,x), \ubar u^n(t,x)\rangle \,\d \tilde\mu^n(t,x) =\\
& = \int_0^T\int_{\R^d}\int_U\langle \phi(t,x), u\rangle\,\d \gamma^n(t,x,u)
\to \int_0^T\int_{\R^d}\int_U \langle\phi(t,x), u\rangle\,\d \gamma(t,x,u),
\end{split}
\end{equation}  
as $n\to+\infty$.
Using Theorem \ref{thm:disint} (specifically Remark \ref{rem:disint_proj}) and observing that $\pi^1_\sharp \gamma = \tilde{\mu}$, we disintegrate $\gamma$ with respect to $\pi^1$ to get
\begin{equation*}
\int_0^T\int_{\R^d}\int_U \langle \phi(t,x), u\rangle \,\d \gamma(t,x,u)
= \int_0^T\int_{\R^d}\int_U\langle \phi(t,x), u\rangle \,\d \gamma_{t,x}(u)\,\d\tilde\mu(t,x).
\end{equation*}  
We define now $\ubar u: [0,T] \times \R^d \to U$ by
\begin{equation}\label{eq:baricentreu}
\ubar{u}(t,x) := \int_U u \, \d \gamma_{t,x}(u), \qquad  \forall\, (t,x) \in [0,T] \times \R^d,
\end{equation}
where the integral in \eqref{eq:baricentreu} is a Bochner integral.
Since the map $(t,x) \in [0,T] \times \R^d \mapsto \gamma_{t,x}\in\PP(U)$ is a Borel map, then
 $\ubar u\in\Bor([0,T] \times \R^d;U)$.
We call the map $\ubar u$ the barycentric projection of $\gamma$ with respect to $\pi^1_\sharp\gamma$.
Since the Bochner integral commutes with continuous linear functionals, it holds 
\begin{equation*}
\int_0^T\int_{\R^d}\int_U\langle\phi(t,x), u\rangle\,\d \gamma_{t,x}(u)\,\d\tilde\mu(t,x) 
=\int_0^T\int_{\R^d}\langle\phi(t,x), \ubar u(t,x)\rangle\,\d\tilde\mu(t,x).
\end{equation*}  
Using \eqref{eq:youngun} we obtain the convergence of $\ubar u^n \to \ubar u$ in the sense of \eqref{convweaknu}.

\smallskip

In order to prove that $(\mu, \ubar u) \in \cA_{\E}(\mu_0)$ we show that \eqref{eq:systemE} is satisfied.
Let $\varphi \in C^1_c([0,T] \times \R^d)$.
Since  $(\mu^n,\ubar u^n) \in \cA_{\E}(\mu^n_0)$, for every $t \in [0,T]$ it holds
\begin{equation}\label{eq:mu_n_continuity}
	\int_{\R^d} \varphi(t,x)\,\d \mu^n_t(x) - \int_{\R^d} \varphi(0,x)\,\d \mu^n_0(x) =
	 \int_0^t \int_{\R^d} \left( \partial_s \varphi(s,x)  +  v^n(s,x)\cdot \nabla \varphi(s,x) \right) \,\d \mu^n_s(x) \,\d s. 
\end{equation} 
By the convergence $\mu^n \to\mu$ in $	C([0,T];\PP_p(\R^d))$ we immediately pass to the limit, as $n \to +\infty$, in the left hand side of equation \eqref{eq:mu_n_continuity} as well as on the term $\int_0^t \int_{\R^d} \partial_s \varphi(s,x) \,\d \mu^n_s(x) \,\d s$.
Finally, let us rewrite
\begin{equation*}
\begin{split}
 &\int_0^t \int_{\R^d}  v^n(s,x)\cdot \nabla \varphi(s,x) \,\d \mu^n_s(x) \,\d s\\
 &=\int_0^t \int_{\R^d}  f(x, \ubar{u}^n(s,x), \mu^n_s) \cdot  \nabla \varphi(s,x) \, \d \mu^n_s(x)\,\d s \\
&= \int_{[0,t] \times \R^d \times U} f(x, u, \mu^n_s)\cdot  \nabla \varphi(s,x) \, \d \gamma^n(s,x,u) \\
&= \int_{[0,t] \times \R^d \times U} ( f(x, u, \mu^n_s) - f(x, u, \mu_s))\cdot   \nabla \varphi(s,x) \, \d \gamma^n(s,x,u) \\
&\quad + \int_{[0,t] \times \R^d \times U}  f(x, u, \mu_s)\cdot  \nabla \varphi(s,x)  \, \d \gamma^n(s,x,u).\\
\end{split}
\end{equation*}  
By \eqref{eq:Lipf}, the first term on the right hand side can be estimated by 
\begin{equation*}
\begin{split}
&\left|\int_{[0,t] \times \R^d \times U} \langle f(x, u, \mu^n_s) - f(x, u, \mu_s),  \nabla \varphi(s,x) \rangle \, \d \gamma^n(s,x,u)\right| \\
&\quad \leq L\sup_{r \in [0,T]}W_p(\mu^n_r, \mu_r)\int_{[0,t] \times \R^d \times U} |\nabla \varphi(s,x)| \, \d \gamma^n(s,x,u) \\
&\quad = L\sup_{r \in [0,T]}W_p(\mu^n_r, \mu_r)\int_{[0,t] \times \R^d} |\nabla \varphi(s,x)| \, \d \mu^n_s(x)\,\d s,
\end{split}
\end{equation*} 
which goes to zero as $n \to +\infty$ by the convergence $\mu^n\to\mu$ in $C([0,T];\PP_p(\R^d))$.
Hence
\begin{equation*}
\begin{split}
\lim_{n \to +\infty} &\int_0^t \int_{\R^d}  v^n(s,x)\cdot \nabla \varphi(s,x) \,\d \mu^n_s(x) \,\d s \\
=\lim_{n \to +\infty} &\int_{[0,t] \times \R^d \times U} f(x, u, \mu_s)\cdot \nabla \varphi(t,x) \, \d \gamma^n(s,x,u)\\
&=\int_{[0,t] \times \R^d \times U} f(x, u, \mu_s)\cdot \nabla \varphi(t,x) \, \d \gamma(s,x,u),\\
\end{split}
\end{equation*}  
by the weak convergence of $\gamma^n$ to $\gamma$.
Recall that, by the Convexity Assumption \ref{CA}, for any $(x,s)\in\R^d\times[0,T]$ the  map $u \mapsto f(x,u,\mu_s)$ is affine. 
Thus, using that $ \gamma_{s,x}$ is a probability measure, we have
\begin{equation*}
\int_U  f(x, u, \mu_s) \, \d \gamma_{s,x}(u) = f\left(x, \int_U u \,\d \gamma_{s,x}(u) , \mu_s\right), \quad \text{ for } \tilde \mu\text{-a.e. } (s,x) \in [0,T] \times \R^d.
\end{equation*}  
Hence we get 
\begin{equation*}
\begin{split}
\int_{[0,t] \times \R^d} &\int_U  f(x, u, \mu_s)\cdot \nabla \varphi(s,x) \, \d \gamma_{s,x}(u) \,\d \tilde \mu(s,x)\\
&=\int_{[0,t] \times \R^d}  f\left(x, \int_U u \,\d \gamma_{s,x}(u) , \mu_s\right)\cdot  \nabla \varphi(s,x) \, \d \tilde \mu(s,x)\\
&=\int_0^t \int_{\R^d}  f\left(x, \ubar u(s,x) ,\mu_s\right)\cdot  \nabla \varphi(s,x) \, \d \mu_s(x) \,\d s.\\
\end{split}
\end{equation*}  
Defining $v(t,x):= f(x,\ubar{u}(t,x), \mu_t)$, we have proved that
\begin{equation*}
\begin{split}
\lim_{n \to +\infty} &\int_0^t \int_{\R^d}  v^n(s,x)\cdot \nabla \varphi(s,x) \,\d \mu^n_s(x) \,\d s \\
&=\int_0^t \int_{\R^d}  v(s,x)\cdot \nabla \varphi(s,x) \,\d \mu_s(x) \,\d s
\end{split}
\end{equation*}  
and this concludes the proof.
\end{proof}

\begin{proposition}[Lower semicontinuity for convex $J_\E$]\label{prop:lscJE}
Let $\S = (U,f,\cC,\cC_T)$ satisfy the Convexity Assumption \ref{CA}.
If $(\mu^n,\ubar u^n)$ converges to $(\mu,\ubar u)$ according to Definition \ref{def:nu_converges}, then 
 \begin{equation}\label{lscJE}
 	\liminf_{n\to+\infty} J_\E(\mu^n,\ubar u^n)\ge J_\E(\mu,\ubar u).
 \end{equation}
\end{proposition}
\begin{proof}
Denoting by  $\tilde{\mu}^n = \mu^n_t\otimes\cL_T\in\PP([0,T] \times \R^d)$ 
and $\tilde{\mu} = \mu_t\otimes\cL_T\in\PP([0,T] \times \R^d)$,
we define $\gamma^n:= (i_{[0,T]\times\R^d}, \ubar{u}^n)_\sharp \tilde{\mu}^n \in \PP(([0,T] \times \R^d) \times U)$.
Reasoning as in the proof of Proposition \ref{prop:compactness}, we obtain that there exists 
$\gamma \in \PP(([0,T] \times \R^d) \times U)$ such that $\pi^1_\sharp \gamma= \tilde{\mu}$
and, up to subsequences, $\gamma^n\to\gamma$ weakly as $n\to+\infty$.
Moreover, defining $\ubar u_{b}$ by \eqref{eq:baricentreu}, up to subsequences, 
 $(\mu^n,\ubar u^n)$ converges to $(\mu,\ubar u_b)$ according to Definition \ref{def:nu_converges}.
Then, 
\begin{equation*}
	\int_0^T\int_{\R^d}\langle\phi(t,x), \ubar u(t,x)\rangle\,\d\tilde\mu(t,x) 
	= \int_0^T\int_{\R^d}\langle\phi(t,x), \ubar u_b(t,x)\rangle\,\d\tilde\mu(t,x), \quad \forall\,\phi\in C_c([0,T]\times\R^d;V'),
\end{equation*}  
which implies that $ \ubar u_b(t,x)= \ubar u(t,x)$ for $\tilde\mu$-a.e. $(t,x)\in [0,T]\times\R^d$.

Let $\cK \subset \PP_p(\R^d)$ be a compact set containing
$\mu^n_t,\mu_t$ for any $n\in\N$ and $t\in[0,T]$.
Then, for any compact $K \subset \R^d$, thanks to the continuity of $\cC$, there exists a modulus of continuity $\alpha:[0,+\infty)\to[0,+\infty)$
for the restriction of $\cC$ to the compact set $K\times U \times\cK$ so that 
$$\sup_{(t,x,u)\in [0,T] \times K \times U}|\cC(x,u,\mu^n_t) - \cC(x,u,\mu_t)| \leq \alpha \Big( \sup_{t\in[0,T]}W_p(\mu_t^{n},\mu_t)\Big) \to 0 \quad \text{as } n\to+\infty.$$
Then, taking into account that $\cC \ge 0$, we get 
\begin{equation*}
\begin{split}
&\liminf_{n \to +\infty}\int_0^T\int_{\R^d}\mathcal C(x,\ubar u^n(t,x),\mu^n_t)\,\d\mu^n_t(x)\,\d t \\
&\geq \liminf_{n \to +\infty}\int_{[0,T] \times K \times U}\mathcal C(x,u,\mu^n_t)\,\d\gamma^n(t,x,u)\\
&\geq \int_{[0,T] \times K \times U}\mathcal C(x,u,\mu_t)\,\d\gamma(t,x,u).
\end{split}
\end{equation*}
Since $K$ is arbitrary we obtain 
\begin{equation}\label{eqliminf}
\liminf_{n \to +\infty}\int_0^T\int_{\R^d}\mathcal C(x,\ubar u^n(t,x),\mu^n_t)\,\d\mu^n_t(x)\,\d t 
\geq \int_{[0,T] \times \R^d \times U}\mathcal C(x,u,\mu_t)\,\d\gamma(t,x,u).
\end{equation}
Denoting $\gamma_{t,x}$ the disintegration of $\gamma$ with respect to $\pi^1$ (as  in the proof of Proposition \ref{prop:compactness}), the convexity of the map $u \mapsto \cC(x, u, \mu_t)$ for any $(t,x)\in[0,T]\times\R^d$
and Jensen's inequality yield
\begin{equation*}
\begin{split}
 \int_{[0,T] \times \R^d \times U}\mathcal C(x,u,\mu_t)\,\d\gamma(t,x,u)
&=\int_0^T \int_{\R^d} \int_U\mathcal C(x,u,\mu_t)\,\d\gamma_{t,x}(u)\,\d \mu_t(x) \,\d t \\
&\geq \int_0^T \int_{\R^d}\mathcal C\left(x,\int_U u \, \d\gamma_{t,x}(u),\mu_t\right)\,\d \mu_t(x) \,\d t \\
&= \int_0^T \int_{\R^d}\mathcal C\left(x,\ubar u_b(t,x),\mu_t\right)\,\d \mu_t(x) \,\d t\\
&= \int_0^T \int_{\R^d}\mathcal C\left(x,\ubar u(t,x),\mu_t\right)\,\d \mu_t(x) \,\d t.
\end{split}
\end{equation*}
By the continuity of $\cC_T$, using the same argument of the proof of
\eqref{eqliminf},
 we obtain
\begin{equation*}
\begin{split}
\liminf_{n \to +\infty}\int_{\R^d}\mathcal C_T(x,\mu^n_T)\,\d\mu^n_T(x) \geq
\int_{\R^d} \mathcal C_T(x,\mu_T)\,\d\mu_T(x).
\end{split}
\end{equation*}
\end{proof}

{Propositions \ref{prop:compactness} and \ref{prop:lscJE} give immediately the existence of optimizers for our optimal control problem in Eulerian formulation.}

\begin{theorem}[Existence of minimizers for convex \E]\label{thm:minE}
Let $\S=(U,f,\cC,\cC_T)$ satisfy the Convexity Assumption \ref{CA}.
If $\mu_0 \in \PP_p(\R^d)$,
then there exists $(\mu,\ubar u)\in\mathcal A_{\E}(\mu_0)$
such that $$J_{\E}(\mu,\ubar u)=V_{\E}(\mu_0).$$
\end{theorem}

As a consequence, we derive the lower semicontinuity of the value function for the Eulerian problem.

\begin{proposition}[Lower semicontinuity of $V_{\E}$]\label{prop:lscVE}
Let $\S=(U,f,\cC,\cC_T)$ satisfy the Convexity Assumption \ref{CA}
and $\mu_0 \in \PP_p(\R^d)$. 
If $\{\mu_0^n\}_{n\in\N}\subset \PP_p(\R^d)$ is a sequence such that $W_p(\mu_0^n, \mu_0) \to 0$ as $n \to +\infty$,
then  
$$\liminf_{n \to +\infty} V_{\E}(\mu_0^n)\geq V_\E(\mu_0).$$
\end{proposition}
\begin{proof}
 From Theorem \ref{thm:minE} there exists a sequence $(\mu^n,\ubar u^n) \in \mathcal{A}_{\E}(\mu_0^n)$ 
 such that $V_{\E}(\mu_0^n)=J_\E(\mu^n,\ubar u^n)$.
By Proposition \ref{prop:compactness} there exists $(\mu,\ubar u) \in \mathcal{A}_{\E}(\mu_0)$ such that, up to subsequences, 
$(\mu^n,\ubar u^n)$ converges to $(\mu,\ubar u)$ according to Definition \ref{def:nu_converges}.
Then, using Proposition \ref{prop:lscJE}, we have 
\[
\liminf_{n \to +\infty} V_{\E}(\mu_0^n) = \liminf_{n \to +\infty} J_{\E}(\mu^n,\ubar u^n) 
\geq J_{\E}(\mu,\ubar u) \geq V_{\E}(\mu_0).  
\]
\end{proof}

\section{Kantorovich optimal control problem and equivalence with the Eulerian}\label{sec:K}

In this section, we provide a further formulation of optimal control problems which we call \emph{Kantorovich formulation} in analogy with the terminology used in optimal transport theory. 
This formulation acts as a bridge between the Lagrangian and the Eulerian formulations and it is
based on the representation of solutions of the continuity equation by superposition of continuous curves in
$\Gamma_T=C([0,T];\R^d)$ (see Theorem \ref{thm:sup_princ}).
This formulation turns out to be equivalent to the Eulerian one and it will be useful in Section \ref{sec:E=L} to prove the equivalence between the Eulerian 
and the Lagrangian problems.\\

\noindent We recall that, for any $t\in[0,T]$, $e_t:\Gamma_T\to\R^d$ denotes the evaluation map  $e_t(\gamma):=\gamma(t)$.

\begin{definition}[Kantorovich optimal control problem (\K)]\label{def:K}
Let $\S=(U,f,\cC,\cC_T)$  satisfy Assumption \ref{BA}. 
Given $\mu_0\in\PP_p(\R^d)$, we say that $(\eeta,u)\in\cA_{\K}(\mu_0)$ if
\begin{itemize}
\item[(i)] $u\in\Bor([0,T]\times\Gamma_T;U)$;
\item[(ii)] $\eeta\in \PP(\Gamma_T)$,  $(e_0)_\sharp\eeta = \mu_0$ and,
defining  $\mu_t:=(e_t)_\sharp\eeta$ for all $t\in[0,T]$, 
\begin{equation}\label{eq:momp}
	\int_0^T\int_{\R^d}|x|^p\,\d\mu_t(x)\,\d t <+\infty.
\end{equation}
$\eeta$ is concentrated on the set of absolutely continuous solutions  of the differential equation
\[\dot\gamma(t)=f(\gamma(t),u(t,\gamma),\mu_t),\qquad \text{for $\cL_T$-a.e. }t\in[0,T].\]
\end{itemize}
We define the \emph{cost functional} $J_{\K}:\PP(\Gamma_T)\times{\Bor([0,T]\times\Gamma_T; U)}\to[0,+\infty]$ by
\begin{equation*}
J_{\K}(\eeta,u):=\int_0^T\int_{\Gamma_T}\cC(\gamma(t),u(t,\gamma),\mu_t)\,\d\eeta(\gamma)\,\d t+\int_{\R^d} \cC_T(x,\mu_T)\,\d\mu_T(x),
\end{equation*}
and the \emph{value function} $V_{\K}:\PP_p(\R^d)\to[0,+\infty)$ by
\begin{equation*}
V_{\K}(\mu_0):=\inf\{J_{\K}(\eeta,u)\,:\,(\eeta,u)\in\cA_{\K}(\mu_0)\}.
\end{equation*}
\end{definition}

\begin{remark}\label{rem:etaACp}
We observe that, by the growth condition of $f$ in \eqref{f:growth} and condition \eqref{eq:momp}, $\eeta$ is actually concentrated on $\AC^p([0,T]; \R^d)$, indeed 
\begin{equation*}
\begin{split}
\int_0^T \int_{\Gamma_T} |\dot \gamma(t)|^p\, \d \eeta(\gamma) \, \d t &
\leq \int_0^T \int_{\Gamma_T} |f(\gamma(t),u(t,\gamma), \mu_t)|^p\, \d \eeta(\gamma) \, \d t \\
&\leq C\int_0^T \int_{\Gamma_T} |1 + |\gamma(t)| + \mathrm{m}_p(\mu)|^p\, \d \eeta(\gamma) \, \d t\\
&\leq \tilde C \Big( 1 +  \int_0^T\int_{\R^d}|x|^p \, \d \mu_t(x) \d t\Big).
\end{split}
\end{equation*}
Hence, by Fubini theorem, for $\eeta$-a.e. $\gamma \in \Gamma_T$, $\dot \gamma \in L^p([0,T]; \R^d)$. 

Moreover, thanks to \eqref{eq:momp} and the growth condition \eqref{eq:growthC},  $J_\K(\eeta,u) < +\infty$ for every $(\eeta,u) \in \cA_{\K}(\mu_0)$ (the proof of $\int_{\R^d}|x|^p\,\d\mu_T(x) <+\infty$
follows by items (i) and (ii) and the same argument used to show \eqref{eq:bound_xp} in Proposition \ref{prop:compactness}).\\
Finally, the value function $V_\K$ is well defined since $\cA_\K$ is non empty (see Propostion \ref{prop:E>K}).
\end{remark}

The aim of this section is to prove the existence of minimizers for the Kantorovich optimal control problem under the Convexity Assumptions \ref{CA} and to show the equivalence with the Eulerian formulation of the problem described in Section \ref{sec_eulerian}. In particular, we get the equality of the corresponding value functions. This is the content of the following theorem.
\begin{theorem}\label{cor:E=K}
Let $\S = (U,f,\cC,\cC_T)$ satisfy the Convexity Assumption \ref{CA}.
If $\mu_0\in \PP_p(\R^d)$, then 
there exist  $(\boldsymbol\eta,u)\in\mathcal A_\K(\mu_0)$ and $(\mu,\ubar u)\in\mathcal A_\E(\mu_0)$ such that 
\[J_\K(\eeta,u)=J_\E(\mu,\ubar u)=V_{\K}(\mu_0)=V_{\E}(\mu_0).\] 
\end{theorem}
The proof of Theorem \ref{cor:E=K} follows by the combination of Theorem \ref{thm:minE} and Propositions \ref{prop:E>K}, \ref{prop:K>E} below.

\medskip

\begin{proposition}\label{prop:E>K}
Let $\S:=(U,f,\mathcal C,\mathcal C_T)$ satisfy Assumption \ref{BA}.
Let $\mu_0\in \PP_p(\R^d)$. If $(\mu,\ubar u)\in\cA_{\E}(\mu_0)$, then there exists $(\eeta,u)\in\cA_{\K}(\mu_0)$ such that
 $J_{\K}(\eeta,u)=J_{\E}(\mu,\ubar u)$.
In particular $\cA_{\K}(\mu_0)\not=\emptyset$ and  $V_\K(\mu_0)\le V_\E(\mu_0)$.
\end{proposition}

\begin{proof}
Let $(\mu,\ubar u) \in \cA_\E(\mu_0)$.
Applying Theorem \ref{thm:sup_princ} to $\mu$  and $v_t(x)=f(x,\ubar u(t,x),\mu_t)$ we get $\eeta \in \PP(\Gamma_T)$ such that 
$(e_t)_\sharp\eeta=\mu_t$ for every $t\in[0,T]$ and $\eeta$ is concentrated on the absolutely continuous solutions of $\dot \gamma(t) = v_t(\gamma(t))$.
Condition \eqref{eq:momp} is automatically satisfied in view of the fact that $(\mu,\ubar u) \in \cA_\E(\mu_0)$.
Then, for every $(t,\gamma)\in[0,T]\times\Gamma_T$ we define $u(t,\gamma):=\ubar u(t,\gamma(t))$ so that $u$ is Borel measurable and we have  $(\eeta,u)\in\cA_{\K}(\mu_0)$. 
Finally, from the evaluation $(e_t)_\sharp \eeta = \mu_t$ it holds that $J_{\K}(\eeta,u)=J_{\E}(\mu,\ubar u)$.
\end{proof}
Under the Convexity Assumption \ref{CA} it also holds that $V_\K(\mu_0)\ge V_\E(\mu_0)$. 

\begin{proposition}\label{prop:K>E}
Let $\S = (U,f,\cC,\cC_T)$ satisfy the Convexity Assumption \ref{CA}.
Let $\mu_0\in \PP_p(\R^d)$. If $(\eeta,u)\in\cA_{\K}(\mu_0)$, then there exists 
$(\mu,\ubar u)\in\cA_{\E}(\mu_0)$ such that
 $J_{\E}(\mu,\ubar u)\le J_{\K}(\eeta,u)$.
In particular $V_\K(\mu_0)\ge V_\E(\mu_0)$.
\end{proposition}

\begin{proof}
Let $(\eeta,u)\in\cA_{\K}(\mu_0)$. We firstly define $\mu_t:=(e_t)_\sharp\eeta$, for every $t\in[0,T]$.
We introduce the continuous evaluation map $e:[0,T]\times\Gamma_T\to[0,T]\times\R^d$ by setting 
$$e(t,\gamma):=(t,\gamma(t))=(i_{[0,T]}(t),e_t(\gamma))$$
and we  denote by $(t,x)\mapsto\tilde\eeta_{t,x}$ the Borel map obtained  from the disintegration of $\cL_T\otimes\eeta$
with respect to $e$, see Theorem \ref{thm:disint}.
Then  we define the function $\ubar u:[0,T]\times\R^d\to U$ by
\[\ubar u(t,x):=\int_{[0,T]\times\Gamma_T}u(t,\gamma)\,\d\tilde\eeta_{t,x}(\gamma).\]
Notice that $\ubar u$ is Borel measurable thanks to the Borel measurability of $(t,x) \mapsto \tilde \eeta_{t,x}$.
The measure $\tilde\eeta_{t,x}$ is concentrated on $\{t\}\times\{\gamma:\gamma(t)=x\}$, so that 
$\tilde\eeta_{t,x}=\delta_t\otimes\eeta_{t,x}$, where for any $t\in[0,T]$, 
the function $x\mapsto\eeta_{t,x}$ is the Borel map given by the disintegration of $\eeta$
with respect to the continuous map $e_t$.
Hence we have also that
\begin{equation}\label{ubaru}
\ubar u(t,x)=\int_{\Gamma_T}u(t,\gamma)\,\d\eeta_{t,x}(\gamma).
\end{equation}

Defining the set
$A:=\{(t,\gamma)\in[0,T]\times\Gamma_T: \exists \, \dot\gamma(t) \, \text{ and } \dot\gamma(t) = f(\gamma(t),u(t,\gamma),\mu_t) \}$,
by item (ii) of Definition \ref{def:K} we have that $\cL_T\otimes\eeta(([0,T] \times \Gamma_T )\setminus A)=0$. Then, for $\cL_T$-a.e. $t\in[0,T]$ we have
\[ \dot\gamma(t)=f(\gamma(t),u(t,\gamma),\mu_t), \qquad \text{for $\eeta$-a.e. }\gamma\in\Gamma_T.\]
Let $\varphi\in C^1_c(\R^d;\R)$.
For any $s,t\in[0,T]$, $s<t$, we have
\begin{equation}\label{eq:continuity_KE}
\begin{aligned}
	\int_{\R^d}\varphi(x)\,\d\mu_t(x) - \int_{\R^d}\varphi(x)\,\d\mu_s(x)  &=
	\int_{\Gamma_T}(\varphi(\gamma(t))- \varphi(\gamma(s)))\,\d\eeta(\gamma)\\
	&=\int_{\Gamma_T}\int_s^t \frac{\d}{\d r}\varphi(\gamma(r))\,\d r\,\d\eeta(\gamma)\\
	&=\int_{\Gamma_T}\int_s^t \nabla\varphi(\gamma(r))\cdot\dot\gamma(r)\,\d r\,\d\eeta(\gamma)\\
	&=\int_s^t \int_{\Gamma_T}\nabla\varphi(\gamma(r))\cdot f(\gamma(r),u(r,\gamma),\mu_r)\,\d\eeta(\gamma)\,\d r.
\end{aligned}
\end{equation}
Using the growth condition of $f$ in \eqref{f:growth} we have
\begin{align*}
	\left| \int_{\R^d}\varphi(x)\,\d\mu_t(x) - \int_{\R^d}\varphi(x)\,\d\mu_s(x)\right|  \leq
	C\|\nabla\varphi\|_\infty \int_s^t \int_{\Gamma_T}(1+|\gamma(r)|+\m_p(\mu_r))\,\d\eeta(\gamma)\,\d r.
\end{align*}
By \eqref{eq:momp} the map $r\mapsto  \int_{\Gamma_T}(1+|\gamma(r)|+\m_p(\mu_r))\,\d\eeta(\gamma)$ belongs to $L^1(0,T)$
and then the map $t\mapsto \int_{\R^d}\varphi(x)\,\d\mu_t(x)$ is absolutely continuous.

Thanks to \eqref{eq:continuity_KE}, for $\cL_T$-a.e. $t\in[0,T]$ we have
\begin{align*}
\frac{\d}{\d t}\int_{\R^d}\varphi(x)\,\d\mu_t(x)
=\int_{\Gamma_T}\nabla\varphi(\gamma(t))\cdot f(\gamma(t),u(t,\gamma),\mu_t)\,\d\boldsymbol\eta(\gamma).
\end{align*}
Using the affinity of $f$, the disintegration of $\eeta$ with respect to $e_t$, 
recalling that  $\eeta_{t,x}$ is concentrated on $\{\gamma:\gamma(t)=x\}$, and by \eqref{ubaru} we obtain
\begin{align*}
&\int_{\Gamma_T}\nabla\varphi(\gamma(t))\cdot f(\gamma(t),u(t,\gamma),\mu_t)\,\d\eeta(\gamma)\\
&=\int_{\R^d}\int_{\Gamma_T}\nabla\varphi(\gamma(t))\cdot f(\gamma(t),u(t,\gamma),\mu_t)\,\d\eeta_{t,x}(\gamma)\,\d\mu_t(x)\\
&=\int_{\R^d}\int_{\Gamma_T}\nabla\varphi(x)\cdot f(x,u(t,\gamma),\mu_t)\,\d\eeta_{t,x}(\gamma)\,\d\mu_t(x)\\
&=\int_{\R^d}\nabla\varphi(x)\cdot f\left(x,\int_{\Gamma_{T}}u(t,\gamma)\,\d\eeta_{t,x}(\gamma),\mu_t\right)\,\d\mu_t(x)\\
&=\int_{\R^d}\nabla\varphi(x)\cdot f(x,\ubar u(t,x),\mu_t)\,\d\mu_t(x).
\end{align*}
Then $\mu$ satisfies the continuity equation $\de_t\mu_t+\div(v_t\mu_t)=0$ for the vector field
$v_t(x)= f(x,\ubar u(t,x),\mu_t)$ in the sense of distributions (see e.g. \cite[equation 8.1.4]{ambrosio2008gradient}).
Since $|v_t(x)|^p\leq \tilde C(1+|x|^p+\m_p^p(\mu_t))$, from \eqref{eq:momp} it follows that $t\mapsto \|v_t\|_{L^p_{\mu_t}(\R^d;\R^d)}$ 
 belongs to $L^p(0,T)$ and then $\mu \in \AC^p([0,T];\PP_p(\R^d))$.
Hence, $(\mu,\ubar u)\in\cA_{\E}(\mu_0)$.

Finally, by the convexity of $\cC$ with respect to $u$ and Jensen's inequality we obtain
\begin{equation*}
\begin{split}
&\int_0^T \int_{\Gamma_T} \cC(\gamma(t),u(t,\gamma),\mu_t) \, \d \boldsymbol\eta(\gamma)\, \d t \\
&= \int_0^T \int_{\R^d}\int_{\Gamma_{T}}\cC(x,u(t,\gamma),\mu_t) \, \d \eeta_{t,x}(\gamma)\d \mu_t(x)\, \d t \\
&\ge \int_0^T \int_{\R^d} \cC\left( x, \int_{\Gamma_{T}} u(t,\gamma)\,\d\eeta_{t,x}(\gamma),\mu_t\right) \,\d \mu_t(x) \,\d t\\
&=  \int_0^T \int_{\R^d} \cC\left( x, \ubar u(t,x),\mu_t\right) \,\d \mu_t(x)\, \d t.
\end{split}
\end{equation*}
Hence we obtain $J_{\K}(\eeta,u) \ge J_{\E}(\mu,\ubar u)$.
\end{proof}

\section{Equivalence of Eulerian and Lagrangian problems}\label{sec:E=L}
In this Section we study the equivalence between the Eulerian and Lagrangian formulations of the optimal control problem. We anticipate here the main results of this section.

\begin{theorem}\label{cor:VL=VE}
Let $\S=(U,f,\cC,\cC_T)$  satisfy the Convexity Assumption \ref{CA} and $(\Omega, \frB, \P)$ be a standard Borel space such that $\P$ is without atoms.
If $X_0\in L^p(\Omega;\R^d)$, then 
$$V_{\pL}(X_0)= V_{\E}((X_0)_\sharp\P).$$
In particular, given $X_0,X_0'\in L^p(\Omega;\R^d)$ s.t. $(X_0)_\sharp\P=(X_0')_\sharp\P$, then
\[V_{\pL}(X_0)=V_{\pL}(X_0').\]
\end{theorem}
The proof of Theorem \ref{cor:VL=VE} follows immediately by Proposition \ref{prop:E<L}, Theorem \ref{prop:K>L} and Theorem \ref{cor:E=K}.

\medskip

We stress again that the Convexity Assumption \ref{CA} is sufficient to prove the existence of a minimizer for the Eulerian and Kantorovich optimal control problems (see Theorems \ref{thm:minE} and \ref{cor:E=K}). However in general, even assuming the Convexity Assumption \ref{CA}, the Lagrangian optimal control problem could not have minimizers as we show in Section \ref{sec:counterexample}.

\medskip

If we remove the Convexity Assumption \ref{CA}, we can still give the following equivalence result.
\begin{theorem}\label{cor:VL=VE'}
Let $\S=(U,f,\cC,\cC_T)$  satisfy Assumption \ref{BA} and $\S'=(\UU,\FF,\CC,\CC_T)$ as in Definition \ref{def:relax_setting}.
Let $(\Omega, \frB, \P)$ be a standard Borel space such that $\P$ is without atoms.
Let $\L, \RL$ be the Lagrangian and Relaxed Lagrangian problems associated to $\S$.
Let also  $\L',\E'$ the Lagrangian and Eulerian problems associated to $\S'$. 
If $X_0\in L^p(\Omega;\R^d)$, then 
$$V_{\pL}(X_0)=V_{\RL}(X_0)=V_{\pL'}(X_0)= V_{\E'}((X_0)_\sharp\P)=V_{\K'}((X_0)_\sharp\P).$$
\end{theorem}

The proof of Theorem \ref{cor:VL=VE'} is postponed at the end of Section \ref{sec:LEK}.

\subsection{Comparison between $\pL$, $\E$ and $\K$}\label{sec:LEK}

We start by comparing the Eulerian and Lagrangian problems under the Convexity Assumption \ref{CA}.
Assuming  $\P$ without atoms, we further exhibit the equivalence between the associated value functions exploiting the Kantorovich formulation introduced in Section \ref{sec:K}.

The following is a first comparison between the Eulerian and Lagrangian problems.
\begin{proposition}\label{prop:E<L}
Let $\S=(U,f,\cC,\cC_T)$  satisfy the Convexity Assumption \ref{CA} and $(\Omega, \frB, \P)$ be a probability space.
If  $X_0\in L^p(\Omega;\R^d)$ and  $(X,u)\in\cA_{\pL}(X_0)$, then there exists $(\mu,\ubar u)\in\cA_{\E}((X_0)_\sharp\P)$
such that $J_{\pL}(X,u) \geq J_{\E}(\mu,\ubar u)$. In particular, $V_{\pL}(X_0)\ge V_{\E}((X_0)_\sharp\P)$.
\end{proposition}
\begin{proof}
Let $(X,u)\in\cA_{\pL}(X_0)$. 
We firstly define $\mu_t:= (X_t)_\sharp \P$, for every $t \in [0,T]$.
Thanks to Proposition \ref{prop:equivSpaces} in Appendix \ref{appendix_A}, $X \in \AC^p([0,T]; L^p(\Omega;\R^d))$ so that  
$\mu \in \AC^p([0,T];\PP_p(\R^d))$.

We define $\sigma\in \PP([0,T] \times \Omega\times U)$ by $\sigma:=\delta_{u(t,\omega)}\otimes\P \otimes \cL_T$
and $\theta\in \PP([0,T]\times \R^d\times U)$ by $\theta:=(i_{[0,T]},X_t,i_U)_\sharp\sigma$.
Let $\pi^{1,2}: [0,T]\times  \R^d \times U \to [0,T]\times \R^d$ be the projection map $(t,x,u) \mapsto (t,x)$, observe that $\pi^{1,2}_\sharp\theta=\mu_t \otimes \cL_T$. 
If we denote $\theta_{t,x}$ the disintegration of $\theta$ with respect to $\pi^{1,2}$, then we have $\theta=\theta_{t,x}\otimes\mu_t \otimes \cL_T$.
We define now $\ubar u:[0,T]\times \R^d\to U$  by  
\begin{equation}\label{eq:ubaru_theta}
\ubar u(t,x) := \int_U u \,\d \theta_{t,x}(u).
\end{equation}
Thanks to Theorem \ref{thm:disint}, the map $(t,x)\mapsto\theta_{t,x}$ is Borel measurable, so that $\ubar u \in \Bor([0,T] \times \R^d;U)$. 

The rest of the proof follows the same line of the proof of Proposition \ref{prop:K>E}. We write the details for the reader's convenience.

Defining the set
$A:=\{(t,\omega)\in[0,T]\times\Omega: \exists \,\dot X_t(\omega) \text{ and } \dot X_t(\omega) = f(X_t(\omega),u(t,\omega),\mu_t) \}$,  by item (ii) of Definition \ref{def:L}, we have that $\cL_T\otimes\P(([0,T] \times \Omega \setminus A)=0$. Then, for $\cL_T$-a.e. $t\in[0,T]$ it holds
\[ \dot X_t(\omega)=f(X_t(\omega),u(t,\omega),\mu_t), \qquad \text{for $\P$-a.e. }\omega\in\Omega.\]
Let $\varphi\in C^1_c(\R^d;\R)$.
For any $s,t\in[0,T]$, $s<t$, we have
\begin{equation}\label{eq:distr_mu}
\begin{aligned}
	\int_{\R^d}\varphi(x)\,\d\mu_t(x) - \int_{\R^d}\varphi(x)\,\d\mu_s(x)  &=
	\int_{\Omega}(\varphi(X_t(\omega))- \varphi(X_s(\omega)))\,\d\P(\omega)\\
	&=\int_{\Omega}\int_s^t \frac{\d}{\d r}\varphi(X_r(\omega))\,\d r\,\d\P(\omega)\\
	&=\int_{\Omega}\int_s^t \nabla\varphi(X_r(\omega))\cdot\dot X_r(\omega)\,\d r\,\d\P(\omega)\\
	&=\int_s^t \int_{\Omega}\nabla\varphi(X_r(\omega))\cdot f(X_r(\omega)),u(r,\omega),\mu_r)\,\d\P(\omega)\,\d r.
\end{aligned}
\end{equation}

Using the growth condition of $f$ in \eqref{f:growth} we have
\begin{align*}
	\left| \int_{\R^d}\varphi(x)\,\d\mu_t(x) - \int_{\R^d}\varphi(x)\,\d\mu_s(x)\right|  \leq
	C\|\nabla\varphi\|_\infty \int_s^t \int_{\Omega}(1+|X_r(\omega)|+\m_p(\mu_r))\,\d\P(\omega)\,\d r.
\end{align*}
Thanks to \eqref{boundXt} in Proposition \ref{prop:estimatesL}, it follows that 
the map $r\mapsto  \int_{\Omega}(1+|X_r(\omega)|+\m_p(\mu_r))\,\d\P(\omega)$ belongs to $L^1(0,T)$
so that the map $t\mapsto \int_{\R^d}\varphi(x)\,\d\mu_t(x)$ is absolutely continuous.
Then, from \eqref{eq:distr_mu} it holds that 
\begin{align*}
\frac{\d}{\d t}\int_{\R^d}\varphi(x)\,\d\mu_t(x)
= \int_{\Omega}\nabla\varphi(X_t(\omega))\cdot f(X_t(\omega)),u(t,\omega),\mu_t)\,\d\P(\omega), \quad \text{ for } \cL_T\text{-a.e.} t\in[0,T].
\end{align*}

For $\cL_T$-a.e. $t \in [0,T]$, we denote now $\sigma_t\in \PP(\Omega\times U)$ and $\theta_t \in \PP(\R^d\times U)$ the disintegrations of $\sigma$ and $\theta$  with respect to the projection maps $\pi^1_{\Omega}: [0,T] \times \Omega \times U \to [0,T]$, $\pi^1_{\R^d}: [0,T] \times \R^d \times U \to [0,T]$, respectively. 
It can be shown that $\theta_t=(X_t,i_U)_\sharp\sigma_t =\theta_{t,x}\otimes\mu_t$, for $\cL_T$-a.e. $t \in [0,T]$. 

Using the affinity of $f$, and the definition of $\ubar u$ in \eqref{eq:ubaru_theta} we obtain
\begin{align*}
& \int_{\Omega}\nabla\varphi(X_t(\omega))\cdot f(X_t(\omega)),u(t,\omega),\mu_t)\,\d\P(\omega)\\
&=\int_{\Omega\times U}\nabla\varphi(X_t(\omega))\cdot f(X_t(\omega),u,\mu_t)\,\d \sigma_t(\omega,u)\\
&=\int_{\R^d\times U}\nabla\varphi(x)\cdot f(x,u,\mu_t)\,\d\theta_t(x,u)\\
&=\int_{\R^d}\nabla\varphi(x)\cdot\int_U f(x,u,\mu_t)\,\d\theta_{t,x}(u)\,\d\mu_t(x) \\
&=\int_{\R^d}\nabla\varphi(x)\cdot f\left(x,\int_U u\,\d\theta_{t,x}(u),\mu_t\right)\,\d\mu_t(x)\\
&=\int_{\R^d}\nabla\varphi(x)\cdot f(x,\ubar u(t,x),\mu_t)\,\d\mu_t(x).
\end{align*}

Then $\mu$ satisfies the continuity equation $\de_t\mu_t+\div(v_t\mu_t)=0$ for the vector field
$v_t(x):= f(x,\ubar u(t,x),\mu_t)$ in the sense of distributions (see e.g. \cite[equation (8.1.4)]{ambrosio2008gradient}).
Since $|v_t(x)|^p\leq \tilde C(1+|x|^p+\m_p^p(\mu_t))$, from \eqref{boundXt} it follows that $t\mapsto \|v_t\|_{L^p_{\mu_t}(\R^d;\R^d)}$ 
 belongs to $L^p(0,T)$ and $\mu \in \AC^p([0,T];\PP_p(\R^d))$.
Hence, $(\mu,\ubar u)\in\cA_{\E}(\mu_0)$.

Finally, by the convexity of $\cC$ with respect to $u$ and Jensen's inequality we obtain
\begin{equation*}
\begin{split}
&\int_0^T \int_{\Omega} \cC(X_t(\omega),u_t(\omega),\mu_t) \, \d \P(\omega) \,\d t \\
&= \int_0^T \int_{\Omega\times U} \cC(X_t(\omega),u,\mu_t) \, \d \sigma_{t}(\omega,u)\,\d t \\
&=\int_0^T \int_{\R^d\times U} \cC(x,u,\mu_t) \, \d \theta_{t}(x,u)\, \d t \\
&=\int_0^T \int_{\R^d} \int_{U}\cC(x,u,\mu_t) \, \d \theta_{t,x}(u)\,\d \mu_t(x)\, \d t \\
&\ge \int_0^T \int_{\R^d} \cC\left( x, \int_U u \, \d \theta_{t,x}(u) ,\mu_t\right) \,\d \mu_t(x) \d t\\
&=  \int_0^T \int_{\R^d} \cC\left( x, \ubar u(t,x),\mu_t\right) \,\d \mu_t(x)\, \d t.
\end{split}
\end{equation*}
This readily implies that $J_{\pL}(X,u) \ge J_{\E}(\mu,\ubar u)$.
\end{proof}

In the next Lemma, {we are given an admissible pair $(\eeta,u)$ for the Kantorovich problem. Considering the evaluation map $Z_t(\gamma)=\gamma(t)$, we associate to $(\eeta,u)$ the pair $(Z,u)$ which is admissible for the Lagrangian problem with parametrization space $(\Gamma_T, \cB_{\Gamma_T}, \eeta)$ and with the same cost as $(\eeta,u)$.}

\begin{lemma}\label{l:fromKtoL}
Let $\S = (U,f,\cC,\cC_T)$  satisfy Assumption \ref{BA}. 
Denote with $Z(t,\gamma):= e_t(\gamma)$, for every $t \in [0,T]$ and $\gamma \in \Gamma_T$.
If $\mu_0 \in \PP_p(\R^d)$, $(\eeta,u) \in \cA_\K(\mu_0)$  and we denote with $\pL_{\eeta} = \pL(\Gamma_T, \cB_{\Gamma_T}, \eeta)$,  then  $(Z,u)\in \cA_{\pL_{\eeta}}(e_0)$.
Moreover, 
\begin{equation*}
	J_{\pL_{\eeta}}(Z,u)=J_{\K}(\eeta,u).
\end{equation*}
\end{lemma}
\begin{proof}
Let $(\eeta,u)\in\cA_{\K}(\mu_0)$ and denote by $\pL_{\eeta}=\pL(\Gamma_T,\cB_{\Gamma_T},\eeta)$. 
Denoting with  $Z:[0,T]\times\Gamma_T\to\R^d$ the map defined by 
$Z(t,\gamma):=\gamma(t)=e_t(\gamma)$, let us show that $(Z,u)\in\cA_{\pL_{\eeta}}(e_0)$. 
Since $Z(0,\gamma)_\sharp \eeta = \mu_0 \in \PP_p(\R^d)$ then $e_0 = Z(0,\cdot) \in L^p_{\eeta}(\Gamma_T)$.
By item (i) in Definition \ref{def:K} we have $u \in \Bor([0,T] \times \Gamma_T;U)$.
Thanks to \eqref{eq:momp} it readily follows that $Z \in L^p_{\eeta}(\Gamma_T; L^p(0,T;\R^d))$ and from Remark \ref{rem:etaACp} we actually have that $Z \in L^p_{\eeta}(\Gamma_T; \AC^p([0,T];\R^d))$.
Moreover, from item (ii) of Definition \ref{def:K}, for $\eeta$-a.e. $\gamma \in \Gamma_T$, we have
\[\dot Z(t,\gamma) = f(Z(t,\gamma), u(t,\gamma), Z(t,\cdot)_\sharp \eeta) \quad \text{ for } \cL_T\text{-a.e.} t \in [0,T].\] 
Hence $(Z,u)\in\cA_{\pL_{\eeta}}(e_0)$ and, by definition of $Z$, $J_{\pL_{\eeta}}(Z,u)=J_{\K}(\eeta,u)$.

\end{proof}

{When the parametrization space $(\Omega,\frB,\P)$ is fixed a priori, an interesting first comparison between the Kantorovich and Lagrangian problems is given below.}

\begin{theorem}\label{prop:K>L}
Let $\S = (U,f,\cC,\cC_T)$  satisfy Assumption \ref{BA} with $U$ convex compact subset of a separable Banach space $V$.
Let $(\Omega,\frB,\P)$ be a standard Borel space such that 
$\P$ is without atoms.
If $\mu_0\in\PP_p(\R^d)$ and $(\eeta,u)\in\cA_{\K}(\mu_0)$, then 
for every $X_0 \in L^p(\Omega;\R^d)$ with $(X_0)_\sharp\P=\mu_0$
there exits a sequence $(X^n,u^n)\in\cA_{\L}(X_0)$ such that
\begin{equation}\label{eq:convKL}
	\lim_{n\to+\infty}J_\pL(X^n,u^n)=J_\K(\eeta,u).
\end{equation}
Moreover, for every $\mu_0\in\PP_p(\R^d)$ and every $X_0 \in L^p(\Omega;\R^d)$ with $(X_0)_\sharp\P=\mu_0$ it holds
\begin{equation}\label{VKgeVL}
V_{\K}(\mu_0)\ge V_{\pL}(X_0).
\end{equation}
\end{theorem}

\begin{proof}

 {\bf Step 1.}
Let $(\eeta,u)\in\cA_{\K}(\mu_0)$. 
We denote by $\pL_{\eeta}:= \pL(\Gamma_T,\cB_{\Gamma_T},\eeta)$. 
Defining 
$Z(t,\gamma):=\gamma(t)=e_t(\gamma)$, by Lemma \ref{l:fromKtoL} it holds that  $(Z,u)\in\cA_{\pL_{\eeta}}(e_0)$ and  
\begin{equation}\label{eq:JLetaeqJK}
	J_{\pL_{\eeta}}(Z,u)=J_{\K}(\eeta,u).
\end{equation}
 
Thanks to the continuity of the evaluation map $e_0: \Gamma_T \to \R^d$,  we apply Proposition \ref{lemmaB} for the problem $\pL_{\eeta}$ in the Polish space $(\Gamma_T, \cB_{\Gamma_T}, \eeta)$.
 Then there exists a sequence $(\bar Z^m,\bar u^m)\in\cA_{\pL_{\eeta}}(e_0)$ such that
  $\bar Z^m:[0,T]\times\Gamma_T\to\R^d$ and  $\bar u^m:[0,T]\times\Gamma_T\to U$ are continuous and
\begin{equation}\label{eq:convJLeta}
	\lim_{m\to+\infty}J_{\pL_{\eeta}}(\bar Z^m,\bar u^m)=J_{\pL_{\eeta}}(Z,u).
\end{equation}

 {\bf Step 2.}
 Let $(\Omega,\frB,\P)$ be a standard Borel space such that 
$\P$ is without atoms and $\tau$ be a Polish topology on $\Omega$ such that $\frB = \cB_{(\Omega,\tau)}$.
Denote with $\pL = \pL(\Omega,\cB_{(\Omega,\tau)},\P)$ the Lagrangian problem for the system $\S$.
Let $X_0\in L^p(\Omega;\R^d)$ with $(X_0)_\sharp\P=\mu_0$.
Given $\bar Z\in C([0,T]\times\Gamma_T;\R^d)$ and $\bar u\in C([0,T]\times\Gamma_T;U)$ such that
 $(\bar Z,\bar u)\in\cA_{\pL_{\eeta}}(e_0)$, let us prove that there exists a sequence 
 $(\tilde X^n,\tilde u^n)\in\cA_{\pL}(X_0)$ such that
 \begin{equation}\label{eq:convJLXtilde}
	\lim_{n\to+\infty}J_{\pL}(\tilde X^n,\tilde u^n)=J_{\pL_{\eeta}}(\bar Z,\bar u).
\end{equation}
 
We define the sets
\begin{itemize}
\item $\tilde{\Gamma}:=\left\{(x,\gamma)\in\R^d\times\Gamma_T\,:\,x=\gamma(0)=e_0(\gamma)\right\};$
\item $\Gamma_0:=\left\{\gamma\in\Gamma_T\,:\,\gamma(0)=e_0(\gamma)=0\right\},$
\end{itemize}
and the continuous maps
\begin{itemize}
\item $r:\tilde\Gamma\to\R^d\times\Gamma_0$, $r(x,\gamma):=(x,\gamma-x)$, 
where $\gamma-x$ is the curve $t\mapsto\gamma(t)-x$.
Notice that $r$ admits a left inverse $r^{-1}:\R^d\times\Gamma_0\to\tilde\Gamma$, $r^{-1}(x,\gamma_0)=(x,\gamma_0+x)$, 
that obviously satisfies $r^{-1}\circ r=i_{\tilde\Gamma}$;
\item $s:\Omega\times\R^d\times\Gamma_0\to\Omega\times\tilde\Gamma$, $s(\omega,x,\gamma_0)=(\omega,x,\gamma_0+x)$. 
Observe that $s=(i_\Omega,r^{-1})$.
\end{itemize}
Let us consider the couplings 
\[\rho:=(i_\Omega, X_0)_\sharp\P\in\mathscr P(\Omega\times\R^d),\quad \tilde\eta:=(e_0, i_{\Gamma_T})_\sharp\boldsymbol\eta\in\mathscr P(\tilde\Gamma),\quad \hat\eta:=r_\sharp\tilde\eta\in\mathscr P(\R^d\times\Gamma_0).\]
Notice that $\pi^2_\sharp\rho=\pi^1_\sharp\tilde\eta=\pi^1_\sharp\hat\eta=\mu_0$.

We define a measure $\hat\sigma\in\mathscr P(\Omega\times\R^d\times\Gamma_0)$ satisfying 
$\pi^{1,2}_\sharp\hat\sigma=\rho$ and $\pi^{2,3}_\sharp\hat\sigma=\hat\eta$.
Since $\P$ is without atoms and $\pi^1_\sharp\rho=\P$ then also $\rho\in\PP(\Omega\times\R^d)$ is without atoms.
Applying Lemma \ref{lemma:young} with $\T=\Omega\times\R^d$, $S = \Gamma_0$, $\lambda=\rho$ and 
$\nu=\hat\sigma$, there exists a sequence of Borel maps $\hat w_n:\Omega\times\R^d\to\Gamma_0$ such that 
\begin{equation}\label{eq:hatyoung}
\hat\sigma_n:=(i_{\Omega\times\R^d},\hat w_n)_\sharp\rho\xrightarrow{\mathcal{Y}}\hat\sigma, \quad \textrm{ as }n\to+\infty.
\end{equation}
Define $w_n:\Omega\times\R^d\to\Gamma_T$ by $w_n(\omega,x):=\hat w_n(\omega,x)+x$ and note that 
$s(\omega,x,\hat w_n(\omega,x))=(\omega,x,w_n(\omega,x))$.
Thanks to the continuity of $s$, then $s_\sharp$ is weakly continuous. 
From Remark \ref{rem:Young-weak}, by the composition rule \eqref{eq:composition} and \eqref{eq:hatyoung}, we have that 
\begin{equation}\label{eq:convYn}
\tilde\sigma_n:=s_\sharp\hat\sigma_n=(i_{\Omega\times\R^d}, w_n)_\sharp\rho\xrightarrow{\mathcal{Y}}s_\sharp\hat\sigma=:\tilde\sigma\in\mathscr P(\Omega\times\R^d\times\Gamma_T),  \quad \textrm{ as }n\to+\infty.
\end{equation}
From \eqref{eq:composition}, a direct computation shows
\begin{align*}
&\pi^{1,2}_\sharp\tilde\sigma=\rho,\\
&\tilde\eta=r^{-1}_\sharp\hat\eta=\left(r^{-1}\circ\pi^{2,3}\right)_\sharp\hat\sigma=\left( r^{-1}\circ\pi^{2,3}\circ s^{-1} \right)_\sharp\tilde\sigma=\pi^{2,3}_\sharp\tilde\sigma.
\end{align*}

We define 
$\eeta^n:=\pi^3_\sharp\tilde\sigma_n\in\PP(\Gamma_T)$. 
Observing that $\pi^3_\sharp\tilde\sigma=\eeta$, by \eqref{eq:convYn} we have $\eeta^n\to\eeta$ weakly in $\PP(\Gamma_T)$.
Notice also that $(e_0)_\sharp\eeta^n=(e_0\circ w_n\circ(i_\Omega, X_0))_\sharp\P=(X_0)_\sharp\P=\mu_0$ 
and $(e_0)_\sharp\eeta=\mu_0$. For every $n \in \N$, denote by $\pL_{\eeta^n}:= \pL(\Gamma_T, \cB_{\Gamma_T}, \eeta^n)$ the Lagrangian problem for the system $\S$. 
Since $e_0$, $\bar u$ are continuous, we can apply Proposition \ref{lemmaA} in the probability space 
$(\Gamma_T, \cB_{\Gamma_T}, \eeta)$, with  $\eeta^n, \eeta\in\PP(\Gamma_T)$ and initial datum $e_0$. Thus if $(Z^n,\bar u)\in\cA_{\pL_{\eeta^n}}(e_0)$, we have that 
\begin{equation}\label{eq:convJLetan}
	\lim_{n\to+\infty}J_{\pL_{\eeta^n}}(Z^n,\bar u)=J_{\pL_{\eeta}}(\bar Z,\bar u).
\end{equation}

Finally, for any $n\in\N$, we define the pair $(\tilde X^n,\tilde u^n)$ by
\begin{align*}
\tilde X^n:[0,T]\times\Omega\to\R^d,\quad &\quad  \tilde X^n(t,\omega):=Z^n(t,w_n(\omega,X_0(\omega)));\\
\tilde u^n:[0,T]\times\Omega\to U,\quad &\quad \tilde u^n(t,\omega):=\bar u(t,w_n(\omega,X_0(\omega))).
\end{align*}
Observe that $\tilde X^n(0,\omega)=Z^n(0,w_n(\omega, X_0(\omega)))=e_0(w_n(\omega, X_0(\omega)))=X_0(\omega)$. Moreover, thanks to the composition rule \eqref{eq:composition} we have $\pi^3_\sharp\tilde\sigma_n=(w_n)_\sharp\rho$ so that 
\[(\tilde X^n_t)_\sharp\P=(Z^n_t\circ w_n\circ (i_\Omega, X_0))_\sharp\P=(Z^n_t)_\sharp\eeta^n.\]
By construction we have $(\tilde X^n,\tilde u^n)\in\cA_\pL(X_0)$ and it is immediate to verify that
\[J_{\pL}(\tilde X^n,\tilde u^n)=J_{\pL_{\eeta^n}}(Z^n,\bar u).\]
Then, by \eqref{eq:convJLetan} we obtain \eqref{eq:convJLXtilde}.

{\bf Step 3.}
We apply Step 2 to the sequence $(\bar Z^m,\bar u^m)$ constructed in Step 1.
Fix $m \in \N$, then there exists a sequence $(\tilde X^{m,n},\tilde u^{m,n})_{n \in \N}$ such that 
$(\tilde X^{m,n},\tilde u^{m,n})\in\cA_{\pL}(X_0)$ for every $n \in \N$ and 
 \begin{equation*}
	\lim_{n\to+\infty}J_{\pL}(\tilde X^{m,n},\tilde u^{m,n})=J_{\pL_{\eeta}}(\bar Z^m,\bar u^m), \qquad \forall\,m\in\N.
\end{equation*}
Thanks to \eqref{eq:convJLeta}, by a simple diagonal argument we can select a (not relabelled) sequence 
$(X^{n},u^{n})\in\cA_{\pL}(X_0)$ satisfying 
\[\lim_{n\to+\infty}J_{\pL}(X^{n},u^{n})=J_{\pL_{\eeta}}(Z, u),\]
where $(Z,u)$ are defined in Step 1.
From \eqref{eq:JLetaeqJK} we finally get 
\eqref{eq:convKL}.

{\bf Step 4.}
By  \eqref{eq:convKL} and the definition of $V_\pL$, for any $\eps>0$ there exists $n_\varepsilon>0$ such that for $n\ge n_\varepsilon$
\[V_{\pL}(X_0)\le J_{\pL}(X^n,u^n)\le J_{\K}(\eeta,u)+\varepsilon.\]
From the arbitrariness of $\varepsilon >0$ we have
\[V_{\pL}(X_0)\le J_{\K}(\eeta,u), \quad \forall \, (\eeta,u)\in\cA_{\K}(\mu_0),\] 
hence the required inequality.
\end{proof}

\begin{remark}\label{rem:minimizerLeta}
Under the Convexity Assumption \ref{CA}, 
if $(\eeta,u) \in \cA_{\K}(\mu_0)$ is an optimal pair, then $(Z,u)$ given by Lemma \ref{l:fromKtoL} is optimal for the Lagrangian problem $\pL_{\eeta}$. 
This is a consequence of Theorem \ref{cor:E=K}, Proposition \ref{prop:E<L} and of \eqref{VKgeVL} in Theorem \ref{prop:K>L}.
\end{remark}

We conclude the section with the proof of Theorem \ref{cor:VL=VE'}.
\begin{proof}[Proof of Theorem \ref{cor:VL=VE'}]
Thanks to Proposition \ref{prop:ConvexRL} and Remark \ref{rem:RL-L} 
the relaxed Lagrangian problem $\RL$ in  $\S$ coincides with $\pL'$ in $\S'$.
Precisely, for every $X_0\in L^p(\Omega;\R^d)$ it holds  $\cA_{\RL}(X_0)=\cA_{\pL'}(X_0)$.
Moreover,  $J_{\RL}(X,\sigma)=J_{\pL'}(X,\sigma)$ for every $(X,\sigma)\in \cA_{\RL}(X_0)=\cA_{\pL'}(X_0)$.
Hence,  $V_{\RL}(X_0)=V_{\pL'}(X_0)$.
The equality $V_{\pL}(X_0)=V_{\RL}(X_0)$ follows from Theorem \ref{cor:VL=VRL}. 
 Finally, Theorem \ref{cor:VL=VE} yields $V_{\pL'}(X_0)= V_{\E'}((X_0)_\sharp\P)=V_{\K'}((X_0)_\sharp\P)$.
\end{proof}

\subsection{Continuity of $V_\E$, $V_\K$ and $V_\pL$}
Here, we prove continuity results for the value functions of the various proposed formulations.

\begin{theorem}[Continuity of $V_\E$ and $V_\K$]\label{th:contVE}
Let $\S = (U,f,\cC,\cC_T)$  satisfy the Convexity Assumption \ref{CA}.
If $\mu_0 \in \PP_p(\R^d)$ and $\{\mu_0^n\}_{n\in\N}\subset \PP_p(\R^d)$ is a sequence such that $W_p(\mu_0^n, \mu_0) \to 0$ as $n \to +\infty$,
then  
\begin{equation*}
\lim_{n \to +\infty} V_{\E}(\mu_0^n) = V_\E(\mu_0), \qquad \lim_{n \to +\infty} V_{\K}(\mu_0^n) = V_\K(\mu_0).
\end{equation*}
\end{theorem}
\begin{proof}
Let $\mu_0^n$ converge to $\mu_0$ in $\PP_p(\R^d)$. 
By Proposition \ref{prop:Skorohod} with $S = \R^d $, there exist $X_0,X_0^n\in \Bor([0,1];\R^d)$ such that $(X_0)_\sharp\cL_1=\mu_0$, $(X^n_0)_\sharp\cL_1=\mu^n_0$ and  $X_0^n(\omega)\to X_0(\omega)$ for $\cL_1$-a.e. $\omega \in [0,1]$.
Since $\mu_0^n, \mu_0 \in \PP_p(\R^d)$ we have $X_0^n, X_0 \in L^p([0,1];\R^d)$.
Moreover by the convergence $W_p(\mu_0^n,\mu_0)\to0$ and Proposition \ref{prop:wassconv} there exists $\psi:[0,+\infty)\to[0,+\infty)$ admissible (according to Definition \ref{def:admphi})
such that
\begin{equation}
\sup_{n \in \N}\int_{[0,1]} \psi(|X^n_0(t)|^p)\,\d\cL_1(t)  = \sup_{n \in \N}\int_{\R^d} \psi(|x|^p)\,\d\mu_0^n(x) < +\infty.
\end{equation}
Thanks to Vitali theorem we get
\[\|X_0^n-X_0\|_{L^p([0,1];\R^d)}\to 0\quad \text{as } n\to+\infty.\]
Applying Proposition \ref{prop:uscVL} to the Lagrangian problem  in $\S$ with $(\Omega, \cB, \P) = ([0,1], \cB_{[0,1]}, \cL_1)$ we get 
$\limsup_{n\to +\infty} V_{\pL} (X_0^n) \leq V_{\pL}(X_0)$.
Theorem \ref{cor:VL=VE} yields 
\[\limsup_{n \to +\infty} V_{\E}(\mu_0^n)  = \limsup_{n\to +\infty} V_{\pL} (X_0^n) \leq V_{\pL}(X_0) = V_\E(\mu_0).\]
By Proposition \ref{prop:lscVE} we get $\lim_{n \to +\infty} V_{\E}(\mu_0^n) = V_\E(\mu_0)$. Finally, the continuity of $V_\K$ follows by Theorem \ref{cor:E=K} and the continuity of $V_\E$.
\end{proof}

\begin{theorem}[Continuity of $V_\pL$]\label{th:contVL}
Let $\S=(U,f,\cC,\cC_T)$  satisfy Assumption \ref{BA} and $(\Omega, \frB, \P)$ be a standard Borel space such that $\P$ is without atoms.
If $X_0\in L^p(\Omega;\R^d)$ and $\{X_0^n\}_{n\in\N}\subset L^p(\Omega;\R^d)$ is a sequence such that $\|X_0^n-X_0\|_{L^p(\Omega;\R^d)}\to 0$ as $n \to +\infty$,
then  
\begin{equation*}
\lim_{n \to +\infty} V_{\pL}(X_0^n) = V_\pL(X_0).
\end{equation*}
\end{theorem}
\begin{proof}
From Theorem \ref{cor:VL=VE'} we  have $V_{\pL}(X_0^n) = V_{\E'}((X_0^n)_\sharp \P)$  and $V_\pL(X_0)= V_{\E'}((X_0)_\sharp \P)$.
The application of Theorem \ref{th:contVE} to $\E'$ in $\S' = (\UU,\FF,\CC,\CC_T)$ (see Definition \ref{def:relax_setting}) concludes the proof.
\end{proof}

\subsection{A counterexample: Non-existence of minimizers for $\pL$}\label{sec:counterexample}
In the previous sections we have shown that, under the Convexity Assumption \ref{CA}, the Eulerian and Kantorovich problems always admit a minimizer, see Theorems \ref{thm:minE}, \ref{cor:E=K}.
This is not always true in the Lagrangian setting.
Existence of minimizers has been shown in Remark \ref{rem:minimizerLeta} in the very particular case $\pL = \pL_{\eeta}$, where $\eeta$ is optimal for a Kantorovich problem.
In general, for a given parametrization space $(\Omega, \frB,\P)$, the choice of the initial condition is relevant as highlighted in the following. 

\begin{theorem}
Let $\S = (U,f,\cC,\cC_T)$ satisfy the Convexity Assumption \ref{CA} and $(\Omega, \frB,\P)$ be a standard Borel space such that $\P$ is without atoms. 
If $\mu_0\in \PP_p(\R^d)$ then there exists $X_0 \in L^p(\Omega;\R^d)$ with $(X_0)_\sharp \P = \mu_0$ and $(X,u) \in \cA_\pL(X_0)$ such that 
\begin{equation}\label{eq:JL=VL}
J_\pL(X,u) = V_\pL(X_0) = V_\E(\mu_0).
\end{equation}
\end{theorem}
\begin{proof}
Let $\mu_0 \in \PP_p(\R^d)$, by Theorem \ref{cor:E=K} there exists $(\eeta,\bar u) \in \cA_{\K}(\mu_0)$ such that $J_\K(\eeta,\bar u) = V_\K(\mu_0)$.
Fix $\tau$ a  Polish topology on $\Omega$ such that $\frB = \cB_{(\Omega,\tau)}$.
Since $\P$ is without atoms, thanks to Proposition \ref{prop:Skorohod} there exists a Borel map $\psi: \Omega \to \Gamma_T$ such that $\psi_\sharp \P = \eeta$.
For every $t \in [0,T]$ we define $X_t:= e_t \circ \psi$ and the Borel map $u(t,\omega):= \bar u(t, \psi(\omega))$.
Using the same techniques as in the proof of Lemma \ref{l:fromKtoL} we deduce that  $(X,u) \in \cA_\pL(X_0)$ (where $X_0 := e_0 \circ \psi$) and 
 $J_\pL(X,u) = J_\K(\eeta,\bar u)$.
 By Theorems \ref{cor:VL=VE} and \ref{cor:E=K} we finally get \eqref{eq:JL=VL}.
\end{proof}

In general, if the initial condition $X_0$ is assigned a priori, existence of minimizers for the Lagrangian problem is not guaranteed.
We consider the Wasserstein barycenter problem, for which we study the Eulerian and Lagrangian formulations.
In particular, we exhibit an initial datum $X_0$ whose corresponding Lagrangian problem does not admit minimizers.
We stress that the system under consideration satisfies the Convexity Assumption \ref{CA}.

\subsubsection{Wasserstein barycenter problem: Eulerian formulation}\label{subsec:ExBarE}
We consider the setting $\S = (U, f, \cC, \cC_T)$ as follows:
let $U := \overline{B_R(0)}\subset \R^d$, for some $R >0$ sufficiently large, $T=1$ and $p =2$.
We fix $\nu\in\PP(\R^d)$ with compact support.
We consider the velocity field $f:\R^d\times U\times \PP_2(\R^d)\to \R^d$, the cost functions $\cC:\R^d\times U\times \PP_2(\R^d)\to [0,+\infty)$ 
and $\cC_T:\R^d\times \PP_2(\R^d)\to [0,+\infty)$ defined by
\[
f(x,u,\mu) = u, \qquad \cC(x,u,\mu) = |u|^2, \qquad \cC_T(x,\mu) = W^2_2(\mu, \nu).
\]
In this setting, the cost functional has the form
\[ J_{\E}(\mu,\ubar u) 
= \int_0^1 \int_{\R^d} |\ubar u(t,x)|^2 \d \mu_t(x) \d t + W_2^2(\mu_1, \nu).	 \]
For any $\mu_0\in\PP_2(\R^d)$, the associated value function is given by 
\begin{equation}\label{ExBarE}
V_{\E}(\mu_0) := \inf_{(\mu,\ubar u) \in \cA_{\E}(\mu_0)} J_{\E}(\mu,\ubar u),
\end{equation}
and recall that by Theorem \ref{thm:minE} the infimum in \eqref{ExBarE} is actually a minimum. 

Let us now fix $\mu_0 \in \PP(\R^d)$ with compact support and characterize the value function and the corresponding minimizers.
By the Benamou-Brenier formula \eqref{B&B}, we have the lower bound
\begin{equation}\label{ExBar2} 
\begin{split}
 \inf_{(\mu,\ubar u) \in \cA_{\E}(\mu_0)} \left(\int_0^1 \int_{\R^d} |\ubar u(t,x)|^2 \d \mu_t(x) \d t + W_2^2(\mu_1, \nu) \right)
 \geq  \inf_{\mu_1\in\PP_2(\R^d)} \left[W_2^2(\mu_0, \mu_1) + W_2^2(\mu_1, \nu) \right].
\end{split}
\end{equation}
Using the triangle inequality, it is easy to prove that
\begin{equation}\label{ExBar1}
 W_2^2(\mu_0, \mu_1) + W_2^2(\mu_1, \nu)\ge \frac12  W_2^2(\mu_0, \nu), \qquad \forall\, \mu_1\in\PP_2(\R^d),
\end{equation}
and, for any constant speed Wasserstein geodesic $\{\sigma_t\}_{t\in[0,1]}$ such that $\sigma_0=\mu_0$ and $\sigma_1=\nu$, the measure 
$\mu_1=\sigma_{1/2}$ realizes  the equality in \eqref{ExBar1}.
Since the supports of $\mu_0$ and $\nu$ are compact, then the support of $\sigma_{1/2}$ is compact and, denoting by
$\{\mu_t\}_{t\in[0,1]}$ a Wasserstein geodesic joining $\mu_0$ to $\sigma_{1/2}$,
 we also have that a vector field $\ubar u$ realizing the equality
\begin{equation}\label{ExBar3} 
\int_0^1 \int_{\R^d} |\ubar u(t,x)|^2 \d \mu_t(x) \d t 
 = W_2^2(\mu_0, \sigma_{1/2})
\end{equation}
is bounded (see e.g. \cite[Section~5.4]{santambrogio2015optimal}).
Then, using $\ubar u$ satisfying \eqref{ExBar3} and choosing $R$ sufficiently large, we obtain the equality in \eqref{ExBar2}. 
The value of the minimum is 
\begin{equation}\label{eq:min_wasserstein_bar}
 V_\E(\mu_0)= W_2^2(\mu_0, \sigma_{1/2}) + W_2^2(\sigma_{1/2}, \nu) = \frac12  W_2^2(\mu_0, \nu).
\end{equation}
Notice that the minimizer $(\mu, \ubar u)$ is not unique a priori.
If at least one of the measures $\nu$ and $\mu_0$ is absolutely continuous with respect to $\cL^d$, 
then the geodesic $\{\sigma_t\}_{t\in[0,1]}$ is unique and the  map $\eta \mapsto W^2_2(\eta, \nu)$ is strictly convex.
In this case $\sigma_{1/2}$ is the {unique} minimizer of the functional $\eta \mapsto W^2_2(\mu_0, \eta) + W^2_2(\eta, \nu)$
and the pair $(\mu,\ubar u)$, with $\mu_t:=\sigma_{t/2}$ for all $t\in[0,1]$, is the (unique) minimizer for the Eulerian problem. 

\subsubsection{Wasserstein barycenter problem: Lagrangian formulation}\label{subsec:ExBarL}
Let $(\Omega, \frB, \P)$ be a standard Borel space such that $\P\in\PP(\Omega)$ is without atoms.
The Lagrangian cost functional of the Wasserstein barycenter problem is given by
\begin{equation}
J_{\pL}(X,u):= \int_0^1 \int_{\Omega} |u_t(\omega)|^2 \d \P(\omega) \d t + W_2^2((X_1)_\sharp\P, \nu).
\end{equation}
For any $X_0 \in L^2(\Omega;\R^d)$, the   corresponding value function is 
\begin{equation*}
V_{\pL}(X_0) = \inf_{(X,u) \in \cA_{\pL}(X_0)} J_{\pL}(X,u). 
\end{equation*}
Since $(X,u) \in \cA_{\pL}(X_0)$ satisfies, for $\P$-a.e. $\omega\in\Omega$, the system
\begin{equation*}
\begin{system}
\dot X_t(\omega) = u_t(\omega), \qquad \text{ for } \cL_1\text{-a.e. } t \in (0,1) \\
X|_{t = 0}(\omega) = {X_0(\omega)},
\end{system}
\end{equation*}
we have
\begin{equation}\label{eq:ex:lagrangian}
\begin{split}
\int_0^1 &\int_{\Omega} |u_t(\omega)|^2 \d \P(\omega) \d t  \geq \int_{\Omega} \left| \int_0^1 u_t(\omega) \d t  \right|^2\d \P(\omega) \\
&=\int_{\Omega} \left| \int_0^1 \dot X_t(\omega) \d t  \right|^2\d \P(\omega) = \int_{\Omega} \left| X_1(\omega) - X_0(\omega)  \right|^2\d \P(\omega) 
\end{split}
\end{equation}
where we have applied Fubini theorem and Jensen's inequality. 
Notice that the inequality in \eqref{eq:ex:lagrangian} becomes an equality if $(X,u)$ belongs to the restrict admissibility class given by 
\begin{equation*}
\begin{split}
\bar \cA_{\pL}(X_0):= \lbrace (X,u)\in  \cA_{\pL}(X_0)\,: \; u_t(\omega) = \bar u(\omega), \, \forall t \in [0,1], \text{ for $\P$-a.e. }\omega\in\Omega \rbrace.
\end{split}
\end{equation*} 

Suppose now that $\mu_0:= (X_0)_\sharp \P $ has compact support (i.e. $X_0$ bounded), 
then we can compare the Lagrangian and Eulerian formulation  of the Wasserstein barycenter problem.
Indeed, for $R$ sufficiently large we have
\begin{equation}\label{ExBarL1}
\begin{split}
V_{\pL}(X_0) &=\inf_{(X,\bar u) \in \bar\cA_{\pL}(X_0)} \left(\int_{\Omega} |X_1(\omega)-X_0(\omega)|^2 \d \P(\omega) + W_2^2((X_1)_\sharp\P, \nu) \right) \\
=V_{\E}(\mu_0) &= \min_{\mu_1 \in \PP_2(\R^d)}\left(W^2_2(\mu_0, \mu_1) + W^2_2(\mu_1, \nu) \right) = \frac12  W_2^2(\mu_0, \nu).
\end{split}
\end{equation}
where  the first equality follows by the choice $(X,\bar u) \in \bar\cA_{\pL}(X_0)$, the second equality is given by Theorem \ref{cor:VL=VE} and the last two equalities are exactly \eqref{eq:min_wasserstein_bar}.

\medskip

We now exhibit an example where the infimum for the Lagrangian problem is not a minimum.
Let us consider  $\Omega = [0,1]$, $\P = \cL_1$ and fix the dimension $d=2$.\\
We set $\nu:= \cL^2 \mres [0,1]^2$, 
$X_0:[0,1]\to\R^2$ defined by $X_0(\omega) = (1/2,\omega)$.
We observe that $\mu_0:= (X_0)_\sharp\P=\HH^1\mres (\lbrace 1/2 \rbrace \times [0,1])$.
We also notice that $X_0^{-1}:\R^2\to [0,1]$,
defined $\mu_0$-a.e. has the form $X_0^{-1}(1/2,\omega)=\omega$.\\
Since
\begin{equation*}
\begin{split}
\int_{[0,1]} \left| X_1(\omega) - X_0(\omega)  \right|^2\d \P(\omega) 
= \int_{\R^2} \left| X_1(X_0^{-1}(x)) - x  \right|^2\d \mu_0(x),
\end{split}
\end{equation*}
and defining
\begin{equation}\label{ExBarLB}
\begin{split}
\BB_{\mu_0} &:=\lbrace X_1\circ X_0^{-1} : X_1 = X_0  + \bar u \, \text{ and }  (X,\bar u) \in \bar \cA_{\pL}(X_0) \rbrace \\
&= \lbrace Y\in L^2((\R^2,\mu_0);\R^2): |Y(x)-x|\leq R, \text{ for $\mu_0$-a.e. }x\in\R^2 \rbrace ,
\end{split}
\end{equation}
we easily get
\begin{equation}\label{ExBarL2}
\begin{split}
V_{\pL}(X_0) &= \inf_{(X,\bar u) \in \bar \cA_{\pL}(X_0)} J_{\pL}(X,\bar u) \\
&= \inf_{Y \in \BB_{\mu_0}} \left[ \int_{\R^2} \left| Y(x) - x  \right|^2\d \mu_0(x) + W^2_2(Y_\sharp \mu_0, \nu)\right].
\end{split}
\end{equation} 
As already observed at the end of subsection \ref{subsec:ExBarE},
since $\nu$ is absolutely continuous with respect to $\cL^2$, there exists a unique geodesic $\{\sigma_t\}_{t\in[0,1]}$ joining $\mu_0$ to $\nu$. 
Furthermore, $\sigma_{1/2}$ is the {unique} minimizer of the functional $\eta \mapsto W^2_2(\mu_0, \eta) + W^2_2(\eta, \nu)$.
Then, for $R$ sufficiently large, from \eqref{ExBarL1}  and \eqref{ExBarL2} we know that 
\begin{equation}\label{ExBarL3}
\begin{split}
W^2_2(\mu_0, \sigma_{1/2}) + W^2_2(\sigma_{1/2}, \nu) &=\inf_{Y \in \BB_{\mu_0}} \left[ \int_{\R^d} \left| Y(x) - x  \right|^2\d \mu_0(x) + W^2_2(Y_\sharp \mu_0, \nu)\right] \\
&\leq \inf_{\substack{Y \in \BB_{\mu_0} \\ Y_\sharp \mu_0 = \sigma_{1/2}} } \int_{\R^d} \left| Y(x) - x  \right|^2\d \mu_0(x) + W^2_2(\sigma_{1/2}, \nu).
\end{split}
\end{equation}
On the other hand, since $\supp(\sigma_{1/2}) \subset [0,1]^2$, by \cite[Theorem~B]{pratelli2007equality} we have 
\begin{equation}\label{OTP}
W^2_2(\mu_0, \sigma_{1/2}) = \inf_{\substack{Y \in \BB_{\mu_0} \\ Y_\sharp \mu_0 = \sigma_{1/2}} } \int_{\R^d} \left| Y(x) - x  \right|^2\d \mu_0(x) 
\end{equation}
and, consequently, equality holds in \eqref{ExBarL3}. Moreover the infimum in \eqref{OTP} is not attained. 
Indeed, the map $T:\R^2\to \R^2$ defined by $T(x_1,x_2)=(1/2,x_2)$ satisfies $T_\sharp\nu=\mu_0$ and 
$T=\nabla\varphi$ for $\varphi(x_1,x_2)=\frac{1}{2}(x_1+x_2^2)$, therefore $T$ is the optimal transport map from $\nu$ to $\mu_0$.
The unique geodesic joining $\mu_0$ to $\nu$ is $\sigma_t=(t(x_1,x_2)+(1-t)(1/2,x_2))_\sharp\nu$ and 
$\sigma_{1/2}=(\frac12x_1+\frac14,x_2))_\sharp\nu$ coincides with the uniform probability measure on $[1/4,3/4]\times[0,1]$. 
The map $T$ is still the optimal transport map from $\sigma_{1/2}$ to $\mu_0$ and the unique optimal transport plan between 
$\sigma_{1/2}$ and $\mu_0$ is $((x_1,x_2),(1/2,x_2))_\sharp\sigma_{1/2}$. Then the unique optimal transport plan between 
$\mu_0$ and $\sigma_{1/2}$  is $\gamma:=((1/2,x_2),(x_1,x_2))_\sharp\sigma_{1/2}$. 
Since $\gamma$  is not concentrated on the graph of a map, the optimal transport map from $\mu_0$ to $\sigma_{1/2}$ does not exist.

Since \eqref{OTP} has not minimizers, then  \eqref{ExBarL2} cannot have minimizers.
Indeed, suppose there exists a minimizer $(X,u) \in \cA_{\pL}(X_0)$ for $J_\pL$ in \eqref{ExBarL2}. 
Then $X_1: [0,1]\to\R^2$ satisfies $(X_1)_\sharp\P=\sigma_{1/2}$, $u =X_1-X_0$ and $X_t=t u+X_0$. 
Defining $Y=X_1\circ X_0^{-1}$, we have that $Y_\sharp\mu_0=\sigma_{1/2}$ so that 
 $Y$ is a minimizer in \eqref{OTP}, which is absurd.

\medskip

\begin{remark}\label{rmk:RLex}
Notice that existence of minimizers is not guaranteed even for Relaxed Lagrangian problems.
Indeed, the same results obtained for the Wasserstein barycenter problem in $\S=(U,f, \cC, \cC_T)$ given in Section \ref{subsec:ExBarE}, can be easily extended to the  lifted system $\S'=(\UU,\FF,\CC,\CC_T)$ associated to $\S$ (see Definition \ref{def:relax_setting}).

In the proposed example, the Lagrangian and Eulerian problems $\pL', \E'$ associated to $\S'$ can be treated as the problems $\pL, \E$ associated to $\S$
 thanks to the following simple observation:
given a probability measure $\rho\in \PP(U)$, by Jensen's inequality we have
\[\int_U |u|^2\,\d\rho(u)\ge\left|\int_Uu\,\d\rho(u)\right|^2,\]
and the equality holds if and only if $\rho=\delta_{\ubar u}$ for some $\ubar u\in U$. 	
This guarantees that possible control minimizers for $\pL', \E'$ are of the form $ \delta_{\ubar u}$ with $\ubar u$ \emph{non-relaxed} control for $\pL, \E$, respectively. 
The corresponding trajectories for $\pL', \E'$ with control $\delta_{\ubar u}$ coincide with the ones associated to $\ubar u$ for problems $\pL$ and $\E$, respectively.
Finally, thanks to Remark \ref{rem:RL-L}, non-existence of minima for $\pL'$ corresponds to non-existence of minima for $\RL$.
\end{remark}

\section{Finite particle systems and Gamma-convergence}\label{sec:finite}

To model the evolution of a finite number of particles, we introduce a discrete finite space $\Omega^N$ with the corresponding normalized counting measure $\P^N$.
In this setting, in order to prove equivalence between Eulerian and Lagrangian problems, we cannot directly apply the results given in Theorems \ref{cor:VL=VE}  due to the requirement on the probability measure $\P$ to be without atoms (see in particular Theorem \ref{prop:K>L}).
Hence, we introduce a further formulation of the Lagrangian problem  in the context of feedback controls (see Definition \ref{def:FL}) and we exploit a discrete formulation of the superposition principle for which we refer to Theorem \ref{lem:EN}.

Furthermore, in Subsections  \ref{sec:gamma} and \ref{sec:gammaE}, we prove a (discrete to continuous) $\Gamma$-convergence result respectively for the Lagrangian and Eulerian cost functionals when the number of particles goes to infinity.

\subsection{Equivalences between $N$-particles problems}\label{sec:finiteN}

Let $(\Omega^N,\Parts(\Omega^N),\P^N)$ given by 
\begin{equation}\label{def:omegaN}
\begin{aligned}
&\Omega^{N}:=\{1,\dots,N\}, \quad \Parts(\Omega^{N}):= \sigma ( \lbrace 1\rbrace, \ldots \lbrace N\rbrace)), \\
&\P^{N}(\{k\}):=\frac 1N, \quad k=1,\dots,{N}.
\end{aligned}
\end{equation}
We will refer to $\P^N$ as the normalized counting measure, which can be written as
\begin{equation*}
\P^N=\frac{1}{N}\sum_{k=1}^N\delta_k.
\end{equation*}
Let us denote with $\pL^N = \pL(\Omega^N, \Parts(\Omega^N),\P^N )$ the Lagrangian problem associated to the probability space $(\Omega^N, \Parts(\Omega^N),\P^N )$.
Notice that the functional space $L^p(\Omega^N;\R^d)$ coincides with the space of all maps $g: \Omega^N \to \R^d$, which can be identified with $(\R^d)^N$.

Differently from the Lagrangian problem $\pL^N$, where we just need to fix the parametrization space, the definition of the $N$-particle Eulerian problem requires the introduction of a further constraint.
Let us firstly define the subspace of $\PP(\R^d)$ given by the discrete measures as
\begin{equation}\label{eq:discreteProb}
\PP^N(\R^d):= \left\lbrace	\mu = \frac{1}{N} \sum_{i =1}^N \delta_{x_i} \; \text{ for some } x_i \in \R^d \right\rbrace.
\end{equation}

\begin{definition}[Discrete Eulerian optimal control problem $(\E^N)$]\label{def:E^N}
Let $\S = (U,f,\cC,\cC_T)$  satisfy Assumption \ref{BA}. Given $\mu_0\in\PP^N(\R^d)$, we say that
 $(\mu,\ubar u)\in\cA_{\E^N}(\mu_0)$, if
\begin{itemize}
\item[(i)] $(\mu,\ubar u)\in\cA_{\E}(\mu_0)$; 
\item[(ii)] $\mu_t \in \PP^N(\R^d)$, for every $t \in [0,T]$.
\end{itemize}
We define the \emph{cost functional}  $J_{\E^N}:=J_\E$ and the \emph{value function} 
\[
V_{\E^N}(\mu_0):=\inf\{J_{\E^N}(\mu,\ubar u): (\mu,\ubar u)\in\cA_{\E^N}(\mu_0)\}.\]
\end{definition}

\begin{remark}
Notice that item (ii) in Definition \ref{def:E^N} does not follow from the requirement $\mu_0 \in \PP^N(\R^d)$. 
Indeed,  
the control map $\ubar u$ in general is not Lipschitz continuous so that uniqueness of characteristics is not guaranteed.
\end{remark}

Observe that, for every $N \in \N$, it holds 
\begin{equation}\label{eq:VEN>VE}
	V_{\E^N}(\mu_0)\geq V_{\E}(\mu_0), \qquad \forall\,\mu_0\in\PP^N(\R^d).
\end{equation}

The main result of this section is given in the following theorem. 
\begin{theorem}\label{cor:LN=EN}
Let $\S = (U,f,\cC,\cC_T)$ satisfy Assumption \ref{BA}.
Let  $p\geq 1$ and  $X_0\in L^p(\Omega^N;\R^d)$. Then 
\[V_{\pL^N}(X_0)= V_{\E^N}((X_0)_\sharp\P^N).\]
\end{theorem}
The proof is a direct consequence of Propositions \ref{prop:LN>EN}, \ref{prop:equivFRL} and \ref{prop:FLN<E} below.

\medskip

\noindent Exploiting the argument contained in \cite[Lemma~6.2]{fornasier2018mean} we derive a first comparison between  $V_{\pL^N}(X_0)$ and $V_{\E^N}((X_0)_\sharp\P^N)$.

\begin{proposition}\label{prop:LN>EN}
Let $\S = (U,f,\cC,\cC_T)$ satisfy Assumption \ref{BA} and let  $X_0 \in L^p(\Omega^N;\R^d)$. If $(X,u) \in \mathcal{A}_{\pL^N}(X_0)$, then there exists 
$(\mu,\ubar u) \in \mathcal{A}_{\E^N}((X_0)_\sharp \P^N)$ such that $J_{\pL^N}(X,u) \geq J_{\E^N}(\mu,\ubar u)$. Moreover,  $V_{\pL^N}(X_0)\ge V_{\E^N}((X_0)_\sharp\P^N)$.
\end{proposition}
\begin{proof}
Let $(X,u) \in \mathcal{A}_{\pL^N}(X_0)$.
Let us define $\mathcal{X}(t) := \{x\in \R^d: X_t(\omega)=x\text{ for some }\omega\in \Omega^N\}$ and 
  \begin{equation}
    \label{eq:J_tx}
    J(t,x):=\{\omega\in \Omega^N:X_t(\omega)=x\}, \qquad \text{ for any }\,  (t,x)\in [0,T]\times \R^d,
  \end{equation}
and denote by $\mathcal{P}$ the collection of partitions $P$ of $\Omega^N$. 
It is clear that the family $P_{X}(t) := \{J(t,x): x \in \mathcal{X}(t)\}$ belongs to $\mathcal{P}$. 
As proved in \cite[Lemma~6.2]{fornasier2018mean}, there exists a finite partition on Borel sets of the interval $[0,T]$ 
 of the form $\{S_P: P\in \mathcal{P}\}$, where $S_P:=\{t\in [0,T]: P_{X}(t)=P\}$.

Given $\omega, \omega' \in \Omega^N$ and a Borel set $S \subset [0,T]$, if $X_t(\omega) = X_t(\omega')$ for any $t \in S$ 
then, by the absolute continuity of the curves $t\mapsto X_t(\omega)$ and $t\mapsto X_t(\omega')$,
we have $\dot X_t(\omega) = \dot X_t(\omega')$, for $\cL_T$-a.e. $t \in S$.

We define $\mu_t:=(X_t)_\sharp\P^N$ and we observe that
\begin{equation}\label{eq:mut_classes}
\mu_t = \frac{1}{N}\sum_{x \in \mathcal{X}(t)}\# J(t,x) \delta_x.
\end{equation}
Moreover, defining for any $(t,x)\in [0,T]\times \R^d$
\[\bar\omega_{t,x}\in\argmin_{\omega \in J(t,x)} \cC(x,u(t,\omega), \mu_t),\]
we set
\begin{equation}\label{eq:defubarclasses}
\ubar u(t,x):= u(t,\bar\omega_{t,x}).
\end{equation}
We show that $(\mu,\ubar u) \in \mathcal{A}_{\E^N}((X_0)_\sharp \P^N)$. 
Let us fix $\varphi \in C_c^1(\R^d)$. 
For $\cL_T$-a.e. $t \in [0,T]$ it holds
\begin{equation}\label{eq:classes_continuity}
\begin{split}
\frac{\d}{\d t} \int_{\R^d} \varphi(x) \,\d ((X_t)_\sharp \P^N)(x) 
&= \int_{\Omega^N} \nabla\varphi(X_t(\omega)) \cdot f(X_t(\omega),u(t,\omega), \mu_t) \,\d \P^N(\omega) \\
&=\frac{1}{N}\sum_{\omega = 1}^N \nabla	\varphi(X_t(\omega)) \cdot f(X_t(\omega),u(t,\omega), \mu_t)\\  
&=\frac{1}{N}\sum_{x \in \mathcal{X}(t)} \nabla \varphi(x) \cdot \sum_{\omega \in J(t,x)} f(x,u(t,\omega), \mu_t).
\end{split}
\end{equation}    
Since $f(x,u(t,\omega), \mu_t) = f(x,\ubar u(t,x), \mu_t)$ for every $\omega \in J(t,x)$, we can rewrite the right hand side of \eqref{eq:classes_continuity} to get 
\begin{equation*}
\begin{split}
\frac{1}{N}\sum_{x \in \mathcal{X}(t)} \nabla \varphi(x) \cdot \sum_{\omega \in J(t,x)} f(x,u(t,\omega), \mu_t) &= \frac{1}{N}\sum_{x \in \mathcal{X}(t)}  \#J(t,x) \,\nabla \varphi(x) \cdot f(x,\ubar u(t,x), \mu_t)\\  
&=\int_{\R^d} \nabla \varphi(x) \cdot f(x,\ubar u(t,x), \mu_t) \,\d \mu_t(x),
\end{split}
\end{equation*}    
where the last equality is a consequence of \eqref{eq:mut_classes}.

For what concerns the cost functional, we have
\begin{equation*}
\begin{split}
&\int_0^T \int_{\Omega} \cC(X_t(\omega),u(t,\omega), \mu_t)\, \d \P^N(\omega) \,\d t \\
&\quad = \frac{1}{N} \int_0^T \sum_{x \in \mathcal{X}(t)}\sum_{\omega \in J(t,x)} \cC(x,u(t,\omega), \mu_t)\, \d t \\
&\quad \geq  \frac{1}{N}\int_0^T \sum_{x \in \mathcal{X}(t)}{\# J(t,x)}\cC(x,\ubar u(t,x), \mu_t)\, \d t \\
&\quad = \int_0^T \int_{\R^d} \cC(x,\ubar u(t,x), \mu_t) \, \d\mu_t(x) \,\d t,
\end{split}
\end{equation*}
where the inequality comes from the definition of $\ubar u$ in \eqref{eq:defubarclasses} and the last equality follows from \eqref{eq:mut_classes}. 
Since $$\int_{\Omega^N} \cC_T(X_T(\omega),\mu_T)\,\d\P^N(\omega)=\int_{\R^d} \cC_T(x,\mu_T)\,\d\mu_T(x),$$
we conclude that $J_{\pL^N}(X,u) \geq J_{\E^N}(\mu,\ubar u)$.
\end{proof}

Here we introduce a feedback formulation of the Lagrangian optimal control problem in order to prove the reverse inequality $V_{\pL^N}(X_0)\le V_{\E^N}((X_0)_\sharp\P^N)$. 
We firstly show its relation with the Lagrangian and Eulerian problems in a general context, i.e. where the probability space $(\Omega,\frB,\P)$ is not necessarily the space $(\Omega^N,\Parts(\Omega^N),\P^N)$ associated to the $N$-particles framework.

\begin{definition}[Feedback Lagrangian optimal control problem \textbf{(FL)}]\label{def:FL}
Let $\S = (U,f,\cC,\cC_T)$  satisfy Assumption \ref{BA} and let $(\Omega, \frB,\P)$ be a probability space. Given $X_0\in L^p(\Omega;\R^d)$, we say that  
$(X,\ubar u)\in\cA_{\FL}(X_0)$
 if
\begin{itemize}
\item[(i)] $\ubar u\in\Bor([0,T]\times\R^d;U)$;
\item[(ii)] $X\in L^p(\Omega;\AC^p([0,T];\R^d))$ and for 
$\P$-a.e. $\omega\in\Omega$, $X(\omega)$ is a
solution of the following Cauchy problem
\begin{equation}\label{eq:systemFRL}
\begin{cases}
\dot X_t(\omega)=f(X_t(\omega),\ubar u_t(X_t(\omega)),(X_t)_\sharp\P), &\text{ for } \cL_T \text{-a.e. } t\in]0,T]\\
X_{|t=0}(\omega)=X_0(\omega), &
\end{cases}
\end{equation}
where $X_t: \Omega \to \R^d$ is defined by $X_t(\omega):= X(t,\omega)$ for $\P$-a.e. $\omega \in \Omega$.
\end{itemize} 
We refer to $(X, \ubar u)\in\cA_{\FL}(X_0)$ as to an \emph{admissible pair}, with $X$ a \emph{trajectory} and $\ubar u$ a \emph{feedback control}.\\
We define the cost functional $J_{\FL}: L^p(\Omega;C([0,T];\R^d))\times {\Bor([0,T]\times\R^d;U)}\to[0,+\infty)$ 
by
\[J_{\FL}(X,\ubar u):=\int_\Omega \int_0^T\cC(X_t(\omega),\ubar u_t(X_t(\omega)),(X_t)_\sharp\P)\,\d t \,\d\P(\omega)
+\int_\Omega \cC_T(X_T(\omega),(X_T)_\sharp\P)\,\d\P(\omega),\]
and the corresponding \emph{value function} $V_{\FL}:L^p(\Omega;\R^d)\to[0,+\infty)$ by
\begin{equation}\label{eq:valueFRL}
V_{\FL}(X_0):=\inf\left\{J_{\FL}(X,\ubar u)\,:\,(X,\ubar u)\in\cA_{\FL}(X_0)\right\}.
\end{equation}
\smallskip
 In the following, $\FL(\Omega, \frB,\P)$ denotes the Feedback Lagrangian problem given in Definition \ref{def:FL}.
We short the notation to $\FL$ when the probability space is clear from the context.  
\end{definition}

\medskip
\begin{remark}
Observe that, choosing a constant (feedback) control function $\ubar u$, from Proposition \ref{prop:existL} 
it is immediate to prove that $\cA_{\FL}(X_0)\not=\emptyset$.
In general, given  $\ubar u\in\Bor([0,T]\times\R^d;U)$, the existence and uniqueness of solutions to the 
Cauchy problem \eqref{eq:systemFRL} is not guaranteed.
\end{remark}
\medskip

The following result follows directly from Definitions \ref{def:L}, \ref{def:E1} and \ref{def:FL}.

\begin{proposition}\label{prop:equivFRL}
Let $\S=(U,f,\cC,\cC_T)$  satisfy Assumption \ref{BA} and $(\Omega, \frB, \P)$ be a probability space.
Let $X_0\in L^p(\Omega;\R^d)$. If $(X,\ubar u)\in\cA_{\FL}(X_0)$, then
\begin{itemize}
\item[(i)] defining $u:[0,T]\times\Omega\to U$ by $u(t,\omega):=\ubar u(t,X_t(\omega))$, 
we have that $(X,u)\in\cA_{\pL}(X_0)$ and $J_{\FL}(X,\ubar u)=J_{\pL}(X,u)$. In particular it holds $V_{\FL}(X_0)\ge V_{\pL}(X_0)$.
\item[(ii)] defining $\mu_t:=(X_t)_\sharp\P$ for any $t\in[0,T]$, 
we have that $(\mu,\ubar u)\in\cA_{\E}((X_0)_\sharp\P)$ 
and $J_{\FL}(X,\ubar u)=J_{\E}(\mu,\ubar u)$.
 In particular it holds $V_{\FL}(X_0)\ge V_{\E}(\mu_0)$.
\end{itemize}
\end{proposition}

Taking advantage of Proposition \ref{prop:equivFRL} and of the discrete superposition principle given in Theorem \ref{lem:EN}, we have the following equivalence result between $\FL^N:=\FL(\Omega^N,\Parts(\Omega^N),\P^N)$ and $\E^N$ defined in Definition \ref{def:E^N}.

\begin{proposition}\label{prop:FLN<E}
Let $\S = (U,f,\cC,\cC_T)$ satisfy Assumption \ref{BA}.
Let $p\geq 1$, $X_0\in L^p(\Omega^N;\R^d)$ and   $\mu_0:=\frac{1}{N}\sum_{\omega=1}^N\delta_{X_0(\omega)}=(X_0)_\sharp\P^N$.
If $(\mu,\ubar u)\in\cA_{\E^N}(\mu_0)$, 
then there exists $X\in C([0,T];L^p(\Omega^N))$ 
such that $\mu_t=(X_t)_\sharp\P^N$ for all $t\in[0,T]$,
$(X,\ubar u)\in\cA_{\FL^N}(X_0)$ and
 $J_{\E^N}(\mu,\ubar u)=J_{\FL^N}(X,\ubar u)$.
Moreover, $V_{\E^N}(\mu_0)=V_{\FL^N}(X_0)$.
\end{proposition}
\begin{proof}
Let $(\mu,\ubar u)\in\cA_{\E^N}(\mu_0)$.
By the superposition principle given in Theorem \ref{lem:EN} applied to $\mu$ and the vector field $v(t,x)=f(x,\ubar u(t,x),\mu_t)$,
there exists $\eta \in \PP(\Gamma_T)$ such that 
\[ \eta= \frac{1}{N} \sum_{\omega = 1}^N \delta_{\gamma_\omega}, \]
where $\gamma_\omega \in \Gamma_T$, $\omega = 1, \ldots, N$.
We define $X:[0,T]\times\Omega^N\to\R^d$ by 
$X(t,\omega):=\gamma_\omega(t)$ that satisfies $(X,u) \in \cA_{\FL^N}$ thanks to \eqref{eq:Cauchyend} in Theorem \ref{lem:EN}.
Hence it readily follows that $J_{\FL^N}(X,\ubar u) =J_{\E^N}(\mu,\ubar u)$ and $V_{\FL^N}(X_0)\le V_{\E^N}(\mu_0)$.
By Proposition \ref{prop:equivFRL} we have that $V_{\FL^N}(X_0)\ge V_{\E^N}(\mu_0)$ and we conclude the proof.
\end{proof}

\subsection{Finite particle approximation for $\L$}\label{sec:gamma}

The aim of this section is to approximate a general Lagrangian problem $\pL=\pL(\Omega,\frB,\P)$ with finite particle Lagrangian problems $\pL^N=\pL(\Omega^N,\Parts(\Omega^N),\P^N)$, $N \in \N$, where $(\Omega^N,\Parts(\Omega^N),\P^N)$ is defined in \eqref{def:omegaN}.
A first result in this direction has been already obtained in Proposition  \ref{prop:approx_lagrangian_n} (see also Remark \ref{rmk:noequiprob}) in  Section \ref{sec:approxLpiecew}. 
Here, we specialize the result in the case of equally distributed masses which is suitable for the application to a finite  particle/agent model.

Recall that if $(\Omega,\frB,\P)$ is a standard Borel space and $\P$ is without atoms, thanks to item (ii) in Proposition \ref{prop:FAP} there exists a family of finite algebras $\frB^N\subset\frB$, $N\in\N$, satisfying the finite approximation property of Definition \ref{approx_prop} and 
$\P(A^N_k)=\frac1N$,  with $k=1,\ldots,N$. 
Recall the definition of $\psi^N$, $\phi^N$ and $\cK^N$ in \eqref{def:psinphin} and \eqref{def:Kn}, respectively. 

\medskip
 
A Gamma-convergence result for the functional $J_\pL$ is given in the following proposition.

\begin{proposition}[Finite particle approximation for $\pL$]\label{prop:gamma_conv}
Let $\S = (U,f,\cC,\cC_T)$  satisfy Assumption \ref{BA}. 
Let $(\Omega,\frB,\P)$ be a standard Borel space such that 
$\P$ is without atoms.
The following holds:
\begin{itemize}
\item[(i)] Suppose that $(X,u) \in L^p(\Omega;\AC^p([0,T];\R^d)) \times \BM([0,T] \times \Omega; U)$ and
$(X^N,u^N) \in C([0,T];L^p(\Omega^N;\R^d)) \times \BM([0,T] \times \Omega^N; U)$ such that 
\begin{equation*}
\lim_{N \to +\infty} \cK^N(X^N, u^N) = (X,u), \quad \text{ in } C([0,T];L^p(\Omega;\R^d)) \times L^1([0,T] \times \Omega; U).
\end{equation*}
Then 
\[ \lim_{N \to +\infty} J_{\pL^N}(X^N,u^N)  = J_{\pL}(X,u).\]
\item[(ii)] Assume that $U$ is a compact convex subset of a separable Banach space $V$. 
Suppose that  $X_0 \in L^p(\Omega;\R^d)$ and $(X,u) \in \mathcal{A}_{\pL}(X_0)$.
If $X_0^N \in L^p(\Omega^N;\R^d)$ such that 
\begin{equation}\label{CNOmega}
	 \lim_{N \to +\infty} \| X^N_0 \circ \psi^N - X_0\|_{L^p(\Omega;\R^d)} = 0
 \end{equation}
then there exists a sequence 
$(X^N,u^N) \in \mathcal{A}_{\pL^N}(X^N_0)$ such that 
\begin{equation*}
\lim_{N \to +\infty} \cK^N(X^N, u^N) = (X,u), \quad \text{ in } C([0,T];L^p(\Omega;\R^d)) \times L^1([0,T] \times \Omega; U)
\end{equation*}
and
\[ \lim_{N \to +\infty} J_{\pL^N}(X^N,u^N)  = J_{\pL}(X,u). \]
\end{itemize}
\end{proposition}
\begin{proof}
Thanks to Proposition \ref{prop:equivSpaces} it holds that $X \in C([0,T]; L^p(\Omega;\R^d))$. 
Item (i) can be proved exactly by the same technique used in the second part of the proof of Proposition \ref{prop:costRLntoRL} applied to the sequence $\cK^N(X^N,u^N)$ and recalling that $J_{\pL}\left(\cK^N(X^N,u^N)\right) = J_{\pL^N}(X^N, u^N)$ (see Proposition \ref{prop:correspLdiscrete}).
Notice that, since $U$ is metrizable and compact the convergence $u^N \to u$ in $L^1([0,T] \times \Omega;U)$ is equivalent to the convergence in $(\cL_T \otimes \P)$-measure. 
Item (ii), is a direct application of Proposition \ref{prop:approx_lagrangian_n} to the sequence of finite algebras $\frB^N$ given in item (ii) of Proposition \ref{prop:FAP}.
\end{proof}

\begin{proposition}[Convergence of the value functions]\label{prop:convVL}
Let $\S = (U,f,\cC,\cC_T)$  satisfy Assumption \ref{BA} with $U$ a compact convex subset of a separable Banach space $V$.
Let $(\Omega,\frB,\P)$ be a standard Borel space such that $\P$ is
without atoms. 
If $X_0 \in L^p(\Omega;\R^d) $ and $X_0^N\in L^p(\Omega^N;\R^d) $ satisfy 
\begin{equation}\label{eq:conv_x0n}
\lim_{N \to+\infty} \| X^N_0 \circ \psi^N - X_0\|_{L^p(\Omega;\R^d)} = 0,
\end{equation}
then 
\begin{equation*}
\limsup_{N \to +\infty} V_{\pL^N} (X_0^N) \leq V_{\pL}(X_0).
\end{equation*}
Moreover, if  $\S =(U,f,\cC,\cC_T)$ satisfies the Convexity Assumption \ref{CA}, then
\begin{equation*}
\liminf_{N \to +\infty} V_{\pL^N} (X_0^N) \geq V_{\pL}(X_0).
\end{equation*} 
In particular, under the Convexity Assumption \ref{CA}, 
\[\lim_{N \to +\infty} V_{\pL^N} (X_0^N) = V_{\pL}(X_0).\]
\end{proposition}

\begin{proof}
By definition of inf, for every $\varepsilon>0$ there exists $(X_\varepsilon,u_\varepsilon)\in\mathcal A_{\pL}(X_0)$ such that $V_{\pL}(X_0) \geq J_{\pL}(X_\varepsilon,u_\varepsilon) - \varepsilon$. Moreover from item (ii) in Proposition \ref{prop:gamma_conv} there exists $(X^N_\varepsilon,u^N_\varepsilon)$ such that $J_{\pL^N}(X^N_\varepsilon,u^N_\varepsilon) \to J_{\pL}(X_\varepsilon,u_\varepsilon)$, as $N\to +\infty$. Hence
\[ \limsup_{N\to +\infty}V_{\pL^N}(X_0^N) \leq \limsup_{N\to +\infty} J_{\pL^N}(X^N_\varepsilon,u^N_\varepsilon) = J_{\pL}(X_\varepsilon,u_\varepsilon) \leq V_{\pL}(X_0) + \varepsilon. \]
By the arbitrariness of $\varepsilon$ we conclude.\\
By Proposition \ref{prop:LN>EN} and \eqref{eq:VEN>VE} we get 
\begin{equation*}
V_{\pL^N}(X_0^N) \geq V_{\E^N}(\mu_0^N) \geq V_{\E}(\mu_0^N).
\end{equation*} 
In the convex setting, by the lower semicontinuity of the value function $V_{\E}$ (see Proposition \ref{prop:lscVE}) 
and by Corollary \ref{cor:VL=VE} ($\P$ is without atoms by assumption) we have the desired convergence. 
\end{proof}

\begin{remark}
A natural choice for $X^N_0$ in Proposition \ref{prop:convVL} is given by $X_0^N:= \tilde X^N_0 \circ \phi^N$ where 
\[ \tilde X^N_0:=  \sum_{k=1}^N \mathds{1}_{A_k^N} \, \fint_{A_k^N} X_0(\omega)\, \d\P(\omega). \] 
For a proof of the convergence \eqref{eq:conv_x0n} we refer to Lemma \ref{lemma:partL1} and Proposition \ref{prop:BeqL1} in  Appendix \ref{app_FP}.
\end{remark}

\subsection{Finite particle approximation for $\E$}\label{sec:gammaE} 	
In this section, we show that the Eulerian problem $\E$ can be approximated by finite particle Eulerian problems $\E^N$ defined in Definition \ref{def:E^N}.
Thanks to Theorem \ref{cor:LN=EN}, we are able to approximate the Eulerian problem also with a sequence of finite particle Lagrangian problems $\pL^N=\pL(\Omega^N,\Parts(\Omega^N),\P^N)$, $N \in \N$, with $(\Omega^N,\Parts(\Omega^N),\P^N)$ as in \eqref{def:omegaN}. This is relevant from the point of view of applications.  
The main result of the section is stated in the following theorem.

\begin{theorem}[Convergence of the value functions]\label{thm:convVE}
Let $\S = (U,f,\mathcal C,\mathcal C_T)$ satisfy the Convexity Assumption \ref{CA}. 
Let $\mu_0 \in \PP_p(\R^d)$ and $\mu_0^N \in \PP^N(\R^d)$
such that $W_p(\mu_0^N,\mu_0)\to 0$ as $N\to+\infty$,
then 
\[\lim_{N \to +\infty} V_{\E^N} (\mu_0^N) = V_{\E}(\mu_0).\]
Moreover, for every $X_0^N \in L^p(\Omega^N;\R^d)$ such that $(X_0^N)_\sharp \P^N = \mu_0^N $ it holds that  
\[\lim_{N \to +\infty} V_{\pL^N} (X_0^N) = V_{\E}(\mu_0).\]
\end{theorem}

In order to prove Theorem \ref{thm:convVE}, we start with the following proposition.

\begin{proposition}[Finite particle approximation for $\E$]\label{prop:ENtoE}
Let $\S=(U,f,\cC,\cC_T)$ satisfy the Convexity Assumption \ref{CA}. 
Let $\mu_0\in \PP_p(\R^d)$ and $\mu_0^N \in \PP^N(\R^d)$,
$N\in\N$, such that 
\begin{equation}\label{eq:conv_mu0N}
\lim_{N\to +\infty}W_p(\mu_0^N,\mu_0)= 0.
\end{equation}
If $(\mu,\ubar u)\in\cA_{\E}(\mu_0)$,
then there exists a sequence $(\mu^N,\ubar u^N)\in\cA_{\E^N}(\mu_0^N)$ such that
\begin{enumerate}
\item $(\mu^N,\ubar u^N)$ converges to $(\mu,\ubar u)$ according to Definition \ref{def:nu_converges};
\item $\displaystyle\lim_{N\to+\infty} J_{\E^N}(\mu^N,\ubar u^N) = J_{\E}(\mu,\ubar u)$.
\end{enumerate}
\end{proposition}
\begin{proof}
{\bf Step 1.} 
Let $(\mu,\ubar u)\in\cA_{\E}(\mu_0)$.
In this step we associate to $(\mu,\ubar u)$ an admissible pair  $(\tilde X, \tilde u) \in \cA_{\pL}(\tilde X_0)$ for the Lagrangian problem $\pL:= \pL([0,1];\cB_{[0,1]},\cL_1)$ such that 
\begin{equation}\label{eq:equivLE}
J_{\pL}(\tilde X,\tilde u)=J_{\E}(\mu,\ubar u).
\end{equation}
 
By Proposition \ref{prop:E>K}, there exists $(\eeta,u)\in\cA_{\K}(\mu_0)$  such that 
$u(t,\gamma)=\ubar u(t,\gamma(t))$, $(e_t)_\sharp\eeta=\mu_t$
and $J_{\K}(\eeta,u)=J_{\E}(\mu,\ubar u)$.
Thanks to Lemma \ref{l:fromKtoL},
the map $Z:[0,T]\times\Gamma_T\to\R^d$ defined by 
$Z(t,\gamma):=\gamma(t)=e_t(\gamma)$ satisfies $(Z,u)\in\cA_{\pL_{\eeta}}(e_0)$  with $\pL_{\eeta}:=\pL(\Gamma_T,\cB_{\Gamma_T},\eeta)$ and 
\begin{equation}\label{eq:JLetaeqJK1}
	J_{\pL_{\eeta}}(Z,u)=J_{\K}(\eeta,u),
\end{equation} 
By Proposition \ref{prop:Skorohod} applied to $S= \Gamma_T$ and $\nu = \eeta$, there exists  a Borel map $P:[0,1]\to\Gamma_T$ such that $P_\sharp\cL_1=\boldsymbol\eta$.

We define 
 $\tilde X:[0,T]\times[0,1]\to\R^d$ by $\tilde X_t(\omega):=Z_t(P(\omega))$ 
 and $\tilde u:[0,T]\times[0,1]\to U$ by $\tilde u_t(\omega):=u_t(P(\omega))$.
Notice that 
\begin{equation}\label{eq:mu_t:P}
(\tilde X_t)_\sharp\cL_1=(e_t\circ P)_\sharp\cL_1=(e_t)_\sharp\eeta=\mu_t,
\end{equation}
and it is easy to prove that
$(\tilde X,\tilde u)\in\mathcal A_{\pL}(\tilde X_0)$.
Moreover $J_{\pL}(\tilde X,\tilde u)=J_{\pL_{\eeta}}(Z,u)$
and, by \eqref{eq:JLetaeqJK1}, we obtain \eqref{eq:equivLE}.

\medskip

{\bf Step 2.} 
We use the partition of $[0,1]$ defined in Lemma \ref{lemma:partL1}.
We define the piecewise constant  initial data $\tilde X_0^N:[0,1]\to\R^d$ and 
controls $\tilde u^N:[0,T]\times[0,1]\to U$ by
\begin{align*}
\tilde X_0^N&:=\sum_{k=1}^{N}\mathds 1_{I^N_k} \,\fint_{I^N_k}\tilde X_0(\omega)\,\d\cL_1(\omega),\\
\tilde u_t^N&:=\sum_{k=1}^{N}\mathds 1_{I^N_k} \, \fint_{I^N_k}\tilde u_t(\omega)\,\d\cL_1(\omega)\, , \quad \textrm{for all }t\in[0,T].
\end{align*}
From the definition of $\tilde X_0^N$ and Lemma \ref{lemma:partL1} we have
\begin{equation}\label{eq:wpn2}
	W_p(\tilde \mu_0^N,\mu_0)\leq \|\tilde X_0^N-\tilde X_0\|_{L^p([0,1];\R^d)}\to 0 \qquad \text{as } N\to+\infty,
\end{equation}
where 
\[\tilde \mu_0^N:=(\tilde X_0^N)_\sharp\cL_1=\frac{1}{N}\sum_{k=1}^N\delta_{\tilde x^N_k}, \qquad \tilde x^N_k:=\fint_{I^N_k}\tilde X_0(\omega)\,\d\cL_1(\omega)\in\R^d.\]

Since $\mu_0^N=\frac{1}{N}\sum_{k=1}^N\delta_{x^N_k}$, 
there exists a permutation of indexes $\sigma^N:\{1,\dots,N\}\to\{1,\dots,N\}$ and a map $X_0^N:[0,1]\to\R^d$ defined by $X_0^N:=\sum_{k=1}^{N} x^N_{\sigma^N(k)} \mathds 1_{I^N_k}$ such that 
\begin{equation}\label{eq:wpn}
	W_p(\mu_0^N,\tilde\mu_0^N)=\|X_0^N-\tilde X_0^N\|_{L^p([0,1];\R^d)}.
\end{equation}
Using \eqref{eq:conv_mu0N}, \eqref{eq:wpn2} and \eqref{eq:wpn} we obtain
 \begin{equation}
 \begin{split}
	\|X_0^N-\tilde X_0\|_{L^p([0,1];\R^d)} &\leq \|X_0^N-\tilde X^N_0\|_{L^p([0,1];\R^d)} + \|\tilde X_0^N-\tilde X_0\|_{L^p([0,1];\R^d)}\\
	&= W_p(\mu_0^N,\tilde\mu_0^N) + \|\tilde X_0^N-\tilde X_0\|_{L^p([0,1];\R^d)}\\
	&\leq W_p(\mu_0^N,\mu_0) + W_p(\mu_0,\tilde\mu_0^N) + \|\tilde X_0^N-\tilde X_0\|_{L^p([0,1];\R^d)}
	\to0.
\end{split}
\end{equation} 

Let $(X^N,\tilde u^N)\in\cA_{\pL}(X_0^N)$.
By Lemma \ref{lemma:partL1} and dominated convergence we have that
$\|\tilde u^N-\tilde u\|_{L^1([0,1] \times [0,T];V)}\to0$ as $N\to+\infty$.
Then, by Proposition \ref{prop:costRLntoRL}, we have
 \begin{equation}\label{convAUX}
	\sup_{t\in[0,T]}\|X_t^N-\tilde X_t\|_{L^p([0,1];\R^d)}\to 0 \qquad \mbox{ as } N\to+\infty
\end{equation}
 and
\begin{equation}\label{eq:convJLnL2}
	\lim_{N\to+\infty}J_{\pL}(X^N,\tilde u^N)=J_{\pL}(\tilde X,\tilde u).
\end{equation}
We observe that, for any $t\in[0,T]$, $X_t^N$ is constant on the elements $I_k^N$ 
of partition $\{I_k^N:k=1,\ldots,N\}$ so that it is of the form
$$X_t^N:=\sum_{k=1}^{N} (x_t^N)_k \mathds 1_{I^N_k}$$
for some $ (x_t^N)_k\in\R^d$.

We define $\mu_t^N:=(X_t^N)_\sharp\cL_1\in\PP_p(\R^d)$.
From the observation above,
$$\mu_t^N=\frac1N\sum_{k=1}^{N} \delta_{(x_t^N)_k} $$

By \eqref{convAUX} we obtain that $\mu^N\to\mu$ in $C([0,T];\PP_p(\R^d))$ as $N\to+\infty$. 

\medskip

{\bf Step 3.} 
For any $t\in[0,T]$ we define $\rho_t^N:=(X_t^N,\tilde u_t^N)_\sharp\cL_1\in\PP(\R^d\times U)$, and notice that 
$\pi^1_\sharp\rho_t^N=\mu_t^N$.
Denoting by $\rho_{t,x}^N\in\PP(U)$ the disintegration of $\rho_t^N$ w.r.t. $\pi^1$,
we define the Borel map $\ubar u^N:[0,T]\times\R^d\to U$  by 
$$\ubar u^N(t,x):=\int_{U}u\,\d\rho_{t,x}^N(u).$$
By the definition of $\mu_t^N$ and $\ubar u^N$, from item (2) of the Convexity Assumption \ref{CA},
we obtain that  $(\mu^N,\ubar u^N)\in\cA_{\E^N}(\mu_0^N)$.
Indeed, given $\varphi\in C^1_c(\R^d;\R)$, for $\cL_T$-a.e. $t\in[0,T]$, we have
\begin{align*}
\frac{\d}{\d t}\int_{\R^d}\varphi(x)\,\d\mu_t^N(x)&=\frac{\d}{\d t}\int_{[0,1]}\varphi(X_t^N(\omega))\,\d\cL_1(\omega) \\
&=\int_{[0,1]}\nabla\varphi(X_t^N(\omega))\cdot \dot X_t^N(\omega)\,\d\cL(\omega)\\
&=\int_{[0,1]}\nabla\varphi(X_t^N(\omega))\cdot f(X_t^N(\omega),\tilde u_t^N(\omega),\mu_t^N)\,\d\mathcal L(\omega)\\
&=\int_{\R^d\times U}\nabla\varphi(x)\cdot f(x,u,\mu_t^N)\,\d\rho_t^N(x,u)\\
&=\int_{\R^d}\nabla\varphi(x)\cdot\int_U f(x,u,\mu_t^N)\,\d\rho_{t,x}^N(u)\,\d\mu_t^N(x)\\
&=\int_{\R^d}\nabla\varphi(x)\cdot f\left(x,\int_U u\,\d\rho_{t,x}^N(u),\mu_t^N\right)\,\d\mu_t^N(x)\\
&=\int_{\R^d}\nabla\varphi(x)\cdot f\left(x,\ubar u^N(t,x),\mu_t^N\right)\,\d\mu_t^N(x).
\end{align*}

Let us conclude the proof of the convergence showing that \eqref{convweaknu} holds. 
For any $t\in[0,T]$, using $(\tilde X,\tilde u)$ introduced in Step 1, we define $\rho_t:=(\tilde X_t,\tilde u_t)_\sharp\cL_1\in \PP(\R^d\times U)$.
By the convergence \eqref{convAUX} and the convergence of $\tilde u^N$ to $\tilde u$, it easily follows that
$\rho^N:=\rho^N_t\otimes\cL_T$ weakly converges to $\rho:=\rho_t\otimes\cL_T$  in $\PP([0,T]\times\R^d\times U)$.

Let now $\phi\in C_c([0,T]\times\R^d;V')$. Using the definition of $\rho^N$ and the weak convergence of $\rho^N$ to $\rho$ we have that
\begin{align*}
\lim_{N\to+\infty}\int_0^T\int_{\R^d}\langle\phi(t,x),\bar u^N(t,x)\rangle\,\d\mu_t^N(x)\,\d t 
&=\lim_{N\to+\infty}\int_0^T\int_{\R^d}\left\langle\phi(t,x),\int_U u\,\d\rho^N_{t,x}(u)\right\rangle\,\d\mu^N_t(x)\,\d t\\
&=\lim_{N\to+\infty}\int_0^T\int_{\R^d}\int_U\left\langle\phi(t,x), u \right\rangle\,\d\rho^N_{t,x}(u)\,\d\mu^N_t(x)\,\d t\\
&=\lim_{N\to+\infty}\int_{[0,T]\times\R^d\times U}\langle\phi(t,x), u\rangle\,\d\rho^N(t,x,u),\\
&=\int_{[0,T]\times\R^d\times U}\langle\phi(t,x) ,u\rangle \,\d\rho(t,x,u),\\
\end{align*}
Using $\eeta \in \PP(C([0,T];\R^d))$, $P: [0,1] \to \Gamma_T$ and \eqref{eq:mu_t:P} introduced in Step 1,  recalling that $\rho = (\tilde X_t, \tilde u_t)_\sharp\cL_1 \otimes \cL_T$ we get  
\begin{align*}
\int_{[0,T]\times\R^d\times U}\langle\phi(t,x) ,u\rangle \,\d\rho(t,x,u) &=\int_0^T\int_0^1\langle\phi(t,\tilde X_t(\omega)),\tilde u_t(\omega)\rangle\,\d\cL_1(\omega)\,\d t\\
&=\int_0^T\int_0^1\langle\phi(t,e_t(P(\omega))), u_t(P(\omega))\rangle\,\d\cL_1(\omega)\,\d t\\
&=\int_0^T\int_{\Gamma_T}\langle\phi(t,e_t(\gamma)) ,u_t(\gamma)\rangle\,\d\eeta(\gamma)\,\d t\\
&=\int_0^T\int_{\Gamma_T}\langle\phi(t,e_t(\gamma)),\ubar u(t,e_t(\gamma))\rangle\,\d\eeta(\gamma)\,\d t\\
&=\int_0^T\int_{\R^d}\langle\phi(t,x) ,\ubar u(t,x)\rangle\,\d\mu_t(x)\,\d t.
\end{align*}

\medskip

{\bf Step 4.} 
Finally we prove that $J_{\E^N}(\mu^N,\ubar u^N) \to J_{\E}(\mu,\ubar u)$, as $N \to +\infty$.\\
Using item (3) of the Convexity Assumption \ref{CA}, we have
\begin{align*}
J_{\E^N}(\mu^N,\ubar u^N)
&=\int_0^T\int_{\R^d}\cC(x,\bar u^N(t,x),\mu_t^N)\,\d\mu^N_t(x)\,\d t + \int_{\R^d}\cC_T(x,\mu_T^N)\,\d\mu_T^N(x)\\
&=\int_0^T\int_{\R^d}\cC\left(x,\int_U u\,\d\rho_{t,x}^N(u),\mu_t^N\right)\,\d\mu_t^N(x)\,\d t
+\int_{[0,1]}\cC_T(X_T^N(\omega),\mu_T^N)\,\d\cL_1(\omega)\\
&\le \int_0^T\int_{\R^d\times U}\cC(x,u,\mu_t^N)\,\d\rho_t^N(x,u)\,\d t+\int_{[0,1]}\cC_T(X_T^N(\omega),\mu_T^N)\,\d\cL_1(\omega)\\
&=\int_0^T\int_{[0,1]}\cC(X_t^N(\omega),\tilde u_t^N(\omega),\mu_t^N)\,\d\cL_1(\omega)\,\d t+\int_{[0,1]}\cC_T(X_T^N(\omega),\mu_T^N)\,\d\cL_1(\omega)\\
&= J_{\pL}(X^N,\tilde u^N).
\end{align*}
By the previous inequality, recalling \eqref{eq:convJLnL2} and \eqref{eq:equivLE}, we get
\[\limsup_{N\to+\infty}J_{\E^N}(\mu^N,\ubar u^N)\le\limsup_{N\to+\infty}J_{\pL}(X^N,\tilde u^N)=J_{\pL}(\tilde X,\tilde u)=J_{\E}(\mu,\ubar u).\]
By Proposition \ref{prop:lscJE} we conclude.
\end{proof}

We conclude this section with the proof of Theorem \ref{thm:convVE}.

\begin{proof}[Proof of Theorem \ref{thm:convVE}]
Observe that $V_{\E^N} (\mu_0^N) \geq V_{\E}(\mu_0^N)$, then by Proposition \ref{prop:lscVE} we obtain
\[\liminf_{N \to +\infty} V_{\E^N} (\mu_0^N) \ge V_{\E}(\mu_0).\]

Let us prove that $\limsup_{N \to +\infty} V_{\E^N} (\mu_0^N) \le V_{\E}(\mu_0)$. 
Fix $\varepsilon>0$. By definition of $V_\E$ there exists 
$(\mu_\varepsilon,\ubar u_\varepsilon)\in\mathcal A_{\E}(\mu_0)$ such that $V_{\E}(\mu_0) \geq J_{\E}(\mu_\varepsilon,\ubar u_\varepsilon) - \varepsilon$. 
By Proposition \ref{prop:ENtoE} there exists $(\mu^N_\varepsilon,\ubar u^N_\varepsilon)\in\mathcal A_{\E^N}(\mu_0^N)$ such that $\lim_{N\to+\infty}J_{\E^N}(\mu^N_\varepsilon,\ubar u^N_\varepsilon) = J_{\E}(\mu_\varepsilon,\ubar u_\varepsilon)$. Thus,
\[ \limsup_{N\to +\infty}V_{\E^N}(\mu_0^N) \leq \limsup_{N\to +\infty} J_{\E^N}(\mu^N_\varepsilon,\ubar u^N_\varepsilon) = J_{\E}(\mu_\varepsilon,\ubar u_\varepsilon) \leq V_{\E}(\mu_0) + \varepsilon. \]
By the arbitrariness of $\varepsilon$ we conclude.

Finally, thanks to Theorem \ref{cor:LN=EN} it holds that $V_{\E^N}(\mu_0^N) = V_{\pL^N}(X_0^N)$ for every $N \in \N$, so that $\lim_{N \to +\infty} V_{\pL^N} (X_0^N) = V_{\E}(\mu_0)$.
\end{proof}

\appendix
\section*{Appendixes}
We organize the material of the appendixes as follows.  
Appendix \ref{appendix_A}  deals with vector-valued Sobolev spaces and Cauchy problems for ODEs in Banach spaces. 
A further stability property of Cauchy problems is then established in Appendix \ref{appendix_B}.
In Appendix \ref{app:SP} we state and prove the superposition principle for the evolution of empirical measures.
Appendix \ref{app_FP} is devoted to the proof of Proposition \ref{prop:FAP} where we construct (equipartite) finite algebras satisfying the Finite Approximation Property of Definition \ref{approx_prop}.

\section{Ordinary differential equations in Banach spaces}\label{appendix_A}
Let $E$ be a Banach space with $\|\cdot\|$ the associated norm.
In the following, if  $u:[0,T]\to E$ is a Bochner integrable function,
we denote by $\int_0^Tu(t)\,\d t$ its Bochner integral.
We recall the following criterion of integrability:
$u:[0,T]\to E$ is Bochner integrable if and only if there exists a sequence
$u_n:[0,T]\to E$ of simple measurable functions such that
$\lim_{n\to+\infty} u_n(t)=u(t)$ for $\cL_T$-a.e. $t\in[0,T]$,
 and
$\int_0^T\|u(t)\|\,\d t<+\infty$.

We recall that, if $u:[0,T]\to E$ is Bochner integrable, then 
\begin{equation}\label{eq:SBochner}
	\left\|\int_a^b u(t)\,\d t \right\| \leq \int_a^b \|u(t)\|\,\d t, \qquad \text{for any }[a,b]\subseteq [0,T],
\end{equation}
\begin{equation}
	\lim_{h\to 0}\frac1h \int_t^{t+h} \|u(s)-u(t)\|\,\d s =0,  \qquad \text{for $\cL_T$-a.e. } t\in[0,T],
\end{equation}
and the above limit exists in every point of continuity of $u$.
Moreover, for every continuous linear operator $A: E \to \tilde E$, with $\tilde E$ a Banach space it holds that 
\begin{equation}\label{eq:operator_bochner}
A \left( \int_0^T u(t)\, \d t\right) =  \int_0^T A(u(t))\, \d t.
\end{equation}
For the definition of Bochner integral, properties and related proofs, see for instance \cite{diestel1977vector}.

We say that  $u \in W^{1,p}(0,T;E)$ if $u \in L^p(0,T;E)$ and there exists $g \in L^p(0,T;E)$ such that 
\begin{equation*}
\int_0^T \varphi'(t) u(t) \, \d t = -\int_0^T \varphi(t) g(t) \, \d t , \quad \forall \, \varphi \in C^\infty_c((0,T);\R).
\end{equation*}  
We recall a classical result (see e.g. \cite[Theorem~1.17]{barbu2010nonlinear}) 
\begin{proposition}\label{p:W=AC}
$u \in W^{1,p}(0,T;E)$ if and only if there exists $\tilde u \in \AC^p([0,T];E)$ such that $\tilde u(t) = u(t)$ and $\tilde u$ is differentiable for $\cL_T$-a.e. $t \in [0,T]$.
\end{proposition}

Let now $(\Omega,\frB,\P)$ be a probability space with $\Omega$ standard Borel.
Fix $T>0$, let $\Leb_{[0,T]}$ be the $\sigma$-algebra of Lebesgue measurable sets on $[0,T]$ and $\cL_T$ the normalized Lebesgue measure restricted to $[0,T]$. Consider $(\Omega_T, \frB_{T}, \mm)$ the product space with
$\Omega_T:=[0,T] \times \Omega$, endowed with the product 
$\sigma$-algebra $\frB_{T}=\Leb_{[0,T]} \otimes \frB$ and probability measure $\mm:=\cL_T\otimes\P\in\PP(\Omega_T)$.

\begin{lemma}\label{prop:BPoint}
Let $p\geq 1$, and $g \in L^p_\mm(\Omega_T; \R^d)$ and  $\hat g$ a Borel representative of $g$. 
Let $\tilde g$ the map defined by $\tilde g(t)(\omega) = \hat g(t,\omega)$ for every $(t,\omega) \in \Omega_T$. 
Then $\tilde g \in L^1(0,T;L^p(\Omega;\R^d))$ and,   
denoting by $G := \int_0^T \tilde g(t)\,\d t$,
it holds
\begin{equation}\label{BochnerPoint}
G(\omega) =\int_0^T \hat g(t,\omega)\,\d t, \quad \text{ for } \P\text{-a.e. } \omega \in \Omega.
\end{equation}
\end{lemma}
The proof of the Lemma follows by Fubini's theorem and the definition of Bochner integral (notice that \eqref{BochnerPoint} holds for simple functions).

\begin{proposition}\label{prop:equivSpaces}
Let $y\in L^p_\mm(\Omega_T; \R^d)$. 
The following are equivalent:
\begin{enumerate}
\item There exists $g \in L^p_\mm(\Omega_T; \R^d)$
such that
\begin{equation*}
\int_{\Omega_T}\eta'(t) \phi(\omega)\,y(t,\omega)\,\d\mm(t,\omega)=-\int_{\Omega_T}\eta(t) \phi(\omega)\,g(t,\omega)\,\d\mm(t,\omega),
\end{equation*}
for every $\eta\in C^1_c((0,T);\R)$ and $\phi\in\BM(\Omega;\R)$ bounded;
\item $y\in W^{1,p}(0,T ;L^p_\P(\Omega; \R^d))$;
\item there exists a Borel representative $\tilde y$ of $y$ 
such that $\tilde y\in \AC^p([0,T];L^p_\P(\Omega; \R^d))$ and differentiable for $\cL_T$-a.e. $t\in[0,T]$ (differentiability is redundant for $p >1$);
\item $y\in L^p_\P(\Omega; W^{1,p}(0,T; \R^d))$;
\item  there exists a Borel representative $\bar y$ of $y$ such that $\bar y\in L^p_\P(\Omega;\AC^p([0,T]; \R^d))$.
\end{enumerate}
Moreover,
\begin{itemize}
\item[(i)] If $(3)$ holds, there exists a Borel function $\tilde g \in L^p_\mm(\Omega_T; \R^d)$ such that $\tilde g(t,\cdot) = \tilde y'(t,\cdot)$ in $L^p_\P(\Omega;\R^d)$, for $\cL_T$-a.e. $t \in [0,T]$. 
Hence, for every $t\in[0,T]$
\begin{equation}\label{eq:tildey}
\tilde y(t,\cdot)=\tilde y(0,\cdot) +\int_0^t \tilde g(s,\cdot)\,\d s , \quad \text{ in } L^p_\P(\Omega).
\end{equation}
\item[(ii)] If $(5)$ holds, there exists a Borel function $\bar g \in L^p_\mm(\Omega_T;\R^d)$ such that for every $\omega \in \Omega$, $ \bar g(t,\omega) = (\bar y(\cdot,\omega))'(t)$, for $\cL_T$-a.e. $t\in[0,T]$. 
Hence, for every $\omega\in\Omega$ it holds
\begin{equation}\label{eq:bary}
\bar y(t,\omega)=\bar y(0,\omega)+\int_0^t \bar g(s,\omega) \d s ,\quad \text{for every }t\in[0,T].
\end{equation}
\item[(iii)] If one of the five conditions above is satisfied, then 
$g = \tilde g = \bar g$,  $\mm$-a.e. in $\Omega_T$.
\end{itemize}

\end{proposition}
\begin{proof}
$(5) \Rightarrow (4)$  and $(3) \Leftrightarrow (2)$ follow from Proposition \ref{p:W=AC}.
$(4) \Rightarrow (2)$ follows from the definition and Fubini's theorem.
$(2) \Rightarrow (1)$ is a consequence of property \eqref{eq:operator_bochner}. 
Finally, the proof of $(1) \Rightarrow (5)$ is contained in a more general form in \cite[Lemma 4.3]{orrieri2019variational}. 

Items $(i)$-$(ii)$ are a consequence of the definition of $\AC^p([0,T];L^p(\Omega;\R^d))$ and $\AC^p([0,T];\R^d)$, respectively.
To show $(iii)$, denote with $y_0(\omega)$ the trace of $y$ in $t =0$, which is well defined thanks to $(1)$.
Then, $y_0(\omega) = \tilde y(0,\omega) = \bar y(0,\omega) $ for $\P$-a.e. $\omega \in \Omega$ so that comparing \eqref{eq:tildey} and \eqref{eq:bary} it holds $\tilde g = \bar g$ $\mm$-a.e. in $\Omega_T$, thanks to \eqref{BochnerPoint}.
\end{proof}

\subsection{Cauchy problem in Banach spaces}
We are interested in a Cauchy problem of this form
\begin{equation}\label{cauchy:FE1}
\begin{cases}
 \dot z_t = F(t,z_t), \qquad \text{for $\cL_T$-a.e. } t\in[0,T]\\
  z_{t=0}= z_0,
\end{cases}
\end{equation}
where a Carath\'eodory function $F:[0,T]\times E\to E$ and $z_0\in E$ are given.
For $z:[0,T]\to E$ we frequently use the notation $z_t:=z(t)$.

\smallskip
In the following, we present some classical results concerning the Cauchy problem \eqref{cauchy:FE1} and we provide a sketch of 
their proofs.

\begin{proposition}\label{prop:EquivE}
Let $F:[0,T]\times E\to E$ be a Carath\'eodory function such that
\begin{equation}\label{IntegrabilityF}
	\int_0^T \|F(t,z_t)\|\,\d t <+\infty , \qquad \forall z\in C([0,T];E). 
\end{equation}
The following assertions are equivalent:
\begin{itemize}
	\item  $z\in C([0,T];E)$ satisfies
\begin{equation}\label{IntegralS}
	z_t=z_0+\int_0^t F(s,z_s)\,\d s, \qquad \forall\,t\in[0,T]. 
\end{equation}
\item $z\in \AC^1([0,T];E)$, it is differentiable for $\cL_T$-a.e. $t\in[0,T]$ and
\begin{equation*}
\begin{cases}
 \dot z_t = F(t,z_t), \qquad \text{for $\cL_T$-a.e. } t\in[0,T]\\
  z_{t=0}= z_0.
\end{cases}
\end{equation*}
\end{itemize}
\end{proposition}
Since $F$ is a Carath\'eodory function and the curve $t \mapsto z_t$ then   the map $t \mapsto F(t,z_t)$ is (strongly) measurable as a map with values in $L^p(\Omega;\R^d)$. 
Notice also that 
if $E$ satisfies the Radon-Nikodym property (for instance when $E$ is reflexive) then a.e. differentiability of $z$ in the second item is redundant.

\begin{theorem}\label{th:ExistenceODE}
Let $F:[0,T]\times E\to E$ be a Carath\'eodory function such that
\begin{equation}\label{LipFE}
	\|F(t,z^1)-F(t,z^2)\|\le \tilde L \|z^1-z^2\|, \qquad \forall\,(t,z^1),(t,z^2)\in[0,T]\times E, 
\end{equation}
for some $\tilde L >0$, and there exists $C_0 \geq 0$ such that
\begin{equation}\label{boundFE}
\sup_{t\in[0,T]}\|F(t,0)\| \leq C_0 < +\infty.
\end{equation}
Then, for any $z_0\in E$ there exists a unique $z\in \AC^\infty([0,T];E)$ and differentiable for $\cL_T$-a.e. $t\in[0,T]$
 such that
\begin{equation}\label{cauchy:FE}
\begin{cases}
 \dot z_t = F(t,z_t), \qquad \text{for $\cL_T$-a.e. } t\in[0,T]\\
  z(0) = z_0.
\end{cases}
\end{equation}

Moreover, the following estimates hold:
\begin{equation}\label{boundZt}
	\sup_{t\in[0,T]}\|z_t\|\leq e^{\tilde LT}\left(\|z_0\|+C_0T\right), 
\end{equation}
\begin{equation}\label{LipZt}
	\|z_t-z_s\|\le |t-s|\,\big(C_0+\tilde L e^{\tilde LT}\left(\|z_0\|+C_0T\right)\big)  \qquad \forall\, s,t\in[0,T]. 
\end{equation}
\end{theorem}
\begin{proof}
We provide only a sketch of the proof.

Assumptions \eqref{LipFE} and \eqref{boundFE} yield the following growth property
\begin{equation}\label{GrFE}
	\|F(t,z)\|\le C_0 +\tilde L \|z\|, \qquad \forall\,(t,z)\in[0,T]\times E. 
\end{equation}

We define the Banach space $(\mathscr S,\|\cdot\|_{\mathscr S})$ as follows
\[\SS:=\left\{z\in C([0,T];E)\,:\,\|z\|_{\SS}:=\sup_{t\in[0,T]}e^{-\tilde Lt}\|z_t\|<+\infty\right\},\]
and the operator
$S:\SS\to \SS$ by
\[S(y)_t=z_0+\int_0^t F(s,y_s)\,\d s.\]
By \eqref{LipFE} and \eqref{GrFE}, using \eqref{eq:SBochner}, 
it is classical to prove that $S$ is well defined and it is a contraction.
Then, by Banach fixed point Theorem
we get the existence and uniqueness of $z\in\mathscr S$ such that \eqref{IntegralS} holds.
The estimates \eqref{boundZt} and \eqref{LipZt} follow from \eqref{IntegralS}, \eqref{GrFE}, and Gronwall inequality.
Finally, $z$ belongs to $\AC^\infty([0,T];E)$ thanks to \eqref{LipZt}.
\end{proof}

\begin{proposition}\label{prop:StabilityODE}
Let $F,F^n:[0,T]\times E\to E$, $n\in\N$, be Carath\'eodory functions satisfying
\eqref{LipFE},\eqref{boundFE} with the same constant $\tilde L$ and $C_0$.
Let $z_0, z^n_0\in E$, $z \in \AC^\infty([0,T];E) $ the solution of \eqref{cauchy:FE} and  $z^n \in \AC^\infty([0,T];E)$ the solutions of
\begin{equation*}\label{cauchy:FE2}
\begin{cases}
 \dot z^n_t = F^n(t,z^n_t), \qquad \text{for $\cL_T$-a.e. } t\in[0,T]\\
  z^n(0) = z^n_0.
\end{cases}
\end{equation*}
Then
\begin{equation}\label{eq:stabE}
\sup_{t\in[0,T]}\| z^n_t-z_t \|\leq e^{\tilde LT} \left(\| z^n_0-z_0 \|+ \int_0^T \|F^n(t,z_t)-F(t,z_t) \|\,\d t \right).
\end{equation}
In particular, if $\lim_{n \to +\infty}\|z_0^n - z_0\| = 0$ and 
\begin{equation}\label{Resto}
	\lim_{n\to+\infty} \int_0^T  \|F^n(t,z_t)-F(t,z_t) \|\,\d t =0, 
\end{equation}
then
\begin{equation*}
\lim_{n\to+\infty}\sup_{t\in[0,T]}\| z^n_t-z_t \| =0.
\end{equation*}
\end{proposition}

\begin{proof}
We have
\begin{align*}
\|z^n_t-z_t\| &=
\left\|z^n_0-z_0 + \int_0^t (F^n(s,z^n_s) - F(s,z_s))\,\d s \right\|\\
&\le \|z^n_0-z_0\| + \int_0^t \|F^n(s,z^n_s) - F(s,z_s))\|\,\d s.
\end{align*}
Using \eqref{LipFE} it holds
\begin{align*}
&\|F^n(s,z^n_s) - F(s,z_s))\| \\ &\leq
\|F^n(s,z^n_s) - F^n(s,z_s))\| +
\|F^n(s,z_s) - F(s,z_s))\| \\
& \le \tilde L \|z^n_s - z_s\|+
\|F^n(s,z_s) - F(s,z_s))\| .
\end{align*}
The last two inequalities yield
\begin{equation*}
\|z^n_t-z_t\| \le
 \|z^n_0-z_0\|  +  \int_0^T \|F^n(s,z_s) - F(s,z_s))\|\,\d s +
\tilde L \int_0^t \|z^n_s - z_s\| \,\d s.
\end{equation*}
By Gronwall lemma we obtain \eqref{eq:stabE}.
\end{proof}

\begin{proposition}\label{prop:EquivLp}
Let $F:[0,T]\times L^p(\Omega;\R^d)\to L^p(\Omega;\R^d)$ be a Carath\'eodory function satisfying \eqref{boundFE} 
Let $Y_0\in L^p(\Omega;\R^d)$.
Then the following assertions are equivalent:
\begin{enumerate}
\item $Y\in C([0,T];L^p(\Omega;\R^d))$ and 
\begin{equation}\label{IntegralSLp}
	Y_t=Y_0+\int_0^t F(s,Y_s)\,\d s, \qquad \forall\,t\in[0,T];
\end{equation}
\item $Y\in \AC^p([0,T];L^p(\Omega;\R^d))$ (if $p=1$ it is also differentiable for a.e. $t\in[0,T]$) and
\begin{equation}\label{cauchy:FELp}
\begin{cases}
 \dot Y_t = F(t,Y_t), \qquad \text{for $\cL_T$-a.e. } t\in[0,T]\\
  Y_{t=0}= Y_0;
\end{cases}
\end{equation}
\item $Y\in L^p(\Omega;\AC^p([0,T];\R^d)))$ and 
\begin{equation}\label{IntegralSLpo}
	Y_t(\omega)=Y_0(\omega)+\int_0^t F(s,Y_s)(\omega)\,\d s, \qquad \forall\,t\in[0,T],\quad \text{for $\P$-a.e. }\omega \in\Omega,
\end{equation}
where $Y_t: \Omega \to \R^d$ is defined by $Y_t(\omega):= Y(t,\omega)$ for $\P$-a.e. $\omega \in \Omega$. 
\item $Y\in L^p(\Omega;\AC^p([0,T];\R^d)))$ and
for $\P$-a.e. $\omega\in\Omega$ it holds 
\begin{equation}\label{cauchy:FELpo}
\begin{cases}
 \dot Y_t(\omega) = F(t,Y_t)(\omega), \qquad \text{for $\cL_T$-a.e. } t\in[0,T]\\
  Y_{t=0}(\omega)= Y_0(\omega),
\end{cases}
\end{equation}
where $Y_t: \Omega \to \R^d$ is defined by $Y_t(\omega):= Y(t,\omega)$ for $\P$-a.e. $\omega \in \Omega$. 
\end{enumerate}
\end{proposition}
\begin{proof}

The assertions (1)-(2) and (3)-(4) are equivalent by Proposition \ref{prop:EquivE}.
The equivalence (1)-(3) is a consequence of the equivalences (3)-(5) in Proposition \ref{prop:equivSpaces} and items (i)-(ii)-(iii).
\end{proof}

\section{A convergence result for solutions of Cauchy problems}
\label{appendix_B}

We state and prove the following {well known result}, for sake of completeness.
\begin{lemma}\label{lemma:youngODE}
Let  $U$ be a Polish space. 
Let $u^n: [0,T] \to U $ be a sequence of $\cL_T$-measurable functions such that
$(i_{[0,T]},u^n)_\sharp\cL_T \xrightarrow{\mathcal{Y}} \nu_t\otimes\cL_T \in \PP([0,T]\times U)$.\\
Let $g: [0,T]\times (\R^d \times U) \to \R^d$ a Carath\'eodory function
such that
\begin{equation}\label{Lipg}
	|g(t,x_1,u)-g(t,x_2,u)|\le L |x_1-x_2|, \qquad \forall\,(t,x_1,u),(t,x_2,u)\in[0,T]\times\R^d\times U, 
\end{equation}
for some $L>0$, and
\begin{equation}\label{boundg0}
	C_0:=\sup_{(t,u)\in[0,T]\times U}|g(t,0,u)| <+\infty.
\end{equation}
Given $X_0\in\R^d$ and $n \in \N$, we denote by $X^n\in \AC^1([0,T];\R^d)$ the unique solution of the Cauchy problem
\begin{equation}\label{cauchy:un}
\begin{cases}
 \dot X^n_t = g(t,X^n_t, u^n_t), &\textrm{for a.e. }t\in(0,T) \\
  X^n(0) = X_0,
\end{cases}
\end{equation}
and by $X\in \AC^1([0,T];\R^d)$ the unique solution of the Cauchy problem
\begin{equation}\label{cauchy:nu}
\begin{cases}
 \dot X_t = \displaystyle\int_{U} g(t, X_t,u) \,\d \nu_t(u), &\textrm{for a.e. }t\in(0,T) \\
  X(0) = X_0.
\end{cases}
\end{equation}
Then
\begin{equation}\label{Uconv}
\lim_{n \to +\infty} \sup_{t \in [0,T]}|X^n_t - X_t| = 0.
\end{equation}
\end{lemma}

\begin{proof}
Observe that existence and uniqueness of solutions of Cauchy problems \eqref{cauchy:un} and \eqref{cauchy:nu} is consequence of 
the fact that $G^n,G:[0,T]\times\R^d\to \R^d$, defined by $G^n(t,x):=g(t,x,u^n_t)$ and $G(t,x):=\displaystyle\int_{U} g(t,x,u) \,\d \nu_t(u)$,
are Charath\'eodory and $L$-Lipschitz continuous w.r.t. $x\in\R^d$.\\
Moreover, by \eqref{Lipg} and \eqref{boundg0} it holds
\begin{equation}\label{Ling}
	|g(t,x,u)|\le C_0+L |x|, \qquad \forall\, (t,x,u)\in[0,T]\times\R^d\times U.
\end{equation}
We define $Y^n\in C([0,T];\R^d)$ by
\begin{equation}\label{cauchy:unt}
 Y^n_t :=X_0+ \int_0^t g(s,X_s, u^n_s)\,\d s.
\end{equation}
From the convergence $(i_{[0,T]}, u^n)_\sharp \cL_T \xrightarrow{\mathcal{Y}} \nu \in \PP([0,T]\times U)$ it follows that
\begin{equation}\label{YntoX}
\lim_{n \to +\infty}Y^n_t = X_0 + \int_0^t \int_U g(s,X_s,u) \,\d \nu_s(u) \,\d s =X_t, \qquad \forall\, t\in [0,T].
\end{equation}
Denoting by $C:=\sup_{s\in[0,T]}|X_s|$, from \eqref{boundg0} and \eqref{Lipg} it is simple to prove that
\begin{equation}\label{eq:EBXn}
|Y^n_t|\le (|X_0|+(C_0+LC)T), \qquad \forall t\in[0,T]
\end{equation}
and 
\begin{equation}\label{eq:ELXn}
|Y^n_t-Y^n_s|\le |t-s|(C_0+LC)T, \qquad \forall t,s\in[0,T].
\end{equation}
By \eqref{eq:EBXn} and \eqref{eq:ELXn},  Ascoli-Arzel\`a theorem and \eqref{YntoX} imply that
\begin{equation}\label{YntoXunif}
\lim_{n \to +\infty} \sup_{t \in [0,T]}|Y^n_t - X_t| = 0.
\end{equation}
Since
\begin{equation*}
\begin{split}
|X^n_t - X_t| &\leq |X^n_t - Y^n_t|+|Y^n_t - X_t| \leq 
\int_0^t \left| g(s,X^n_s,u^n_s) - g(s,X_s,u^n_s)\right| \,\d s +|Y^n_t-X_t| \\
&\leq L \int_0^t |X^n_s-X_s| \,\d s + \sup_{s \in [0,T]} |Y^n_s-X_s|,
\end{split}
\end{equation*}
by Gronwall inequality we have that
\begin{equation}\label{Gw}
|X^n_t - X_t| \leq e^{Lt}\sup_{s \in [0,T]} |Y^n_s-X_s|.
\end{equation}
The convergence \eqref{Uconv} follows from \eqref{Gw} and \eqref{YntoXunif}.
\end{proof} 

\section{An empirical Superposition Principle}\label{app:SP}
In this  appendix, we give a refined version of the Superposition Principle (see Theorem \ref{thm:sup_princ} for the classical result) in the case of trajectories of the form $\mu_t\in\PP^N(\R^d)$ for any $t\in[0,T]$, where $\PP^N(\R^d)$ is the space of empirical probability measures
\[
\PP^N(\R^d):= \left\lbrace	\mu = \frac{1}{N} \sum_{i =1}^N \delta_{x_i} \; \text{ for some } x_i \in \R^d \right\rbrace.
\]

The novelty consists in proving that  if $\mu_t \in \PP^N(\R^d)$ for every $t \in [0,T]$, then there exists  representative $\eta \in \PP^N(\Gamma_T)$.
This result has been used to prove Proposition \ref{prop:FLN<E} and Corollary \ref{cor:LN=EN}. 

\begin{theorem}\label{lem:EN}
Let $N \in \N$ and $\mu \in \AC([0,T];\PP_1(\R^d))$ such that $\mu_t \in \PP^N(\R^d)$ for every $t\in[0,T]$.
\begin{enumerate}
\item 
There exists a unique (up to $\cL_T\otimes \mu_t$-negligible sets) 
Borel vector field $v:[0,T]\times\R^d\to\R^d$ satisfying
\begin{equation}\label{eq:Summabilityv}
\int_0^T\int_{\R^d}|v_t(x)|\,\d\mu_t(x)\,\d t<+\infty,
\end{equation}
such that $\mu$ is a distributional solution of the continuity equation
\begin{equation}\label{eq:CIend}
\partial_t\mu_t+\mathrm{div}(v_t\mu_t)=0,\quad \text{in }[0,T]\times\R^d.
\end{equation}
\item There exists $\eta\in\PP^N(\Gamma_T)$ of the form 
\begin{equation}\label{eq:discreteeta}
\eta=\frac{1}{N}\sum_{i=1}^N\delta_{\gamma_i},
\end{equation}
such that $(e_t)_\sharp\eta=\mu_t$ for every $t\in[0,T]$ and
for any $i=1,\dots,N$, $\gamma_i\in  \AC([0,T];\R^d)$
solves the differential equation
\begin{equation}\label{eq:Cauchyend}
\dot\gamma_i(t)=v_t(\gamma_i(t))\quad
\text{for }\cL_T \text{-a.e. }t\in[0,T].
\end{equation}
\end{enumerate}
\end{theorem}

 \begin{proof}
Let us recall that the metric derivative $|\mu'|$  of the absolutely continuous curve $\mu$, given by
\begin{equation*}
|\mu'|(t):=\lim_{s\to t}\frac{W_1(\mu_t,\mu_s)}{|t-s|},\quad\text{for a.e. } t\in[0,T],
\end{equation*}
belongs to $L^1(0,T)$ and satisfies
\begin{equation}
W_1(\mu_{t_1},\mu_{t_2})\le\displaystyle\int_{t_1}^{t_2}|\mu'|(t)\,\d t,\quad\text{ for any }0\le t_1\le t_2\le T.
\label{eq:last}
\end{equation}
in particular 
there exists $\psi:[0,+\infty)\to[0,+\infty)$ increasing, convex and superlinear at $+\infty$ such that 
\begin{equation}\label{eq:Summability-3}
\int_0^T \psi\big(|\mu'|(t)\big) \,\d t<+\infty.
\end{equation}

First of all we prove the existence of $\eta\in\PP(\Gamma_T)$ such that $(e_t)_\sharp\eta=\mu_t$ for all $t\in[0,T]$ and $\eta$ is of the form \eqref{eq:discreteeta} for some $\gamma_i\in\Gamma_T$, $i=1,\dots,N$.

\smallskip

Let $M\in\N$ and consider the diadic discretization of the interval $[0,T]$, with time step $\tau_M=T\, 2^{-M}$. 
Since $\mu_t \in \PP^N(\R^d)$ for every $t \in [0,T]$, there exists $x_i(t) \in \R^d$, $i =1, \ldots, N$, such that 
\[\mu_t=\frac{1}{N}\sum_{i=1}^N \delta_{x_i(t)}, \qquad \forall \, t \in [0,T].
\]
For $n=0,\dots,2^M$ and $i=1,\dots,N$, we set $x^n_{M,i}:=x_i(n\tau_M)$, and $\mu^n_M:=\mu_{n\tau_M}$. 
For $n=1,\dots,2^M$, let $\varrho^{n-1,n}_M\in\Gamma_o(\mu^{n-1}_M,\mu^n_M)$ be an optimal plan for the $1$-Wasserstein distance.
Since $\mu^{n-1}_M$ and $\mu^n_M$ belong to $\PP^N(\R^d)$, then $\varrho^{n-1,n}_M$ is of the form
\begin{equation}\label{eq:optrho}
\varrho_M^{n-1,n}=\frac{1}{N}\sum_{i=1}^N \delta_{x^{n-1}_{M,i}}\otimes\delta_{x^n_{M,\sigma^n_M(i)}},
\end{equation}
for some permutation $\sigma^n_M$ of $\{1,\dots,N\}$.
Let us define $\sigma_M^{0,n}:=\sigma_M^{n}\circ\sigma_M^{n-1}\circ\dots\circ\sigma_M^{1}$ and $\sigma_M^{0,0}(i)=i$ for $i=1,\ldots,N$.\\
For $i=1,\dots,N$ we define the curves $\gamma_{M,i}\in\Gamma_T$ by linear time interpolation as
\[\gamma_{M,i}(t):=\frac{n\tau_M-t}{\tau_M}\, x^{n-1}_{M,\sigma_M^{0,n-1}(i)}+\frac{t-(n-1)\tau_M}{\tau_M}\, x^n_{M,\sigma_M^{0,n}(i)},\quad\text{for } t\in[(n-1)\tau_M,n\tau_M],\]
$n=1,\dots,2^M$.
%

We claim that, for any $k=1,\ldots,N$, the sequence $\{\gamma_{M,k}\}_{M\in\N}$ uniformly converges to a curve 
$\gamma_k \in \mathrm{AC}([0,T];\R^d)$.
Indeed, 
\begin{equation}\label{eq:compsik}
\begin{split}
\int_{0}^T\psi\Big(\frac{1}{N}|\dot\gamma_{M,k}(t)|\Big)\,\d t
&\leq \int_{0}^T\psi \Big(\frac{1}{N}\sum_{i=1}^N|\dot\gamma_{M,i}(t)|\Big)\,\d t \\
&= \sum_{n=1}^{2^M}\tau_M\,\psi\Big(\frac{1}{N}\sum_{i=1}^N\frac{\left|x^n_{M,\sigma^{0,n}_M(i)}-x^{n-1}_{M,\sigma^{0,n-1}_M(i) }\right|}{\tau_M}\Big) \\
&= \sum_{n=1}^{2^M}\tau_M\,\psi\Big(\frac{1}{\tau_M}\int_{\R^d\times\R^d}|x-y|\,\d\varrho_M^{n-1,n}(x,y)\Big) \\
&= \sum_{n=1}^{2^M}\tau_M\,\psi\Big(\frac{1}{\tau_M}W_1(\mu_M^{n-1},\mu_M^n)\Big)\\
&\leq \sum_{n=1}^{2^M}\tau_M\,\psi\Big(\frac{1}{\tau_M}\int_{(n-1)\tau_M}^{n\tau_M}|\mu'|(t)\,\d t\Big) \\
&\leq \sum_{n=1}^{2^M}\int_{(n-1)\tau_M}^{n\tau_M}\psi(|\mu'|(t))\,\d t
=\int_{0}^{T}\psi(|\mu'|(t))\,\d t,
\end{split}
\end{equation}
where we employed the definition of the optimal plan in \eqref{eq:optrho}, 
\eqref{eq:last}, Jensen's inequality, and \eqref{eq:Summability-3}.\\
Since $\gamma_{M,k}(0)=x_k(0)$ for any $M\in\N$, and \eqref{eq:compsik} and \eqref{eq:Summability-3} hold, by Ascoli-Arzel\'a Theorem the sequence  $\{\gamma_{M,k}\}_{M\in\N}$ is compact in $C([0,T];\R^d)$. Furthermore, by \eqref{eq:compsik} and the lower semicontinuity of the functional
\[\gamma\mapsto\int_0^T\psi\left(\frac{1}{N}|\dot\gamma(t)|\right)\d t\]
w.r.t.~weak convergence in $\mathrm{AC}([0,T];\R^d)$, we get $\gamma_k\in \mathrm{AC}([0,T];\R^d)$.
Moreover, if $t=n_0\tau_{M_0}$ for some $M_0\in\N$ and $n_0\in\{0,1,\ldots,2^{M_0}\}$, then 
$\gamma_{M,k}(t)$ is constant for any $M\in\N$, $M>M_0$ and the claim is proved.

\medskip

Defining
\begin{equation} \label{eq:etaempiric}
	\eta_M:=\frac{1}{N}\sum_{i=1}^N\delta_{\gamma_{M,i}}, 
	\qquad \eta:=\frac{1}{N}\sum_{i=1}^N\delta_{\gamma_i},
\end{equation}
from the convergence of $\gamma_{M,i}$ to $\gamma_{i}$ it follows that
$\eta_M$ weakly converges to $\eta$ as $M\to+\infty$.
Moreover, if $t=n_0\tau_{M_0}$ for some $M_0\in\N$ and $n_0\in\{0,1,\ldots,2^{M_0}\}$, then 
 $(e_t)_\sharp\eta_M=\mu_t$ for any $M\in\N$, $M>M_0$.
 Then, by the continuity of $t\mapsto \mu_t$ and of $t\mapsto (e_t)_\sharp\eta$, 
 we conclude that $(e_t)_\sharp\eta=\mu_t$ and $\mu_t=\frac 1N \sum_{i=1}^N \delta_{\gamma_i(t)}$  
 for all $t\in[0,T]$.

\medskip
It remains to define a vector field $v$ such that \eqref{eq:Cauchyend} and
\eqref{eq:CIend} hold, also showing that $v$ is uniquely characterized by \eqref{eq:CIend}.

Since $\gamma_i \in \mathrm{AC}([0,T];\R^d)$ for any $i=1,\ldots,N$,
the Borel set $A:=\{t\in[0,T]: \exists\, k\in\{1,\ldots,N\} \text{ such that } \gamma_k \text{ is not differentiable at }t\}$
is $\cL_T$-negligible. Moreover, the sets
\begin{equation}
N_{i,k}:=\{ t\in[0,T]\setminus A:\gamma_i(t)=\gamma_k(t), \dot \gamma_i(t)\neq \dot\gamma_k(t)\}
\end{equation}
satisfy $\cL_T( N_{i,k}) =0$ for any $i,k\in\{1,\ldots,N\}$.
We define $\tilde N:= \Big(\bigcup_{i,k\in\{1,\ldots,N\}}  N_{i,k}\Big)\bigcup A$ 
noticing that $\cL_T( \tilde N) =0$ so that $\cL_T\otimes \mu_t (\tilde N\times \R^d)=0$,
and 
$S:=\{(t,\gamma_i(t)):t\in [0,T],\ i\in \{1,\ldots,N\}\}=\supp(\mu_t\otimes \cL_T)$.
We can thus define a Borel vector field $v:[0,T]\times \R^d \to\R^d$ by
\begin{equation*}
	v(t,x):=\begin{cases}
	0&\text{if }t\in \tilde N\text{ or } (t,x)\in \big([0,T]\times \R^d\big)\setminus S,\\
	\dot \gamma_i(t)&\text{if }x=\gamma_i(t)\ \text{for $t\in [0,T]\setminus \tilde N$ and some }i\in \{1,\ldots,N\},
	\end{cases}
\end{equation*}
so that $ \dot\gamma_i(t)= v_t(\gamma_i(t))$ for every $t\in [0,T]\setminus \tilde N$.
It is then easy to check that \eqref{eq:Summabilityv} and \eqref{eq:CIend} hold.

Let us eventually check that \eqref{eq:CIend} uniquely characterizes $v(t_0,x_0)$ 
for every $(t_0,x_0)\in S\setminus (\tilde N\times \R^d)$. 
Notice that for every $\varphi\in C_c^\infty(\R^d)$ we have
\begin{equation}\label{eq:CE2}
	\frac {\d}{\d t} \int_{\R^d} \varphi\,\d\mu_t\Big|_{t=t_0}=
	\frac{1}{N}\sum_{i=1}^N\nabla\varphi(\gamma_i(t_0))\cdot\dot\gamma_i(t_0)=
\frac{1}{N}\sum_{i=1}^N\nabla\varphi(\gamma_i(t_0))\cdot v({t_0},\gamma_i(t_0)).
\end{equation}
Setting $K:=\{k\in \{1,\ldots,N\}: \gamma_k(t_0)\neq x_0\}$ and $r_0:=\min \{|\gamma_k(t_0)-x_0|:
k\in K\}>0$, for every $\xi\in \R^d$ we 
can find 
a test function
 $\varphi\in C^\infty_c(\R^d)$ such that $\supp \varphi \subset B_{r_0}(x_0)$
 and $\nabla\varphi(x_0)=\xi$:
\eqref{eq:CE2} then yields
 \begin{equation}\label{eq:CE3}
		\frac {\d}{\d t} \int_{\R^d} \varphi\,\d\mu_t\Big|_{t=t_0}=
\frac{n}{N}\xi\cdot v({t_0},x_0)\quad
\text{where}\quad
n:=N-\# K.
\end{equation}
Since $\xi$ is arbitrary, \eqref{eq:CE3} uniquely characterizes $v(t_0,x_0)$ in terms of $\mu$.
 \end{proof}

 \section{Finite Partitions}\label{app_FP}
 
In this section we provide a proof of Proposition \ref{prop:FAP}.
For sake of clarity,  we divide the statement of Proposition \ref{prop:FAP} in three separate lemmas of independent interest.

Given a standard Borel space $(\Omega, \frB, \P)$, in Lemma \ref{l:part1} we construct a family of algebras $\frB^n$, $n \in \N$, satisfying the finite approximation property of Definition \ref{approx_prop}.
Then we fix $(\Omega, \frB, \P) = ([0,1],\cB, \cL_1)$, where $\cB$ is the Borel $\sigma$-algebra and $\cL_1$ the Lebesgue measure restricted to the interval $[0,1]$.
With this choice of parametrization space, in Lemma \ref{lemma:partL1} we show that the family of algebras $\cB^N$ associated to the uniform partition of $[0,1]$ with elements' size $1/N$ satisfies the finite approximation property.
Finally, we combine the previous results in Lemma \ref{l:part3}, where we consider a general standard Borel space $(\Omega, \frB, \P)$ and $\P$ is without atoms.
This is possible thanks to the following fundamental result on  Borel equivalence of Probability spaces 
(see e.g. \cite[Chapter 15, Theorem 9]{Royden}).

\begin{proposition}\label{prop:BeqL1}
Let $\Omega$ be a Polish space and $\P \in \PP(\Omega)$ without atoms. 
Then there exist a Borel set $\Omega_0\subset\Omega$ such that $\P(\Omega_0)=0$, a Borel set $I_0\subset[0,1]$ such that
$\cL_1(I_0)=0$ and a bijective function $\psi:\Omega\setminus\Omega_0\to [0,1]\setminus I_0$ such that
$\psi$ and $\psi^{-1}$ are Borel,  $\psi_{\sharp} \P =\cL_1$ and $(\psi^{-1})_{\sharp} \cL_1 = \P$.
\end{proposition}

The first part of Proposition \ref{prop:FAP} is restated in the following Lemma.

 \begin{lemma}[Proposition \ref{prop:FAP}, part 1]\label{l:part1}
Let $(\Omega,\frB,\P)$ be a standard Borel space. 
Then there exists a family of  finite algebras
$\frB^n \subset \frB$, $n \in \N$, satisfying the finite approximation property of Definition \ref{approx_prop}.
\end{lemma}

\begin{proof}
Since $(\Omega,\frB,\P)$ is standard Borel we can choose a Polish topology $\tau$ such that $\frB = \cB_{(\Omega,\tau)}$,  then there exists a countable basis $\cA=\{B^i:i\in\N\}$ of its topology. Then
$\cB_{(\Omega,\tau)}=\sigma(\{B^i:i\in\N\})$.
We define $\frB^1:=\sigma(B^1)$ and $\frB^n:=\sigma(\{B^n\}\cup\frB^{n-1})$.
It follows from the definition that $\frB^n\subset \frB^{n+1}$ for any $n\in\N$ and
$\cB_{(\Omega,\tau)}=\sigma\left(\bigcup_{n=1}^{+\infty}\frB^n\right)$.

For any $n\in\N$, the finite algebra $\frB^n$ induces a 
minimal (with respect to the inclusion) partition of $\Omega$, denoted by 
$\cP^n=\{A_k^n:k=1,\ldots,k(n)\}\subset\frB^n$. 
Then for any $A_k^{n+1}\in\cP^{n+1}$ there exists 
$h\in\N$ such that $A_h^{n}\in\cP^{n}$ and $A_k^{n+1}\subset A_h^{n}$.

We define the sequence of linear operators $P_n:L^1(\Omega;E)\to L^1(\Omega;E)$ defined by
\[P_ng:=\sum_{k=1}^{k(n)}\mathds{1}_{A^n_k} \, \fint_{A^n_k}g(\omega)\,\d\P(\omega),\]
with the convention that $\fint_{A^n_k}g(\omega)\,\d\P(\omega)=0$ if $\P(A^n_k)=0$.

It is simple to prove that
\begin{equation}\label{contrPn}
	\|P_ng\|_{L_\P^1(\Omega;E)}\leq\|g\|_{L_\P^1(\Omega;E)}, \qquad \forall\, g\in {L_\P^1(\Omega;E)}.
\end{equation}

Given a Borel function $g:\Omega\to E$ such that $g \in L_\P^1(\Omega; E)$, we define $g^n:=P_ng$ and we prove that
the properties of Definition \ref{approx_prop} hold.

Property \ref{FAP1} is obvious since $g^n$ is constant on the elements of the partition $\cP^n\subset\frB^n$.

Property \ref{FAP2} follows from the fact that (see for instance \cite[Corollary 8, p. 48]{diestel1977vector})
\[
	\fint_A g(\omega)\,\d\P(\omega) \in \mathrm{\overline{co}}\left(g(A)\right), \quad \forall\, A\in\cB_{(\Omega,\tau)} : \P(A)>0.
\]

In order to prove property \ref{FAP3} we start with the particular case $g=a\mathds{1}_{A}$ 
for a given Borel set $A$ and a given $a\in E$.
Since $(\Omega,\tau)$ is Polish, for any $\eps>0$ there exists an open set $A_\eps$ and a compact set $K_\eps$ such that
$K_\eps\subset A \subset A_\eps$ and $\P(A_\eps\setminus K_\eps)<\eps$.
Since $A_\eps$ is union of elements of the basis $\cA$, there exists a finite covering of $K_\eps$ of the form 
$\{B^j:j\in J\}\subset\cA$, for a suitable finite $J\subset\N$, such that
$\cup_{j\in J}B^j\subset A_\eps$.
Since 
$$P_ng-g= a\sum_{k=1}^{k(n)} \frac{\P(A_k^n\cap A)}{\P(A_k^n)}\mathds{1}_{A_k^n}-a\mathds{1}_{A},$$
by setting $n_\eps:=\max J$, it holds that
$$\|P_ng-g\|_{L_\P^1(\Omega; E)}\leq \|a\mathds{1}_{A_\eps}-a\mathds{1}_{A}\|_{L_\P^1(\Omega; E)} \le \| a\|_E \ \P(A_\eps \setminus A) < \|a\|_E\,\eps \ , \quad \forall  \ n\geq n_\eps \ .$$ 
Since $P_n$ is linear, then \ref{FAP3} holds for any $g$ simple function.
In the general case, take $g \in L_\P^1(\Omega; E)$ and $\eps>0$, and let $g_\eps:\Omega\to E$ be a simple function such that 
$\|g-g_\eps\|_{L_\P^1(\Omega; E)}<\eps$. 
Observing that
\[
\|P_ng-g\|_{L_\P^1(\Omega; E)}\leq \|P_ng-P_ng_\eps\|_{L_\P^1(\Omega;E)} +\|P_ng_\eps-g_\eps\|_{L_\P^1(\Omega; E)}  
+\|g_\eps-g\|_{L_\P^1(\Omega; E)},
\]
by \eqref{contrPn} and  property \ref{FAP3} applied to $g_\eps$ it holds that
$\limsup_{n\to+\infty} \|P_ng-g\|_{L_\P^1(\Omega; E)}\leq 2\eps$ and we conclude.

Finally, property \ref{FAP4} follows from the measurability of $G$, Fubini Theorem and the definition of $P_n$.
\end{proof}

Consider now the Polish space $([0,1],\cB, \cL_1)$.

\begin{lemma}[Proposition \ref{prop:FAP}, part 2]\label{lemma:partL1}
For any $N\in\N$ we define $I_k^N:=[(k-1)/N,k/N)$, $k=1,\ldots,N-1$,  $I_N^N:=[(N-1)/N,1]$.
If $E$ is a Banach space, $g\in L^1([0,1];E)$ and 
 \[g^N:=\sum_{k=1}^{N} \mathds{1}_{I^N_k}\, \fint_{I^N_k}g(s)\,\d s ,\]
then 
\begin{equation}\label{limitgN}
	\lim_{N\to+\infty}\|g^N-g\|_{L^1([0,1];E)}=0.
\end{equation}
Moreover, the family of finite algebras $\cB^N:=\sigma(\{I^N_k:k=1,\ldots,N\})$, $N\in\N$, satisfies the finite
approximation property of Definition \ref{approx_prop}.\\
Finally, if $g\in L^p([0,1];E)$, for some $p\in(1,+\infty)$, then $\lim_{N\to+\infty}\|g^N-g\|_{L^p([0,1];E)}=0$.
\end{lemma}
\begin{proof}
For any $x\in[0,1]$ and $N\in\N$, there exists a unique $k(x,N)$ such that $x\in I_{k(x,N)}^N$.
From the definition of $I^N_k$ it follows that  $I_{k(x,N)}^N\subset B_{1/N}(x)$.
Since
 \[g^N(x)-g(x)=\sum_{k=1}^{N}\mathds{1}_{I^N_k}(x) \,\fint_{I^N_k}(g(s)-g(x))\,\d s,\]
then 
\[\|g^N(x)-g(x)\|_E\leq 2 \fint_{B_{1/N}(x)}\|g(s)-g(x)\|_E\,\d s.\]
By the Bochner version of the Lebesgue differentiation Theorem (see for instance  \cite[Theorem 9, p. 49]{diestel1977vector})
we obtain that
$\lim_{N\to+\infty}\|g^N(x)-g(x)\|_{E}=0$ for $\cL_1$-a.e. $x\in[0,1]$.

Since $\int_0^1\|g(x)\|_E\,\d x <+\infty$ there exists a convex, increasing, superlinear function $\psi:[0,+\infty)\to[0,+\infty)$
such that  $\int_0^1\psi\left(\|g(x)\|_E\right)\,\d x <+\infty$.
Since
\[\|g^N(x)\|_E\leq 2 \fint_{B_{1/N}(x)}\|g(s)\|_E\,\d s,\]
 by Jensen's inequality,
\[
\begin{split}
 \int_0^1\psi\left(\|g^N(x)\|_E\right)\,\d x & \leq 2 \int_0^1\psi\left( \fint_{B_{1/N}(x)} \|g(s)\|_E\,\d s\right)\,\d x \\
&\leq 2 \int_0^1 \fint_{B_{1/N}(x)}\psi\left(\|g(s)\|_E\right)\,\d s\,\d x \\
&= N \int_0^1 \int_0^1 \mathds{1}_{(x-1/N,x+1/N)}(s) \psi\left(\|g(s)\|_E\right)\,\d s\,\d x\\
&= N \int_0^1 \int_0^1 \mathds{1}_{(s-1/N,s+1/N)}(x) \psi\left(\|g(s)\|_E\right)\,\d s\,\d x\\
&\leq 2 \int_0^1\psi\left(\|g(s)\|_E\right)\,\d s<+\infty,
\end{split}
\]
which implies the equi-integrability of the sequence $\|g^N\|_E$.
 Then \eqref{limitgN} holds.
 
The finite approximation property for $\cB^N$ follows as in the proof of Lemma \ref{l:part1}. The final assertion is a consequence of the equi-integrability of the sequence $\|g^N\|^p_E$.
 \end{proof}

 \begin{lemma}[Proposition \ref{prop:FAP}, part 3]\label{l:part3}
Let $(\Omega,\frB,\P)$ be a standard Borel space and $\P$ without atoms. Then there exists a family $\frB^N \subset \frB$, $N \in \N$, satisfying the finite approximation property of Definition \ref{approx_prop} such that the associated minimal partition $\cP^N= \{A_k^N:k=1,\ldots,N\} $ contains exactly $N$ elements and $\P(A^N_k)=\frac1N$, for $k=1,\ldots,N$.
\end{lemma}
\begin{proof}
Let $\tau$ be a Polish topology on $\Omega$ such that $\frB = \cB_{(\Omega,\tau)}$.
Let also $\Omega_0$, $I_0$, $\psi$, $\psi^{-1}$ be given by Proposition \ref{prop:BeqL1}.
Using the notation of Lemma \ref{lemma:partL1} we define the sets
$A^N_1:=\psi^{-1}(I^N_1\setminus I_0)\cup \Omega_0$ and
$A^N_k:=\psi^{-1}(I^N_k\setminus I_0)$ for $k=2,\ldots,N$.

It is immediate to prove that  $\P(A^N_j)=\frac1N$ for $j=1,\ldots,N$ and $\{A^N_1,\ldots,A^N_N\}$ is a partition of $\Omega$.
Moreover, given a Banach space $E$ and $g\in L^1_\P(\Omega;E)$, we denote by
$\tilde g:=g\circ\psi^{-1} \in L^1([0,1];E)$.
Denoting by $\tilde g^N$ the sequence given by Lemma \ref{lemma:partL1} applied to $\tilde g$, 
we define $g^N:=\tilde g^N\circ\psi$ and the finite approximation property for $(\Omega,\cB_{(\Omega,\tau)},\P)$ follows by
Lemma \ref{lemma:partL1}.

\end{proof}

{
\section*{Acknowledgements}
The authors acknowledge the support of MIUR-PRIN 2017 project 
\emph{Gradient flows, Optimal Transport and Metric Measure Structures}. 
G.~Cavagnari, S.~Lisini and C.~Orrieri acknowledge the support of the 
INDAM-GNAMPA project 2019 ``Trasporto ottimo per dinamiche con interazione''.
C.~Orrieri has also been supported by the project Fondazione Cariplo-Regione Lombardia MEGAsTAR ``Matematica d'Eccellenza in biologia ed ingegneria come acceleratore di una nuova strateGia per l'ATtRattivit\`a dell'ateneo pavese''.}
G.~Savar\'e gratefully acknowledges the support of the Institute of Advanced Study or the Technical University of Munich and of IMATI-CNR, Pavia.
G.~Cavagnari and G.~Savar\'e are also grateful to the Department of Mathematics of the University of Pavia where this project has been developed.

\bibliography{biblio}
\bibliographystyle{plain}

\end{document}